\documentclass[11pt]{amsart}
\usepackage{cgeometry}
\usepackage[final]{microtype}
\newcommand\blfootnote[1]{%
  \begingroup
  \renewcommand\thefootnote{}\footnote{#1}%
  \addtocounter{footnote}{-1}%
  \endgroup
}

\usepackage{xurl}
\hypersetup{breaklinks=true}

\bibliography{ComplexGeometry}

\title[Partial Okounkov bodies]{Partial Okounkov bodies and Duistermaat--Heckman measures of non-Archimedean metrics}
\author{Mingchen Xia}

\begin{document}

\begin{abstract}
	Let $X$ be a smooth complex projective variety.
	We construct partial Okounkov bodies associated with Hermitian big line bundles $(L,\phi)$ on $X$. We show that partial Okounkov bodies are universal invariants of the singularities of $\phi$.
	As an application, we construct Duistermaat--Heckman measures associated with finite energy metrics on the Berkovich analytification of an ample line bundle.
\end{abstract}



\blfootnote{
	Mingchen Xia, \textsc{Institut de Mathématiques de Jussieu-Paris Rive Gauche}\par\nopagebreak
	\textit{Email address}, \texttt{mingchen@imj-prg.fr}\par\nopagebreak
	\textit{Homepage}, \url{http://mingchenxia.github.io/home/}.
}
 \normalsize 

\maketitle

\tableofcontents

\section{Introduction}
\subsection{Background}
Let $X$ be an irreducible smooth projective variety of dimension $n$ and $L$ be a big holomorphic line bundle on $X$. Given any admissible flag $X=Y_0\supseteq Y_1\supseteq\cdots\supseteq Y_n$
on $X$ (see \cref{def:admfl} for the precise definition), one can attach a natural convex body $\Delta(L)$ of dimension $n$ to $L$, generalizing the classical Newton polytope construction in toric geometry. This construction was first considered by Okounkov \cite{Oko96, Oko03} and then extended by Lazarsfeld--Mustaț{\u{a}} \cite{LM09} and Kaveh--Khovanskii \cite{KK12}. The convex body $\Delta(L)$ is known as the \emph{Okounkov body} or \emph{Newton--Okounkov body} associated with $L$ (with respect to the given flag). We briefly recall its definition: given any non-zero $s\in \mathrm{H}^0(X,L^k)$, let $\nu_1(s)$ be the vanishing order of $s$ along $Y_1$. Then $s$ can be regarded as a section of $\mathrm{H}^0(X,L^k\otimes \mathcal{O}_X(-\nu_1(s)Y_1))$. It follows that $s_1\coloneqq s|_{Y_1}$ is a non-zero section of $L|_{Y_1}^k\otimes \mathcal{O}_X(-\nu_1(s)Y_1)|_{Y_1}$. We can then repeat the same procedure with $s_1$, $Y_2$ in place of $s$, $Y_1$. Repeating this construction, we end up with $\nu(s)=(\nu_1(s),\ldots,\nu_n(s))\in \mathbb{N}^n$. In fact, $\nu$ extends naturally to a rank $n$ valuation on $\mathbb{C}(X)$.
Consider the semigroup
\[
	\Gamma(L)\coloneqq \left\{(\nu(s),k)\in \mathbb{Z}^{n+1}:k\in \mathbb{N}, s\in \mathrm{H}^0(X,L^k)^{\times} \right\}.
\]
Then $\Delta(L)$ is the intersection of the closed convex cone in $\mathbb{R}^{n+1}$ generated by $\Gamma(L)$ with the hyperplane $\{(x,1):x\in \mathbb{R}^n\}$. A key property of $\Delta(L)$ is that the Lebesgue volume of $\Delta(L)$ is proportional to the volume of the line bundle $L$:
\begin{equation}\label{eq:volD}
	\vol \Delta(L)=\frac{1}{n!}\langle L^n\rangle.
\end{equation}
Here $\langle \bullet\rangle$ denotes the movable intersection product in the sense of \cite{BDPP13, BFJ09}.

In \cite{LM09}, Lazarsfeld--Mustaț{\u{a}} showed moreover that $\Delta(L)$ depends only on the numerical class of $L$. Conversely, it is shown by Jow \cite{Jow10} that the information of all Okounkov bodies with respect to various flags actually determines the numerical class of $L$. In other words, Okounkov bodies can be regarded as universal numerical invariants of big line bundles.

This paper concerns a similar problem. Assume that $L$ is equipped with a singular plurisubharmonic (psh) metric $\phi$.
We will construct universal invariants of the singularity type of $\phi$.
We call these universal invariants the \emph{partial Okounkov bodies} of $(L,\phi)$.

\subsection{Main results}
Let us explain more details about the construction of partial Okounkov bodies.
Recall that any admissible flag on $X$ induces a rank $n$ valuation on $\mathbb{C}(X)$ with values in $\mathbb{Z}^n$. We will work more generally with such valuations, not necessarily coming from admissible flags on $X$.
We define a set
\begin{equation}\label{eq:Gamma1}
	\Gamma(L,\phi)\coloneqq \left\{(\nu(s),k)\in \mathbb{Z}^{n+1}:k\in \mathbb{N}, s\in \mathrm{H}^0(X,L^k\otimes \mathcal{I}(k\phi))^{\times} \right\}
\end{equation}
similar to $\Gamma(L)$. Here $\mathcal{I}(\bullet)$ denotes the multiplier ideal sheaf in the sense of Nadel.
However, a key difference here is that $\Gamma(L,\phi)$ is not a semigroup in general. Thus, the constructions in both \cite{LM09} and \cite{KK12} break down.
We will show that in this case, there is still a canonical construction of Okounkov bodies.

Before stating our main theorem, let us recall the definition of volume. The volume of $(L,\phi)$ is defined as
\[
	\vol (L,\phi)\coloneqq \lim_{k\to\infty}\frac{1}{k^n}h^0(X,L^k\otimes \mathcal{I}(k\phi)).
\]
The existence of this limit is proved in \cite{DX21}.

\begin{maintheorem}\label{thm:partOkobody}
	Let $(L,\phi)$ be as above. Assume that $\vol(L,\phi)>0$. Then there is a convex body $\Delta(L,\phi)\subseteq \Delta(L)$ associated with $(L,\phi)$ satisfying
	\begin{equation}\label{eq:volD1}
		\vol \Delta(L,\phi)=\vol(L,\phi).
	\end{equation}

	Moreover, $\Delta(L,\phi)$ is continuous in $\phi$ if $\int_X (\ddc\phi)^n>0$. Here the set of $\phi$ is endowed with the $d_S$-pseudometric in the sense of \cite{DDNLmetric} and the set of convex bodies is endowed with the Hausdorff metric.

	Define
	\[
		\Gamma_k\coloneqq \left\{k^{-1}\nu(s)\in \mathbb{R}^n:s\in \mathrm{H}^0(X,L^k\otimes \mathcal{I}(k\phi))^{\times}\right\}
	\]
	and let $\Delta_k$ denote the convex hull of $\Gamma_k$. Then 
    \begin{equation}\label{eq:Hausconp}
    \Delta_k \to \Delta(L,\phi)
    \end{equation}
    with respect to the Hausdorff metric if $\vol(L,\phi)>0$.
\end{maintheorem}
Observe that the last assertion actually uniquely determines $\Delta(L,\phi)$, so $\Delta(L,\phi)$ can be regarded as canonically attached to the given data $(X,L,\phi,\nu)$.

The convex body $\Delta(L,\phi)$ is called the \emph{partial Okounkov body} of $(L,\phi)$ with respect to the given valuation.
 Here the word \emph{partial} refers to the fact that the partial Okounkov bodies are contained in $\Delta(L)$. One should not confuse them with the notion of Okounkov bodies with respect to partial flags.

We will also extend the definition to the case $\vol(L,\phi)=0$ in \cref{subsec:limpob}, at the expense of losing continuity in $\phi$.

Observe that \eqref{eq:volD1} bears strong resemblance with \eqref{eq:volD}. In fact, when $\phi$ has minimal singularities, $\Delta(L,\phi)=\Delta(L)$ and
\eqref{eq:volD1} just reduces to \eqref{eq:volD}.

The second main result says that partial Okounkov bodies uniquely determine the $\mathcal{I}$-singularity type of $\phi$.
\begin{maintheorem}\label{thm:IeqDelta}
	Let $L$ be a big line bundle on $X$. Let $\phi$, $\phi'$ be two singular psh metrics on $L$ with positive volumes. Then the following are equivalent:
	\begin{enumerate}
		\item $\phi\sim_{\mathcal{I}}\phi'$.
		\item $\Delta(L,\phi)=\Delta(L,\phi')$ for all rank $n$ valuations on $\mathbb{C}(X)$ taking values in $\mathbb{Z}^n$.
	\end{enumerate}
\end{maintheorem}
Recall that $\phi\sim_{\mathcal{I}}\phi'$ means $\mathcal{I}(k\phi)=\mathcal{I}(k\phi')$ for all real $k>0$. This relation is studied in detail in \cite{DX22, DX21}. It captures a lot of important information about the singularity of a psh metric.
\cref{thm:IeqDelta} should be regarded as a metric analogue of Jow's theorem.

As a byproduct of our proof of \cref{thm:IeqDelta}, we reprove a formula computing the generic Lelong numbers of currents of minimal singularities in $c_1(L)$, slightly generalizing \cite[Theorem~5.4]{Bou02h}:
\begin{theorem}[{=\cref{cor:numin}}]
	Let $L$ be a big line bundle on $X$.
	Consider a current $T_{\min}$ of minimal singularity in $c_1(L)$. Then for any prime divisor $E$ over $X$, we have
	\begin{equation}
		\nu(T_{\min},E)=\lim_{k\to\infty}\frac{1}{k}\ord_E \mathrm{H}^0(X,L^k).
	\end{equation}
    Here $\nu(T_{\min},E)$ denotes the generic Lelong number of $T_{\min}$ along $E$.
\end{theorem}
As a consequence, we find a new formula computing the multiplier ideal sheaf $\mathcal{I}(T_{\min})$ in \cref{cor:IVtheta}.

The third main result is an analogue of \cite{WN14}. Given any continuous metric $\psi$ on $L$, one can naturally construct a convex function $c[\psi]$ on $\Int \Delta(L)$, known as the \emph{Chebyshev transform} of $\psi$. The main property of $c[\psi]$ is that given another continuous metric $\psi'$ on $L$, we have
\begin{equation}
	\int_{\Delta(L)} \left(c[\psi]-c[\psi']\right)\,\mathrm{d}\lambda=\vol(\psi,\psi'),
\end{equation}
where $\vol(\psi,\psi')$ is the relative volume as studied in \cite{BB10, BBWN11}  and $\mathrm{d}\lambda$ is the Lebesgue measure on $\mathbb{R}^n$. In our setup, we also associate a convex function: $c_{[\phi]}[\psi]:\Int \Delta(L,\phi)\rightarrow \mathbb{R}$. Moreover,
\begin{maintheorem}\label{thm:intdiffc}
	Assume that the valuation $\nu$ is induced by an admissible flag on $X$.
	Let $\psi,\psi'$ be two continuous metrics on $L$, then
	\begin{equation}
		\int_{\Delta(L,\phi)} \left(c_{[\phi]}[\psi]-c_{[\phi]}[\psi']\right)\,\mathrm{d}\lambda=-\mathcal{E}^{\theta}_{[\phi]}(\psi)+\mathcal{E}^{\theta}_{[\phi]}(\psi'),
	\end{equation}
	where $\mathcal{E}^{\theta}_{[\phi]}$ is the partial equilibrium energy functional defined in \eqref{eq:Eequil}.
\end{maintheorem}
\cref{thm:partOkobody}, \cref{thm:IeqDelta} and \cref{thm:intdiffc} together give convex-geometric interpretations of the main results of \cite{DX22, DX21}. These results also provide us with a convex-geometric approach to the study of psh singularities.

As an application of our theory, we prove a generalization of  Boucksom--Chen theorem (\cref{thm:BCgen}).
Recall the theorem Boucksom--Chen \cite{BC11} says that given a \emph{multiplicative} filtration $\mathscr{F}$ on the section ring $R(X,L)$, one can naturally associate a probability measure on $\mathbb{R}$, known as the \emph{Duistermaat--Heckman measure}.
Moreover, the Duistermaat--Heckman measure is the weak limit of a sequence of discrete measures $\mu_k$ associated with the filtration $\mathscr{F}$ on $\mathrm{H}^0(X,L^k)$.
We show that this construction can be generalized to all $\mathcal{I}$-model test curves, not necessarily coming from filtrations.
Here we only prove the generalized Boucksom--Chen theorem for filtrations on the full graded linear series, which suffices for our purpose. It is, however, easy to see that the techniques apply to more general situations.

More generally, we introduce the notion of an Okounkov test curve (\cref{def:Otc}) and generalize Duistermaat--Heckman measures to this setting.

When $L$ is ample, this construction allows us to associate a Radon measure $\DHm(\eta)$ on $\mathbb{R}$ with each element $\eta$ in the non-Archimedean space $\mathcal{E}^1(L^{\An})$ in the sense of \cite{BJglobal}, see \cref{def:DHmNA}. The space $\mathcal{E}^1(L^{\An})$ can be seen as the completion of the space of test configurations.

\begin{theorem}\label{thm:DHextintro}
	The Duistermaat--Heckman measure construction of test configurations as in \cite{WN12} admits a unique continuous extension $\DHm:\mathcal{E}^1(L^{\An})\rightarrow \mathcal{M}(\mathbb{R})$. Here $\mathcal{M}(\mathbb{R})$ is the space of Radon measures on $\mathbb{R}$.
\end{theorem}
The Duistermaat--Heckman measure of a non-Archimedean metric is also constructed by Inoue \cite{Ino22} using a different method. See \cref{rmk:Ino} for more details.

In \cref{thm:contDH}, we will furthermore prove that $\DHm(\eta)$ contains a lot of interesting information of $\eta$.

In the last section, we interpret the partial Okounkov bodies in the toric setting. We prove the following results:
\begin{theorem}\label{thm:torint}
	Let $X$ be a smooth toric variety of dimension $n$ and $(L,\phi)$ be a toric invariant Hermitian big line bundle on $X$ with positive volume. Fix a toric invariant admissible flag on $X$. Recall that upon choosing a toric invariant rational section of $L$, $\phi$ can be identified with a convex function $\phi_{\mathbb{R}}$ on $\mathbb{R}^n$. Then the partial Okounkov body $\Delta(L,\phi)$ is naturally identified with the closure of the image of $\nabla \phi_{\mathbb{R}}$.
\end{theorem}

\begin{maintheorem}\label{thm:tormixedmass}
	Let $(L_i,\phi_i)$ ($i=1,\ldots,n$) be toric invariant Hermitian big line bundles on $X$ of positive volumes. 
	If the toric invariant flag $(Y_{\bullet})$ satisfies the additional condition: $Y_n$ is not contained in the polar locus of any $\phi_i$, then 
 \[
 \int_X \ddc\phi_1\wedge\cdots\wedge\ddc \phi_n=n!\vol(\Delta(L_1,\phi_1),\ldots,\Delta(L_n,\phi_n)).
 \]
\end{maintheorem}

It is of interest to generalize \cref{thm:tormixedmass} to the non-toric setting as well. As shown by \cref{ex:mixvol}, the non-toric generalization has to involve all valuations instead of just one.

Lastly, let us mention that our generalization of Boucksom--Chen theorem has important consequences in Archimedean pluripotential theory as well. When applied to \emph{generalized deformation to the normal cone} in the sense of \cite{Xia20},
it gives a number of interesting equidistribution results of the jumping numbers of multiplier ideal sheaves. As a detailed investigation would lead us too far away, we do not include these results in this paper.

\subsection{Strategy of the proofs}
We will sketch the proof of these theorems.

\textbf{The proof of} \cref{thm:partOkobody}.
In general, the graded linear space
\[
	W(L,\phi)\coloneqq \bigoplus_{k=0}^{\infty}\mathrm{H}^0(X,L^k\otimes \mathcal{I}(k\phi))
\]
is not an algebra and similarly $\Gamma(L,\phi)$ as defined in \eqref{eq:Gamma1} is not a semigroup.
Thus, one can not directly apply the theory of graded linear series or the theory of semigroups
as in \cite{LM09} and \cite{KK12}. 

A key observation here is that although $\Gamma(L,\phi)$ is not a semigroup, it is not too far away from being one. 

To make this precise, we introduce a pseudometric $d$ on the space $\hat{\mathcal{S}}$ of subsets of $\mathbb{Z}^{n+1}$ lying in a suitable strictly convex cone:
\[
	d(S,S')\coloneqq \varlimsup_{k\to\infty}k^{-n}\left(|S_k|+|S_k'|-2|S_k\cap S_k'|\right).
\]
Let $\sim$ be the equivalence relation defined by $d$. 
The classical Okounkov body construction associates with each semigroup a convex body.
As we will prove later, this map factorizes through the $\sim$-equivalence classes, and it extends continuously to an \emph{almost semigroup}, namely an object in $\hat{\mathcal{S}}$ which can be approximated by certain \emph{nice} semigroups with respect to $d$. 

In order to define the Okounkov body of $(L,\phi)$, we will actually show that $\Gamma(L,\phi)$ is an almost semigroup and we could simply define
\[
	\Delta(L,\phi)\coloneqq \Delta(\Gamma(L,\phi)).
\]
The proof follows the same pattern as the proof in \cite{DX21}. We proceed by approximations. We first consider the case where $\phi$ has analytic singularities. In this case, after taking a suitable resolution, we can easily see that $W(L,\phi)$ can be approximated by graded linear series both from above and from below. 
In the case of a singular $\phi$ with $\ddc\phi$ being a Kähler current, we make use of analytic approximations as in \cite{DPS01, Cao14}. More precisely, take a quasi-equisingular approximation $\phi^j$ of $\phi$.
Based on the convergence theorems proved in \cite{DX21}, we can show that $\Gamma(L,\phi^j)$ converges to $\Gamma(L,\phi)$ with respect to the pseudometric $d$, which enables us to conclude in this case. Finally, in the general case, a trick discovered in \cite{DDNLmetric} and \cite{DX21} enables us to reduce to the previous case.

Along the lines of the proof, we actually find that $\Gamma(L,\phi)$ satisfies a stronger property \eqref{eq:Hausconp}. This property is essential to the proof of \cref{thm:IeqDelta}, we call it the \emph{Hausdorff convergence property}. 

\textbf{The proof of} \cref{thm:IeqDelta}.
Recall that in the classical setting, we can read information about the asymptotic base loci of $L$ from the Okounkov body $\Delta(L)$ directly, see \cite{CHPW18}. In our setup, the analogue says that the Okounkov body $\Delta(L,\phi)$ gives information about the generic Lelong numbers of $\phi$.
We will prove a qualitative version of \cref{thm:IeqDelta}:
\begin{theorem}\label{thm:nuvarphiE}
	Let $E$ be a prime divisor over $X$. Let $\pi:Z\rightarrow X$ be a birational model of $X$ such that $E$ is a divisor on $Z$. Take an admissible flag $(Y_{\bullet})$ on $Z$ with $Y_1=E$, then
	\[
		\nu(\phi,E)=\min_{x\in \Delta(\pi^*L,\pi^*\phi)}x_1.
	\]
\end{theorem}
Here $\nu(\phi,E)$ is the generic Lelong number of $\phi$ along $E$, defined as the minimum of the Lelong numbers $\nu(\pi^*\phi,x)$ for all $x\in E$. The proof of \cref{thm:nuvarphiE} again follows the same pattern as in the proof of \cref{thm:partOkobody}. With some efforts, we can reduce the problem to the case where $\phi$ has analytic singularities along some normal crossing $\mathbb{Q}$-divisor on $X$ and $\ddc\phi$ is a Kähler current. In this case, the desired result follows from a result proved in \cite{Xia20}.

\textbf{The proofs of} \cref{thm:intdiffc} and \cref{thm:tormixedmass}. The proofs roughly follow the same pattern as above. Namely, we first handle the case of analytic singularities and then conclude the general case by suitable approximations. We will not repeat the details here.

As explained above, our approach to general psh singularities requires a number of approximations, this motivates the study of the metric geometry of the space of psh singularity types. We prove the continuity of mixed masses under $d_S$-approximations:
\begin{theorem}[{= \cref{thm:dsmixedmass}}]\label{thm:mixc}
	Let $\theta_i$ ($i=1,\ldots,n$) be smooth closed real $(1,1)$-forms representing big classes on a connected compact K\"ahler manifold $X$ of dimension $n$.
	Let $\varphi_i^k,\varphi_i\in \PSH(X,\theta_i)$ ($i=1,\ldots,n$, $k\in \mathbb{N}$). Assume that $\varphi_i^k\xrightarrow{d_{S,\theta_i}} \varphi_i$ for all $i$ as $k\to\infty$. Then
	\begin{equation}
		\lim_{k\to\infty}\int_X \theta_{1,\varphi_1^k}\wedge\cdots\wedge\theta_{n,\varphi_n^k}=\int_X \theta_{1,\varphi_1}\wedge\cdots\wedge\theta_{n,\varphi_n}.
	\end{equation}
	Here the Monge--Ampère operators are taken in the non-pluripolar sense.
\end{theorem}
This theorem and its various consequences are indispensable in all of our proofs. They are of independent interests as well. 

\subsection{Structure of the paper}

In \cref{sec:pre}, we collect a few preliminaries.

In \cref{sec:clo}, we study the Okounkov bodies of almost semigroups.

In \cref{sec:singtype}, we further develop the theory of $d_S$-pseudometrics on the space of singularity types initiated in \cite{DDNLmetric}.

In \cref{sec:PoB}, we define partial Okounkov bodies associated with Hermitian pseudo-effective line bundles and prove a number of properties.

In \cref{sec:Che}, we define and study Chebyshev transforms of continuous metrics.

In \cref{sec:BC}, we generalize the theory of Boucksom--Chen and study the non-Archimedean Duistermaat--Heckman measures.

In \cref{sec:tor}, we give an explicit description of partial Okounkov bodies construction in terms of the moment polytope in the toric situation.

\subsection{Conventions}
In this paper, Monge--Ampère operators $\theta_{\varphi}^n$ refer to the non-pluripolar product in the sense of \cite{BEGZ10}.
The group $\mathbb{Z}^n$ is always endowed with the lexicographic order. A line bundle always refers to a holomorphic line bundle. We do not distinguish a line bundle and the associated invertible sheaf.
When talking about a birational modification (resolution) $\pi:Y\rightarrow X$, we always assume that $Y$ is smooth and $\pi$ is projective. We follow the convention that $\ddc=\frac{\mathrm{i}}{2\pi}\partial\overline{\partial}$. 

\subsection{Acknowledgements}
I benefited a lot from discussions with David Witt Nyström, Chen Jiang, Sébastien Boucksom,Yi Yao, Jian Xiao, Tamás Darvas, Kewei Zhang, Longke Tang and Rémi Reboulet. I would like to thank especially Yi Yao for explaining his computations in the toric setting to me, Chen Jiang for providing \cref{ex:mixvol} and Yaxiong Liu for pointing out a number of typos in the arXiv version.

I also want to express my gratitude to the referees for numerous valuable comments helping to improve this paper in many aspects.

The author is supported by Knut och Alice Wallenbergs Stiftelse grant KAW 2021.0231.

\section{Preliminaries}\label{sec:pre}

\subsection{Hausdorff metric of convex bodies}\label{subsec:Hausmetric}
In this section, we recall the theory of Hausdorff metrics on the set of convex bodies following \cite[Section~1.8]{Sch14}.
Fix $n\in \mathbb{N}$. Recall that a convex body in $\mathbb{R}^n$ is a non-empty compact convex subset of $\mathbb{R}^n$, which may have empty interior.
Let $\mathcal{K}_n$ denote the set of convex bodies in $\mathbb{R}^n$. We will fix the Lebesgue measure $\mathrm{d}\lambda$ on $\mathbb{R}^n$, normalized so that the unit cube has volume $1$.

Recall the definition of the Hausdorff metric between $K_1,K_2\in \mathcal{K}_n$:
\[
	d_n(K_1,K_2)\coloneqq \max\left\{\adjustlimits\sup_{x_1\in K_1}\inf_{x_2\in K_2}|x_1-x_2|, \adjustlimits\sup_{x_2\in K_2}\inf_{x_1\in K_1}|x_1-x_2|\right\}.
\]
We extend $d_n$ to an extended metric on $\mathcal{K}_n\cup\{\emptyset\}$ by setting
\[
	d_n(K,\emptyset)=\infty
\]
for all $K\in \mathcal{K}_n$.

\begin{theorem}
	The metric space $(\mathcal{K}_n,d_n)$ is complete.
\end{theorem}
\begin{theorem}[Blaschke selection theorem]\label{thm:Bst}
	Every bounded sequence in $\mathcal{K}_n$ has a convergent subsequence.
\end{theorem}
\begin{theorem}\label{thm:contvol}
	The Lebesgue volume $\vol:\mathcal{K}_n\rightarrow \mathbb{R}_{\geq 0}$ is continuous.
\end{theorem}
\begin{theorem}\label{thm:Hausconvcond}
	Let $K_i, K\in \mathcal{K}_n$ ($i\in \mathbb{N}$). Then $K_i\xrightarrow{d_n}K$ if and only if the following conditions hold
	\begin{enumerate}
		\item Each point $x\in K$ is the limit of a sequence $x_i\in K_i$.
		\item The limit of any convergent sequence $(x_{i_j})_{j\in \mathbb{N}}$ with $x_{i_j}\in K_{i_j}$ lies in $K$, where $i_j$ is a subsequence of $1,2,\ldots$.
	\end{enumerate}
\end{theorem}
The proofs of all these results can be found in \cite[Section~1.8]{Sch14}.

\begin{lemma}\label{lma:volcbimpeq}
	Let $K_0,K_1\in \mathcal{K}_n$. Assume that $K_0\subseteq K_1$ and
	\[
		\vol K_0=\vol K_1>0.
	\]
	Then $K_0=K_1$.
\end{lemma}
\begin{proof}
	In fact, if $K_1\neq K_0$, then $K_1\setminus K_0$ is a non-empty open subset of $K_1$. As $\vol K_1>0$,  $(K_1\setminus K_0)\cap \Int K_1\neq \emptyset$. Thus, $\vol K_1>\vol K_0$, which is a contradiction.

\end{proof}

Let $K\in \mathcal{K}_n$ be a convex body with positive volume. For $\delta>0$ small enough, let $K^{\delta}\coloneqq \{x\in K:d(x,\partial K)\geq \delta\}$. Then $K_{\delta}\in \mathcal{K}_n$ for $\delta$ small enough.

\begin{lemma}\label{lma:latcvb}
	Let $K\in \mathcal{K}_n$ be a convex body with positive volume and $K'\in \mathcal{K}_n$. Assume that for some large enough $k\in \mathbb{Z}_{>0}$, $K'$ contains $K\cap (k^{-1}\mathbb{Z})^n$, then $K'\supseteq K^{n^{1/2}k^{-1}}$.
\end{lemma}
\begin{proof}
    Let $x\in K^{n^{1/2}k^{-1}}$, by assumption, the closed ball $B$ with center $x$ and radius $n^{1/2}k^{-1}$ is contained in $K$. Observe that $x$ can be written as a convex combination of points in $B\cap (k^{-1}\mathbb{Z})^n$, which are contained in $K'$ by assumption. It follows that $x\in K'$. 
\end{proof}

Given a sequence of convex bodies $K_i$ ($i\in \mathbb{N}$), we set
\[
	\varliminf_{i\to\infty}K_i= \overline{\bigcup_{i=0}^{\infty}\bigcap_{j\geq i}K_j}.
\]

Suppose $K$ is the limit of a subsequence of $K_i$, we have
\begin{equation}\label{eq:liminflimsup}
	\varliminf_{i\to\infty}K_i\subseteq K.
\end{equation}
This is a simple consequence of \cref{thm:Hausconvcond}.

\subsection{Admissible flags and valuations}
Let $X$ be an irreducible normal projective variety of dimension $n$.
\begin{definition}\label{def:admfl}
	An \emph{admissible flag} $(Y_{\bullet})$ on $X$ is a flag of subvarieties
	\[
		X=Y_0\supseteq Y_1\supseteq \cdots\supseteq Y_n
	\]
	such that $Y_i$ is irreducible of codimension $i$ and smooth at the point $Y_n$.
\end{definition}

Given any admissible flag $(Y_{\bullet})$, we can define a rank $n$ valuation $\nu_{(Y_{\bullet})}:\mathbb{C}(X)^{\times}\rightarrow \mathbb{Z}^n$ as in \cite{LM09}. Here we consider $\mathbb{Z}^n$ as a totally ordered Abelian group with the lexicographic order.
We recall the definition: let $s\in \mathbb{C}(X)^{\times}$. Let $\nu_1(s)=\ord_{Y_1}s$. After localization around $Y_n$, we can take a local defining equation $t^1$ of $Y_1$, set $s_1=(s(t^1)^{-\nu_1(s)})|_{Y_1}$. Then $s_1\in \mathbb{C}(Y_1)$. We can repeat this construction with $Y_2$ in place of $Y_1$ to get $\nu_2(s)$ and $s_2$. Repeating this construction $n$ times, we get $\nu_{(Y_{\bullet})}(s)=\nu(s)=(\nu_1(s),\nu_2(s),\ldots,\nu_n(s))\in \mathbb{Z}^n$.
It is easy to verify that $\nu$ is indeed a rank $n$ valuation.

\begin{remark}\label{rmk:Abhyankar}
Conversely, by a theorem of Abhyankar, any valuation of $\mathbb{C}(X)$ with Noetherian valuation ring of rank $n$ is equivalent to a valuation taking value in $\mathbb{Z}^n$, see \cite[Chapter~0, Theorem~6.5.2]{FK18}. As shown in \cite[Theorem~2.9]{CFKLRS17}, any such valuation is equivalent to (but not necessarily equal to) a valuation induced by an admissible flag on a birational modification of $X$. Here two valuations $\nu,\nu'$ with value in $\mathbb{Z}^n$ are equivalent if one can find a matrix $G$ of the form $\mathrm{I}+N$, where $N$ is strictly upper triangular with integral entries, such that $\nu'=G \nu$.
\end{remark}

\subsection{Model potentials and \texorpdfstring{$\mathcal{I}$}{I}-model potentials}
Let $X$ be a connected compact Kähler manifold of dimension $n$ and $\theta$ be a smooth closed real $(1,1)$-form representing a $(1,1)$-cohomology class $[\theta]$.
Define $V_{\theta}\coloneqq \sup\{\varphi\in \PSH(X,\theta):\varphi\leq 0\}$. For any two $\varphi,\psi\in \PSH(X,\theta)$, we say $\varphi$ is \emph{more singular} than $\psi$ and write $[\varphi]\preceq[\psi]$ if there is a constant $C$ such that $\varphi\leq\psi+C$. When $\varphi\preceq 
\psi$ and $\psi\preceq \varphi$, we say that they have the same \emph{singularity type}.
We write $\theta_{\varphi}=\theta+\ddc\varphi$.

\begin{definition}
	Let $\varphi\in \PSH(X,\theta)$. Define
	\begin{equation}\label{eq:Cdef}
		C^{\theta}[\varphi]\coloneqq \sups\left\{\psi\in \PSH(X,\theta):[\varphi]\preceq[\psi],\psi\leq 0,\int_X \theta_{\varphi}^k\wedge \theta_{V_{\theta}}^{n-k}=\int_X \theta_{\psi}^k\wedge \theta_{V_{\theta}}^{n-k}, \forall k  \right\}.
	\end{equation}
	If $C^{\theta}[\varphi]=\varphi$, we say $\varphi$ is a \emph{model potential}. We omit $\theta$ from the notation if there is no risk of confusion.
\end{definition}
Here and in the sequel the Monge--Amp\`ere type operators are taken in the non-pluripolar sense \cite{BEGZ10}.

\begin{proposition}[{\cite[Proposition 2.6]{DDNLmetric}}]
	For any $\varphi\in \PSH(X,\theta)$, $C^{\theta}[\varphi]$ is a model potential in $\PSH(X,\theta)$. When $\int_X \theta^n_\varphi>0$  we have
	\[
		C^{\theta}[\varphi]=P^{\theta}[\varphi],
	\]
	where
	\begin{equation}\label{eq:Pdef}
		P^{\theta}[\varphi]\coloneqq \sups\left\{\psi\in \PSH(X,\theta):[\psi]\preceq[\varphi],\psi\leq 0\right\}.
	\end{equation}
	In general, we only have
	\begin{equation}\label{eq:CandP}
		C^{\theta}[\varphi]=\lim_{\epsilon\to 0+}P^{\theta}[(1-\epsilon)\varphi+\epsilon V_{\theta}].
	\end{equation}
\end{proposition}
We omit $\theta$ from the notation $P^{\theta}[\varphi]$ if there is no risk of confusion.

\begin{definition}\label{def:BM}
	A \emph{birational model} of $X$ is a projective birational morphism $\pi:Y\rightarrow X$ from a \emph{smooth} projective variety $Y$ to $X$.
\end{definition}

Recall that $\mathcal{I}(\varphi)$ denotes the multiplier ideal sheaf of a qpsh function $\varphi$ on $X$ in the sense of Nadel, namely the coherent subsheaf of $\mathcal{O}_X$ consisting of functions $f$ such that $|f|^2 \exp(-\varphi)$ is locally integrable.

\begin{definition}\label{def:Icomp}
	Let $\varphi,\psi$ be two quasi-psh functions, we say $\varphi\preceq_{\mathcal{I}}\psi$ if the following equivalent conditions are satisfied:
	\begin{enumerate}
		\item $\mathcal{I}(k\varphi)\subseteq \mathcal{I}(k\psi)$ for all real $k>0$.
		\item $\mathcal{I}(k\varphi)\subseteq \mathcal{I}(k\psi)$ for all integer $k>0$.
		\item For any birational model $\pi:Y\rightarrow X$ and any $y\in Y$, we have $\nu(\pi^*\varphi,y)\geq \nu(\pi^*\psi,y)$.
	\end{enumerate}
    The equivalence between (1) and (3) is just \cite[Corollary~2.16]{DX22}. The equivalence between (2) and (3) follows from \cite[Proposition~2.14]{DX22}.
 
	We say $\varphi\sim_{\mathcal{I}}\psi$ if $\varphi\preceq_{\mathcal{I}}\psi$ and $\psi\preceq_{\mathcal{I}}\varphi$.

	Given any $\varphi\in \PSH(X,\theta)$, we define
	\[
		P^{\theta}[\varphi]_{\mathcal{I}}\coloneqq \sup\left\{\psi\in \PSH(X,\theta):\psi\preceq_{\mathcal{I}}\varphi,\psi\leq 0\right\}.
	\]
 We omit $\theta$ when there is no risk of confusion.
    We say $\varphi$ is \emph{$\mathcal{I}$-model} if $\varphi=P[\varphi]_{\mathcal{I}}$. 
\end{definition}
It is shown in \cite[Proposition~2.18]{DX22} that $P[\varphi]_{\mathcal{I}}\in \PSH(X,\theta)$ and $\varphi\sim_{\mathcal{I}}P[\varphi]_{\mathcal{I}}$. Moreover, $P[\varphi]_{\mathcal{I}}$ is always $\mathcal{I}$-model.
We can also talk about the $\sim_{\mathcal{I}}$ relation of two psh metric on $L$ in the obvious manner.

Typical model potentials are not $\mathcal{I}$-model, however, the converse is true:
\begin{proposition}
	If $\psi\in \PSH(X,\theta)$ is an $\mathcal{I}$-model potential then it is model.
\end{proposition}
\begin{proof}We need to show that  $\psi\sim_{\mathcal{I}} C[\psi]$. Let $\pi:Z\rightarrow X$ be a birational modification. Let $z\in Z$. As $\psi \leq C[\psi]+C$ for some constant $C$, it suffices to show that
	\[
		\nu(C[\psi],z)\geq \nu(\psi,z).
	\]
	Here $\nu(\psi,z)$ denotes the Lelong number of $\pi^*\psi$ at $z$.
	By \eqref{eq:CandP} and the upper semi-continuity of Lelong numbers (see \cite[Page~73, Exercise~2.7]{GZ17}), we find
	\[
		\nu(C[\psi],z)\geq \lim_{\epsilon\to 0+}\nu(P[(1-\epsilon)\psi+\epsilon V_{\theta}],z)= \lim_{\epsilon\to 0+}\nu((1-\epsilon)\psi+\epsilon V_{\theta},z)=\nu(\psi,z).
	\]
	We conclude our assertion.
\end{proof}

\subsection{Potentials with analytic singularities}
\begin{definition}\label{def:ana}
	A quasi-plurisubharmonic function (quasi-psh) $\varphi$ on $X$ is said to have \emph{analytic singularities} if for each $x\in X$, there is a neighborhood $U_x\subseteq X$ of $x$ with respect to the Euclidean topology, such that on $U_x$,
	\begin{equation}\label{eq:varphiloc}
		\varphi=c\log\left(\sum_{j=1}^{N_x}|f_j|^2\right)+\psi,
	\end{equation}
	where $c\in \mathbb{Q}_{\geq 0}$, the $f_j$'s are analytic functions on $U_x$, $N_x\in \mathbb{Z}_{>0}$ is an integer depending on $x$, $\psi\in L^{\infty}(U_x)$.
\end{definition}

\begin{definition}\label{def:anaD}
	Let $D$ be an effective normal crossing $\mathbb{R}$-divisor on $X$.   Let $D=\sum_i a_i D_i$ with $D_i$ being prime divisors and $a_i\in \mathbb{R}_{>0}$. We say that
	a quasi-psh function $\varphi$ has \emph{analytic singularities along $D$} if locally (in the Euclidean topology),
	\[
		\varphi=\sum_i a_i\log|s_i|^2+\psi,
	\]
	where $s_i$ is a local holomorphic function defining $D_i$, $\psi$ is a bounded function.
\end{definition}
In the sequel, when we talk about a normal crossing divisor, we always assume that it is effective.

Note that a potential with analytic singularities along a normal crossing $\mathbb{Q}$-divisor has analytic singularities in the sense of \cref{def:ana}.

For any quasi-psh function $\varphi$ on $X$ with analytic singularities, there is always a birational model $\pi:Y\rightarrow X$ such that $\pi^*\varphi$ has analytic singularities along a normal crossing $\mathbb{Q}$-divisor on $Y$. See \cite[Lemma~2.3.19]{MM07} for example. We remind the readers that in \cite{MM07}, the definition of analytic singularities differs slightly from ours: they require the remainder $\psi$ to be smooth instead of just bounded. However, the proof of \cite[Lemma~2.3.19]{MM07} works \emph{verbatim} with our definition.

\subsection{Quasi-equisingular approximations}
We recall the concept of quasi-equisingular approximations in the sense of \cite{Cao14, DPS01}.

Let $X$ be a connected compact Kähler manifold of dimension $n$ and $\theta$ (resp. $\theta_i$, $i=1,\ldots,n$) be a smooth real $(1,1)$-form representing a pseudo-effective $(1,1)$-cohomology class $[\theta]$ (resp. $[\theta_i]$).
Take a Kähler form $\omega$ on $X$.

\begin{definition}\label{def:quasi-eq}
	Let $\varphi\in \PSH(X,\theta)$. A \emph{quasi-equisingular approximation} is a sequence $\varphi^j\in \PSH(X,\theta+\epsilon_j\omega)$ with $\epsilon_j\to 0$ such that
	\begin{enumerate}
		\item $\varphi^j\to \varphi$ in $L^1$.
		\item $\varphi^j$ has analytic singularities.
		\item $\varphi^{j+1}\leq \varphi^j$.
		\item For any $\delta>0$, $k>0$, there is $j_0>0$ such that for $j\geq j_0$,
		      \[
			      \mathcal{I}(k(1+\delta)\varphi^j)\subseteq \mathcal{I}(k\varphi)\subseteq \mathcal{I}(k\varphi^j).
		      \]
	\end{enumerate}
\end{definition}
The existence of a quasi-equisingular approximation follows from the arguments in \cite{Cao14, Dem15, DPS01}.

\subsection{Volumes of Hermitian pseudo-effective line bundles}
Let $X$ be a smooth irreducible projective variety of dimension $n$.
\begin{definition}
	A \emph{Hermitian pseudo-effective(psef) line bundle} on $X$ is a pair $(L,\phi)$, where $L$ is a pseudo-effective line bundle on $X$ and $\phi$ is a psh metric on $L$.

 When $L$ is big, we say $(L,\phi)$ is a Hermitian big line bundle.
\end{definition}

Let $(L,\phi)$ be a Hermitian psef line bundle on $X$.
In this section, we recall the main results in \cite{DX22, DX21} concerning the volume of $(L,\phi)$. 
\begin{definition}
	The \emph{volume} of $(L,\phi)$ is defined as
	\[
		\vol(L,\phi)\coloneqq \lim_{k\to\infty}\frac{1}{k^n}h^0\left(X,L^k\otimes \mathcal{I}(k\phi)\right).
	\]
\end{definition}
The existence of the limit follows from \cite[Theorem~1.1]{DX21}.

We take a smooth Hermitian metric $h$ on $L$. Set $\theta=c_1(L,h)$. Then we can identify $\phi$ with a $\theta$-psh function $\varphi$, namely $\phi=h\exp(-\varphi)$.

\begin{theorem}[{\cite[Theorem~1.1]{DX21}}]\label{thm:DXmain}
	Under the above assumptions,
	\[
		\vol(L,\phi)=\frac{1}{n!}\int_X \theta_{P[\varphi]_{\mathcal{I}}}^n.
	\]
\end{theorem}

We argue that $\vol$ deserves the name \emph{volume} by proving that it satisfies the Brunn--Minkowski inequality. 
\begin{corollary}\label{cor:BMineq}
	Let $(L,\phi)$, $(L,\phi')$ be two Hermitian psef line bundles on $X$. Then
	\begin{equation}\label{eq:BM}
		\vol(L+L',\phi+\phi')^{1/n}\geq \vol(L,\phi)^{1/n}+\vol(L',\phi')^{1/n}.
	\end{equation}
\end{corollary}
\begin{proof}
	Fix a smooth Hermitian metric $h'$ on $L'$ with $\theta'=c_1(L',h')$. We identify $\phi'$ with $\varphi'\in \PSH(X,\theta')$.
	By \cref{thm:DXmain}, \eqref{eq:BM} is equivalent to
	\[
		\left(\int_X \left(\theta+\theta'+\ddc P^{\theta+\theta'}[\varphi+\varphi']_{\mathcal{I}}\right)^n\right)^{1/n}\geq \left(\int_X \theta_{P^{\theta}[\varphi]_{\mathcal{I}}}^n\right)^{1/n}+\left(\int_X \theta'{}_{P^{\theta'}[\varphi']_{\mathcal{I}}}^n\right)^{1/n}.
	\]
	Observe that
	\[
		P^{\theta+\theta'}[\varphi+\varphi']_{\mathcal{I}}\geq P^{\theta}[\varphi]_{\mathcal{I}}+P^{\theta'}[\varphi']_{\mathcal{I}}.
	\]
	Thus, by the monotonicity theorem of \cite{WN19}, it suffices to show that
	\[
		\left(\int_X \left(\theta+\theta'+\ddc P^{\theta}[\varphi]_{\mathcal{I}} +\ddc P^{\theta'}[\varphi']_{\mathcal{I}}\right)^n\right)^{1/n}
		\geq \left(\int_X \theta_{P^{\theta}[\varphi]_{\mathcal{I}}}^n\right)^{1/n}+\left(\int_X \theta'{}_{P^{\theta'}[\varphi']_{\mathcal{I}}}^n\right)^{1/n}.
	\]
	This follows from \cite[Theorem~6.1]{DDNL19log}.
\end{proof}

\subsection{Non-Archimedean pluripotential theory}\label{subsec:nAp}
In this section, we briefly recall the notion of Berkovich analytification of a smooth complex projective variety and the pluripotential theory in the sense of Boucksom--Jonsson \cite{BJglobal} on it.

For simplicity, we assume that $X$ is a connected smooth projective variety of dimension $n$ and $L$ is an ample line bundle on $X$. 

The set of real valuations on $\mathbb{C}(X)$ trivial on $\mathbb{C}$ is denoted by $X^{\val}$. This set can be defined in the same way for non-smooth varieties as well.

The center of a valuation $v$ is the scheme-theoretic point $c=c(v)$ of $X$ such that $v\geq 0$ on $\mathcal{O}_{X,c}$ and $v>0$ on the maximal ideal $\mathfrak{m}_{X,c}$ of $\mathcal{O}_{X,c}$. The center exists and is unique.

Let $X^{\An}$ denote the Berkovich analytification $X^{\An}$ of $X$ with respect to the trivial valuation on $\mathbb{C}$. As a set, $X^{\An}$ is the set of semi-valuations on $X$, in other words, real-valued valuations $v$
on irreducible reduced subvarieties $Y$ in $X$ that is trivial on $\mathbb{C}$. We call $Y$ the \emph{support} of the semi-valuation $v$.
In other words, 
\[
X^{\An}=\coprod_{Y} Y^{\val}.
\]
The Berkovich space $X^{\An}$ admits a natural topology, called the Berkovich topology and a sheaf of analytic functions. The natural morphism of ringed spaces $X^{\An}\rightarrow X$ allows us to pull-back $L$ to an invertible sheaf $L^{\An}$ on $X^{\An}$.
See \cite{Berk12} for more details.

In \cite{BJglobal}, Boucksom--Jonsson developed a pluripotential theory with respect to $(X^{\An},L^{\An})$, similar to its complex counterpart. In particular, there is a natural notion of plurisubharmonic metrics on $L^{\An}$. 
In \cite[Section~7.1]{BJglobal}, Boucksom--Jonsson defined the notion of energy pairings $(\varphi_0,\ldots,\varphi_n)$ between $(n+1)$ plurisubharmonic metrics $\varphi_0,\ldots,\varphi_n$ on $L^{\An}$. One can then define the space $\mathcal{E}^1(L^{\An})$ of finite energy metrics as the space of plurisubharmonic functions $\varphi$ on $L^{\An}$ such that 
\[
E(\varphi)\coloneqq \frac{1}{n+1}(\varphi,\ldots,\varphi)>-\infty.
\]
Note that our definition of $E$ differs from the definition of \cite{BJglobal} by a multiple $\frac{1}{V}$.
We will explain the relation between the non-Archimedean pluripotential theory and the complex pluripotential theory in \cref{subsec:apptonA}.

\section{The Okounkov bodies of almost semigroups}\label{sec:clo}
Fix an integer $n\geq 0$. Fix a closed convex cone $C\subseteq \mathbb{R}^{n}\times \mathbb{R}_{\geq 0}$ such that $C\cap \{x_{n+1}=0\}=\{0\}$. Here $x_{n+1}$ is the last coordinate of $\mathbb{R}^{n+1}$.
\subsection{Generality on semigroups}
Write $\hat{\mathcal{S}}(C)$ for the set of subsets of $C\cap \mathbb{Z}^{n+1}$ and $\mathcal{S}(C)$ for the set of sub-semigroups $S\subseteq C\cap \mathbb{Z}^{n+1}$. For each $k\in \mathbb{N}$ and $S\in \hat{\mathcal{S}}(C)$, we write
\[
S_k\coloneqq \left\{x\in \mathbb{Z}^n: (x,k)\in S \right\}.
\]
Note that $S_k$ is a finite set by our assumption on $C$.

We introduce a pseudometric on $\hat{\mathcal{S}}(C)$ as follows:
\[
d(S,S')\coloneqq \varlimsup_{k\to\infty} k^{-n} (|S_k|+|S'_k|-2|(S\cap S')_k|).
\]
Here $|\bullet|$ denotes the cardinality of a finite set.
\begin{lemma}\label{lma:dps}
    The above defined $d$ is a pseudometric on $\hat{\mathcal{S}}(C)$.
\end{lemma}
\begin{proof}
    Only the triangle inequality needs to be argued. Take $S,S',S''\in \hat{\mathcal{S}}(C)$. We claim that for any $k\in \mathbb{N}$,
    \[
    |S_k|+|S'_k|-2|S_k\cap S'_k|+|S''_k|+|S'_k|-2|S''_k\cap S'_k|\geq |S_k|+|S''_k|-2|S_k\cap S''_k|.
    \]
    From this the triangle inequality follows. To argue the claim, we rearrange it to the following form:
    \[
    |S_k'|-|S_k\cap S_k'|\geq |S'_k\cap S_k''|-|S_k\cap S''_k|,
    \]
    which is obvious.
\end{proof}
Given $S,S'\in \hat{\mathcal{S}}(C)$, we say $S$ is equivalent to $S'$ and write $S\sim S'$ if $d(S,S')=0$. This is an equivalence relation by \cref{lma:dps}.

\begin{lemma}\label{lma:dBil}
    Given $S,S',S''\in \hat{\mathcal{S}}(C)$, we have
    \[
    d(S\cap S'',S'\cap S'')\leq d(S,S').
    \]
    In particular, if $S^i,S'^i\in \hat{\mathcal{S}}(C)$ ($i\in \mathbb{N}$) and $S^i\to S$, $S'^i\to S'$, then
    \[
    S^i\cap S'^i\to S\cap S'.
    \]
\end{lemma}
\begin{proof}
    Observe that for any $k\in \mathbb{N}$,
    \[
    |S_k\cap S_k''|-|S_k\cap S_k'\cap S_k''|\leq |S_k|-|S_k\cap S_k'|.
    \]
    The same holds if we interchange $S$ with $S'$. It follows that
    \[
    |S_k\cap S_k''|+|S_k'\cap S_k''|-2|S_k\cap S_k'\cap S_k''|\leq |S_k|+|S'_k|-2|S_k\cap S_k'|.
    \]
    The first assertion follows.

    Next we compute
    \[
    d(S^i\cap S'^i,S\cap S')\leq d(S^i\cap S'^i,S^i\cap S')+d(S^i\cap S',S\cap S')\leq d(S'^i,S')+d(S^i,S)
    \]
    and the second assertion follows.
\end{proof}

The volume of $S\in \mathcal{S}(C)$ is defined as
\[
\vol S\coloneqq \lim_{k\to\infty} (ka)^{-n}|S_{ka}|=\varlimsup_{k\to\infty} k^{-n}|S_k|,
\]
where $a$ is a sufficiently divisible positive integer.
The existence of the limit and its independence from $a$ both follow from the more precise result \cite[Theorem~2]{KK12}.
\begin{lemma}\label{lma:vollip}
    Let $S,S'\in \mathcal{S}(C)$, then
    \[
    |\vol S-\vol S'|\leq d(S,S').
    \]
\end{lemma}
\begin{proof}
    By definition, we have
    \[
    d(S,S')\geq \vol S+\vol S'-2\vol (S\cap S').
    \]
    It follows that $\vol S-\vol S'\leq d(S,S')$ and $\vol S'-\vol S\leq d(S,S')$.
\end{proof}
We define $\overline{\mathcal{S}}(C)$ as the closure of $\mathcal{S}(C)$ in $\hat{\mathcal{S}}(C)$ with respect to the topology defined by the pseudometric $d$. 
By \cref{lma:vollip}, $\vol\colon \mathcal{S}(C)\rightarrow \mathbb{R}$ admits a unique $1$-Lipschitz extension to 
\begin{equation}\label{eq:volex}
\vol\colon \overline{\mathcal{S}}(C)\rightarrow \mathbb{R}.
\end{equation}
\begin{lemma}\label{lma:volcompa}
    Suppose that $S,S'\in \overline{\mathcal{S}}(C)$ and $S\subseteq S'$. Then
    \[
    \vol S\leq \vol S'.
    \]
\end{lemma}
\begin{proof}
    Take sequences $S^j,S'^j$ in $\mathcal{S}(C)$ such that $S^j\to S$, $S'^j\to S'$. By \cref{lma:dBil}, after replacing $S^j$ by $S^j\cap S'^j$, we may assume that $S^j\subseteq S'^j$ for each $j$. Then our assertion follows easily.
\end{proof}

\subsection{Okounkov bodies of semigroups}
Given $S\in \hat{\mathcal{S}}(C)$, we will write $C(S)\subseteq C$ for the closed convex cone generated by $S\cup \{0\}$. Moreover, for each $k\in \mathbb{Z}_{>0}$, we define
\[
\Delta_k(S)\coloneqq \Conv \left\{k^{-1}x\in \mathbb{R}^n:x\in S_k \right\}\subseteq \mathbb{R}^n.
\]
Here $\Conv$ denotes the convex hull.
\begin{definition}
    Let $\mathcal{S}'(C)$ be the subset of $\mathcal{S}(C)$ consisting of semigroups $S$ such that $S$ generates $\mathbb{Z}^{n+1}$ (as an Abelian group).
\end{definition}
Note that for any $S\in \mathcal{S}'(C)$, the cone $C(S)$ has full dimension (i.e. the topological interior is non-empty). Given a full-dimensional subcone $C'\subseteq C$, it is clear that $C'\cap \mathbb{Z}^{n+1}\in \mathcal{S}'(C)$.
 
This class behaves well under intersections:
\begin{lemma}\label{lma:intersecS'}
    Let $S,S'\in \mathcal{S}'(C)$. Assume that $\vol (S\cap S')>0$, then $S\cap S'\in \mathcal{S}'(C)$. 
\end{lemma}
The lemma obviously fails if $\vol (S\cap S')=0$.
\begin{proof}
    We first observe that the cone $C(S)\cap C(S')$ has full dimension since otherwise $\vol(S\cap S')=0$. Take a full-dimensional subcone $C'$ in $C(S)\cap C(S')$ such that $C'$ intersects the boundary of $C(S)\cap C(S')$ only at $0$. It follows from \cite[Theorem~1]{KK12} that there is an integer $N>0$ such that for any $x\in \mathbb{Z}^{n+1}\cap C'$ with Euclidean norm no less than $N$ lies in $S\cap S'$. Therefore, $S\cap S'\in \mathcal{S}'(C)$.
\end{proof}

We recall the following definition from \cite{KK12}.
\begin{definition}\label{def:Okokk}
 Given $S\in \mathcal{S}'(C)$, its \emph{Okounkov body} is defined as follows
\[
\Delta(S)\coloneqq \left\{x\in \mathbb{R}^n:(x,1)\in C(S) \right\}.
\]    
\end{definition}

\begin{theorem}\label{thm:HausOkoun}
	For each $S\in \mathcal{S}'(C)$, we have
	\begin{equation}\label{eq:volWvolDelta}
		\vol S=\lim_{k\to\infty} k^{-n}|S_k|=\vol \Delta(S)>0.
	\end{equation}
    Moreover, as $k\to\infty$,
	\begin{equation}\label{eq:HausconvDeltaGLS}
		\Delta_{k}(S)\xrightarrow{d_n}\Delta(S).
	\end{equation}
\end{theorem}
This is essentially proved in \cite[Lemma~4.8]{WN14}, which itself follows from a theorem of Khovanskii \cite{Kho92}. We remind the readers that \eqref{eq:volWvolDelta} fails for a general $W\in \mathcal{S}(C)$, see \cite[Theorem~2]{KK12}.
\begin{proof}
	The equalities \eqref{eq:volWvolDelta} follow from the general theorem \cite[Theorem~2]{KK12}.



	It remains to prove \eqref{eq:HausconvDeltaGLS}. By the argument of \cite[Lemma~4.8]{WN14},  for any compact set $K\subseteq \Int \Delta(S)$, there is $k_0>0$ such that for any $k\geq k_0$, $\alpha\in K\cap (k^{-1}\mathbb{Z})^n$ implies that $\alpha\in \Delta_k(S)$.

	In particular, taking $K=\Delta(S)^{\delta}$ for any $\delta>0$ and applying \cref{lma:latcvb}, we find
	\[
		d_n(\Delta(S),\Delta_k(S))\leq n^{1/2}k^{-1}+\delta
	\]
	when $k$ is large enough. This implies \eqref{eq:HausconvDeltaGLS}.
\end{proof}
\begin{corollary}\label{cor:dist}
    Let $S,S'\in \mathcal{S}'(C)$. Assume that $\vol (S\cap S')>0$, then we have
    \[
    d(S,S')=\vol(S)+\vol(S')-2\vol(S\cap S').
    \]
\end{corollary}
\begin{proof}
    This is a direct consequence of \cref{lma:intersecS'} and \eqref{eq:volWvolDelta}.
\end{proof}

\begin{lemma}\label{lma:regularizat}
    Given $S\in \mathcal{S}'(C)$, we have $S\sim \Reg(S)$.
\end{lemma}
Recall that the regularization $\Reg(S)$ of $S$ is defined as $C(S)\cap \mathbb{Z}^{n+1}$.
\begin{proof}
    Since $S$ and $\Reg(S)$ have the same Okounkov body, 
    we have $\vol S=\vol \Reg(S)$ by \cref{thm:HausOkoun}. By \cref{cor:dist} again,
    \[
    d(\Reg(S),S)=\vol \Reg(S)-\vol S=0.
    \]
\end{proof}

\begin{lemma}\label{lma:Deltaindclass}
	Let $S,S'\in \mathcal{S}'(C)$. Assume that $d(S,S')=0$, then $\Delta(S)=\Delta(S')$.
\end{lemma}
\begin{proof}
Observe that $\vol (S\cap S')>0$, as otherwise 
\[
d(S,S')\geq \vol S+\vol S'>0,
\]
which is a contradiction.

It follows from \cref{lma:intersecS'} that $S\cap S'\in \mathcal{S}'(C)$. It suffices to show that $\Delta(S)=\Delta(S\cap S')$.
In fact, suppose that this holds, since $\vol \Delta(S')=\vol S'=\vol S=\vol \Delta(S)$, the inclusion $\Delta(S')\supseteq \Delta(S\cap S')=\Delta(S)$ is an equality.

By \cref{lma:dBil}, we can therefore replace $S'$ by $S\cap S'$ and assume that $S\supseteq S'$. Then clearly $\Delta(S)\supseteq \Delta(S')$. By \eqref{eq:volWvolDelta},
	\[
		\vol \Delta(S)=\vol \Delta(S').
	\]
	Thus, $\Delta(S)=\Delta(S')$ by \cref{lma:volcbimpeq}.
\end{proof}

\begin{lemma}\label{lma:Sprimeint}
    Suppose that $S^i\in \mathcal{S}'(C)$ is a decreasing sequence such that 
    \[
    \lim_{i\to\infty}\vol S^i>0.
    \]
    Then there is $S\in \mathcal{S}'(C)$ such that $S^i\to S$.
\end{lemma}
In general, one cannot simply take $S=\bigcap_i S^i$. For example, consider the sequence $S^i=S^1\cap \{x_{n+1}\geq i\}$.
\begin{proof}
    By \cref{lma:regularizat}, we may replace $S^i$ by its regularization and assume that $S^i=C(S^i)\cap \mathbb{Z}^{n+1}$. We define
    \[
    S=\left( \bigcap_{i=1}^{\infty}C(S^i) \right) \cap \mathbb{Z}^{n+1}.
    \]
    Since $\bigcap_{i=1}^{\infty}C(S^i)$ is a full-dimensional cone by assumption, we have $S\in \mathcal{S}'(C)$.
    By \cref{cor:dist} and \cref{thm:HausOkoun}, we can compute the distance
    \[
    d(S,S^i)=\vol S^i-\vol S=\vol \Delta (S^i)-\vol \Delta (S),
    \]
    which tends to $0$ by construction.
\end{proof}

\subsection{Okounkov bodies of almost semigroups}\label{subsec:Okobalmosg}
\begin{definition}
    We define $\overline{\mathcal{S}'(C)}_{>0}$ as elements in the closure of $\mathcal{S}'(C)$ in $\hat{\mathcal{S}}(C)$ with positive volume. An element in $\overline{\mathcal{S}'(C)}_{>0}$ is called an \emph{almost semigroup} in $C$.
\end{definition}
Recall that the volume here is defined in \eqref{eq:volex}.

Our goal is to prove the following theorem:
\begin{theorem}\label{thm:Okocont}
	The Okounkov body map $\Delta\colon \mathcal{S}'(C) \rightarrow \mathcal{K}_n$
    as defined in \cref{def:Okokk} admits a unique continuous extension 
	\begin{equation}\label{eq:Deltagensg}
		\Delta \colon \overline{\mathcal{S}'(C)}_{>0}\rightarrow \mathcal{K}_n.
	\end{equation}
     Moreover, for any $S\in \overline{\mathcal{S}'(C)}_{>0}$, we have
	\begin{equation}\label{eq:volWfinal}
		\vol S=\vol \Delta(S).
	\end{equation}
\end{theorem}

\begin{proof}
The uniqueness of the extension is clear as long as it exists. Moreover, \eqref{eq:volWfinal} follows easily from \cref{thm:HausOkoun} and \cref{thm:contvol} by continuity. 
It remains to argue the existence of the continuous extension. We first construct an extension and prove its continuity.

\textbf{Step~1}. We construct the desired map \eqref{eq:Deltagensg}. Let $S\in \overline{\mathcal{S}'(C)}_{>0}$. We wish to construct a convex body $\Delta(S)\in \mathcal{K}_n$.

    Let $S^i\in \mathcal{S}'(C)$ be a sequence that converges to $S$ such that 
    \[
    d(S^i,S^{i+1})\leq 2^{-i}.
    \]
    For each $i,j\geq 0$, we introduce
    \[
    S^{i,j}=S^i\cap S^{i+1}\cdots \cap S^{i+j}.
    \]
    Then by \cref{lma:dBil}, 
    \[
    d(S^{i,j},S^{i,j+1})\leq 2^{-i-j}.
    \]
    Take $i_0>0$ large enough so that for $i\geq i_0$, $\vol S^i>2^{-1}\vol S$ and $2^{2-i}< \vol S$ and hence
    \[
        \vol S^i-\vol S^{i,j}\leq d(S^{i,0},S^{i,1})+d(S^{i,1},S^{i,2})+\cdots+d(S^{i,j-1},S^{i,j})\leq 2^{1-i}.
    \]
    It follows that $\vol S^{i,j}> 2^{-1}\vol S-2^{1-i}>0$ whenever $i\geq i_0$. In particular, by \cref{lma:intersecS'}, $S^{i,j}\in \mathcal{S}'(C)$ for $i\geq i_0$.
    
    By \cref{lma:Sprimeint}, for $i\geq i_0$, there exists $T^i\in \mathcal{S}'(C)$ such that $S^{i,j}\to T^i$ as $j\to\infty$. Moreover,
    \[
    d(T^i,S)=\lim_{j\to\infty} d(S^{i,j},S)\leq \lim_{j\to\infty} d(S^{i,j},S^i)+d(S^i,S)\leq 2^{1-i}+d(S^i,S).
    \]
    Therefore, $T^i\to S$. We then define
    \[
    \Delta(S)\coloneqq \overline{\bigcup_{i=i_0}^{\infty} \Delta(T^i)}.
    \]
    In other words, we have defined
    \[
	\Delta(S)\coloneqq \varliminf_{i\to\infty}\Delta(S^i).
    \]
    This is an honest limit: if $\Delta$ is the limit of a subsequence of $\Delta(S^i)$, then $\Delta(S)\subseteq \Delta$  by \eqref{eq:liminflimsup}. Comparing the volumes, we find that equality holds. So by \cref{thm:Bst},
    \begin{equation}\label{eq:deltawtemp}
	\Delta(S)=\lim_{i\to\infty}\Delta(S^i).
    \end{equation}

    Next we claim that $\Delta(S)$ as defined above does not depend on the choice of the sequence $S^i$. In fact, suppose that $S'^i\in \mathcal{S}'(C)$ is another sequence satisfying the same conditions as $S^i$. 
    The same holds for $R^i\coloneqq S^{i+1}\cap S'^{i+1}$. It follows that
    \[
        \lim_{i\to\infty}\Delta(R^i)\subseteq \lim_{i\to\infty}\Delta(S^i). 
    \]
    Comparing the volumes, we find that equality holds. The same is true with $S'^i$ in place of $S^i$. So we conclude that $\Delta(S)$ as in \eqref{eq:deltawtemp} does not depend on the choices we made.

    \textbf{Step~2}. It remains to prove the continuity of $\Delta$ defined in Step~1.
    Suppose that $S^i\in \overline{\mathcal{S}'(C)}_{>0}$ is a sequence with limit $S\in \overline{\mathcal{S}'(C)}_{>0}$.
    We want to show that
	\begin{equation}\label{eq:temp5}
		\Delta(S^i)\xrightarrow{d_n}\Delta(S).
	\end{equation}
    
    We first reduce to the case where $S^i\in \mathcal{S}'(C)$. 
    By \eqref{eq:deltawtemp}, for each $i$, we can choose $T^i\in \mathcal{S}'(C)$ such that  $d(S^i,T^i)<2^{-i}$ and $d_n(\Delta(S^i),\Delta(T^i))<2^{-i}$. If we have shown $\Delta(T^i)\xrightarrow{d_n}\Delta(S)$, then \eqref{eq:temp5} follows immediately.

    Next we reduce to the case where $d(S^i,S^{i+1})\leq 2^{-i}$. In fact, thanks to \cref{thm:Bst}, in order to prove \eqref{eq:temp5}, it suffices to show that each subsequence of $\Delta(S^i)$ admits a subsequence that converges to $\Delta(S)$. Hence, we easily reduce to the required case.

    After these reductions, \eqref{eq:temp5} is nothing but \eqref{eq:deltawtemp}.
\end{proof}

\begin{corollary}\label{cor:Okocomp}
    Suppose that $S,S'\in \overline{\mathcal{S}'(C)}_{>0}$ with $S\subseteq S'$, then 
    \begin{equation}\label{eq:Deltacontain}
    \Delta(S)\subseteq \Delta(S').
    \end{equation}
\end{corollary}
\begin{proof}
    Let $S^j,S'^j\in \mathcal{S}'(C)$ be elements such that $S^j\to S$, $S'^j\to S'$. Then it follows from \cref{lma:dBil} that $S^j\cap S'^j \to S$. Since $\vol$ is continuous, for large $j$, $S^j\cap S'^j$ has positive volume and hence lies in $\mathcal{S}'(C)$ by \cref{lma:intersecS'}. We may therefore replace $S^j$ by $S^j\cap S'^j$ and assume that $S^j\subseteq S'^j$. Hence \eqref{eq:Deltacontain} follows from the continuity of $\Delta$ proved in \cref{thm:Okocont}.
\end{proof}

\begin{remark}
    As the readers can easily verify, the construction of $\Delta$ is independent of the choice of $C$ in the following sense: Suppose that $C'$ is another cone satisfying the same assumptions as $C$ and $C'\supseteq C$, then the Okounkov body map $\Delta \colon \overline{\mathcal{S}'(C')}_{>0}\rightarrow \mathcal{K}_n$ is an extension of the corresponding map \eqref{eq:Deltagensg}. We will constantly use this fact without further explanations. 
\end{remark}

\section{The metric on the space of singularity types}\label{sec:singtype}

Let $X$ be a connected compact Kähler manifold of dimension $n$ and $\theta$ (resp. $\theta_i$, $i=1,\ldots,n$) be a smooth real $(1,1)$-form representing a big $(1,1)$-cohomology class $[\theta]$ (resp. $[\theta_i]$).
Let $\omega$ be a Kähler form on $X$.

In this section, we develop  further the metric geometry on the space of singularity types of quasi-psh functions, initiated in \cite{DDNLmetric} and studied further in \cite{DX22}.

As explained in \cite[Section~3]{DDNLmetric}, one can introduce a pseudometric $d_S$ on the set of singularity types of functions in $\PSH(X,\theta)$. In particular, $d_S$ lifts to a pseudometric on $\PSH(X,\theta)$ as well. 
We do not recall the precise definition, as the following double inequality from \cite[Proposition~3.5]{DDNLmetric} will be enough for us. For any $\varphi,\psi\in \PSH(X,\theta)$ we have
\begin{equation}\label{eq:defds}
	d_S(\varphi,\psi)\leq \sum_{i=0}^n \left(2\int_X\theta^i_{\max\{\varphi,\psi\}}\wedge \theta_{V_{\theta}}^{n-i}-\int_X\theta^i_{\varphi}\wedge \theta_{V_{\theta}}^{n-i}-\int_X\theta^i_{\psi}\wedge \theta_{V_{\theta}}^{n-i}\right)\leq C_0d_S(\varphi,\psi),
\end{equation}
where $C_0>1$ is a constant depending only on $n$. In addition, $d_S(\varphi,\psi)=0$ if and only if
\[
	C[\varphi]=C[\psi].
\]
When there is a risk of confusion, we write $d_{S,\theta}$ instead of $d_S$.

\begin{lemma}\label{lma:Ceilmass}
	Let $\varphi_i\in \PSH(X,\theta_i)$ ($i=1,\ldots,n$). Then
	\[
		\int_X \theta_{1,\varphi_1}\wedge \cdots \wedge \theta_{n,\varphi_n}=\int_X \theta_{1,C[\varphi_1]}\wedge \cdots \wedge \theta_{n,C[\varphi_n]}.
	\]
\end{lemma}
\begin{proof}
	From the definition \eqref{eq:Cdef}, we have $[u] \preceq [C[u]]$, the $\leq$ direction is obvious. For the reverse direction, recall that $C[\varphi_i]=\lim_{\epsilon\to 0+}P[(1-\epsilon)\varphi_i+\epsilon V_{\theta_i}]$. Thus, for $\epsilon\in (0,1)$,
	\[
		\int_X \theta_{1,C[\varphi_1]}\wedge \cdots \wedge \theta_{n,C[\varphi_n]}\geq (1-\epsilon)^n\int_X \theta_{1,\varphi_1}\wedge \cdots \wedge \theta_{n,\varphi_n}.
	\]
	Letting $\epsilon\to 0+$, we conclude.
\end{proof}

\begin{theorem}\label{thm:dsmixedmass}
	Let $\varphi_i^k,\varphi_i\in \PSH(X,\theta_i)$ ($i=1,\ldots,n$, $k\in \mathbb{N}$). Assume that $\varphi_i^k\xrightarrow{d_{S,\theta_i}} \varphi_i$ as $k\to\infty$. Then
	\begin{equation}\label{eq:mixedmassconv}
		\lim_{k\to\infty}\int_X \theta_{1,\varphi_1^k}\wedge\cdots\wedge\theta_{n,\varphi_n^k}=\int_X \theta_{1,\varphi_1}\wedge\cdots\wedge\theta_{n,\varphi_n}.
	\end{equation}
\end{theorem}
\begin{proof}
	By \cref{lma:Ceilmass} and \cite[Theorem~3.3]{DDNLmetric}, we may assume that $\varphi_i^k$ and $\varphi^k$ are model potentials.

	\textbf{Step 1}. We assume that there is a constant $\delta>0$, such that for all $i$ and $k$,
	\[
		\int_X \theta_{i,\varphi_i^k}^n>\delta.
	\]
	In order to show \eqref{eq:mixedmassconv}, it suffices to prove that any subsequence of $\int_X \theta_{1,\varphi_1^k}\wedge\cdots\wedge\theta_{n,\varphi_n^k}$ has a converging subsequence with limit $\int_X \theta_{1,\varphi_1}\wedge\cdots\wedge\theta_{n,\varphi_n}$. Thus, by \cite[Theorem~5.6]{DDNLmetric}, we may assume that for each fixed $i$, $\varphi_i^k$ is either increasing or decreasing. We may assume that for $i\leq i_0$, the sequence is decreasing and for $i>i_0$, the sequence is increasing.

	Recall that in \eqref{eq:mixedmassconv} the $\geq$ inequality always holds \cite[Theorem~2.3]{DDNL18mono}, it suffices to prove
	\begin{equation}\label{eq:limsup}
		\varlimsup_{k\to\infty}\int_X \theta_{1,\varphi_1^k}\wedge\cdots\wedge\theta_{n,\varphi_n^k}\leq\int_X \theta_{1,\varphi_1}\wedge\cdots\wedge\theta_{n,\varphi_n}.
	\end{equation}
	By Witt Nyström's monotonicity theorem \cite{WN19, DDNL18mono}, in order to prove \eqref{eq:limsup}, we may assume that for $j>i_0$, the sequences $\varphi_{j}^k$ are constant. Thus, we are reduced to the case where for all $i$, $\varphi_i^k$ are decreasing.

	In this case,
	for each $i$ we may take an increasing sequence $b^k_i>1$, tending to $\infty$, such that
	\[
		(b^k_i)^n\int_X \theta_{i,\varphi_i}^n> \left((b^k_i)^n-1\right) \int_X\theta_{i,\varphi_i^k}^n.
	\]
	Let $\psi_i^k$ be the maximal $\theta_i$-psh function, such that
	\[
		(b^k_i)^{-1}\psi_i^k+\left(1-(b^k_i)^{-1}\right)\varphi_i^k\leq  \varphi_i,
	\]
	whose existence is guaranteed by \cite[Lemma~4.3]{DDNLmetric}.

	Then by Witt Nyström's monotonicity theorem \cite{WN19, DDNL18mono} again,
	\[
		\prod_{i=1}^n\left(1-(b^k_i)^{-1}\right)\int_X \theta_{1,\varphi_1^k}\wedge\cdots\wedge\theta_{n,\varphi_n^k}\leq \int_X \theta_{1,\varphi_1}\wedge\cdots\wedge\theta_{n,\varphi_n}.
	\]
	Let $k\to\infty$, we conclude \eqref{eq:limsup}.

	\textbf{Step 2}. Now we deal with the general case.

	We claim that if $t\in (0,1]$, $(1-t)\varphi_i^k+tV_{\theta_i}\xrightarrow{d_S} (1-t)\varphi_i+tV_{\theta_i}$ as $k\to\infty$ for each $i$. From this and Step~1, we find that for $t_i\in (0,1]$,
	\[
		\lim_{k\to\infty}\int_X \theta_{1,(1-t_1)\varphi_1^k+t_1V_{\theta_1}}\wedge\cdots\wedge\theta_{n,(1-t_n)\varphi_n^k+t_nV_{\theta_n}}=\int_X \theta_{1,(1-t_1)\varphi_1+t_1V_{\theta_1}}\wedge\cdots\wedge\theta_{n,(1-t_n)\varphi_n+t_nV_{\theta_n}}.
	\]
	Thus, \eqref{eq:mixedmassconv} follows, after letting $t_i \searrow 0$.

	It remains to prove the claim. For simplicity, let us suppress the $i$ indices momentarily. We need to argue that
	\[
		2\int_X \theta_{\max\{(1-t)\varphi^k+tV_{\theta},(1-t)\varphi+tV_{\theta}\}}^j\wedge \theta_{V_{\theta}}^{n-j}-\int_X \theta_{(1-t)\varphi^k+tV_{\theta}}^j\wedge \theta_{V_{\theta}}^{n-j}-\int_X \theta_{(1-t)\varphi+tV_{\theta}}^j\wedge \theta_{V_{\theta}}^{n-j}\to 0.
	\]
	Note that the above expression is a linear combination of terms of the following type:
	\[
		2\int_X \theta^r_{\max\{\varphi^k,\varphi\}} \wedge \theta^{n-r}_{V_{\theta}}-\int_X \theta^r_{\varphi^k} \wedge \theta^{n-r}_{V_{\theta}}-\int_X \theta^r_{\varphi^k} \wedge \theta^{n-r}_{V_{\theta}}.
	\]
    Thanks to \eqref{eq:defds}, all these expressions tend to $0$ as $k\to \infty$ since $\varphi^k\xrightarrow{d_S}\varphi$, which proves our claim.
\end{proof}

\begin{corollary}\label{cor:dspert}
	Let $\varphi^k,\varphi\in \PSH(X,\theta)$ ($k\in \mathbb{N}$). Let $\omega$ be a Kähler form on $X$. Assume that  $\varphi^k\xrightarrow{d_{S,\theta}}\varphi$. Then $\varphi^k\xrightarrow{d_{S,\theta+\omega}}\varphi$.
\end{corollary}
\begin{proof}
	It suffices to show that for each $j=0,\ldots,n$, we have
	\[
		2\int_X (\theta+\omega)_{\max\{\varphi^k,\varphi\}}^j\wedge (\theta+\omega)_{V_{\theta+\omega}}^{n-j}-\int_X (\theta+\omega)_{\varphi^k}^j\wedge (\theta+\omega)_{V_{\theta+\omega}}^{n-j}-\int_X (\theta+\omega)_{\varphi}^j\wedge (\theta+\omega)_{V_{\theta+\omega}}^{n-j}\to 0
	\]
	as $k\to\infty$. Note that this quantity is a linear combination of terms of the following form:
	\[
		2\int_X \theta_{\max\{\varphi^k,\varphi\}}^r\wedge \omega^{j-r}\wedge (\theta+\omega)_{V_{\theta+\omega}}^{n-j}-\int_X \theta_{\varphi^k}^r\wedge \omega^{j-r}\wedge (\theta+\omega)_{V_{\theta+\omega}}^{n-j}-\int_X \theta_{\varphi}^r\wedge \omega^{j-r}\wedge (\theta+\omega)_{V_{\theta+\omega}}^{n-j},
	\]
	where $r=0,\ldots,j$. By \cref{thm:dsmixedmass}, it suffices to show that $\max\{\varphi,\varphi^k\}\xrightarrow{d_S}\varphi$. But this follows from \cite[Proposition 3.5]{DDNLmetric}.
\end{proof}

\begin{corollary}\label{cor:I-modelpert}
	Let $\varphi\in \PSH(X,\theta)$ be an  $\mathcal{I}$-model potential of positive mass. Let $\omega$ be a Kähler form on $X$.
	Then $P^{\theta+\omega}[\varphi]$ is $\mathcal{I}$-model.
\end{corollary}
\begin{proof}
By \cite[Theorem~3.8]{DX21}, we may take a sequence $\varphi^j$ with analytic singularities such that $\varphi^j\xrightarrow{d_{S,\theta}}\varphi$. Then $\varphi^j\xrightarrow{d_{S,\theta+\omega}}\varphi$ by \cref{cor:dspert}. Thus, $P^{\theta+\omega}[\varphi]$ is $\mathcal{I}$-model.
\end{proof}

\begin{corollary}\label{cor:dssum}
	Let $\varphi^j,\varphi\in \PSH(X,\theta_1), \psi^j,\psi\in \PSH(X,\theta_2)$ ($j\in \mathbb{N}$). Assume that $\varphi^j\xrightarrow{d_{S,\theta_1}}\varphi$, $\psi^j\xrightarrow{d_{S,\theta_2}}\psi$. Then
	\[
		\varphi^j+\psi^j\xrightarrow{d_{S,\theta_1+\theta_2}}\varphi+\psi.
	\]
\end{corollary}
\begin{proof}
	Let $\theta=\theta_1+\theta_2$.
	It suffices to show that for each $r=0,\ldots,n$,
	\[
		2\int_X \theta_{\max\{\varphi^j+\psi^j,\varphi+\psi\}}^r\wedge \theta_{V_{\theta}}^{n-r}-\int_X \theta_{\varphi^j+\psi^j}^r\wedge \theta_{V_{\theta}}^{n-r}-\int_X \theta_{\varphi+\psi}^r\wedge\theta_{V_{\theta}}^{n-r}\to 0.
	\]
	Observe that
	\[
		\max\{\varphi^j+\psi^j,\varphi+\psi\}\leq \max\{\varphi^j,\varphi\}+\max\{\psi^j,\psi\}.
	\]
	Thus, it suffices to show that
	\[
		2\int_X \theta_{\max\{\varphi^j,\varphi\}+\max\{\psi^j,\psi\}}^r\wedge \theta_{V_{\theta}}^{n-r}-\int_X \theta_{\varphi^j+\psi^j}^r\wedge \theta_{V_{\theta}}^{n-r}-\int_X \theta_{\varphi+\psi}^r\wedge\theta_{V_{\theta}}^{n-r}\to 0.
	\]
	The left-hand side is a linear combination of
	\[
		2\int_X \theta_{1,\max\{\varphi^j,\varphi\}}^a\wedge \theta_{2,\max\{\psi^j,\psi\}}^{r-a}\wedge \theta_{V_{\theta}}^{n-r}-\int_X \theta_{1,\varphi^j}^a\wedge \theta_{2,\psi^j}^{r-a} \wedge \theta_{V_{\theta}}^{n-r}-\int_X \theta_{1,\varphi}^a\wedge \theta_{2,\psi}^{r-a} \wedge \theta_{V_{\theta}}^{n-r}
	\]
	with $a=0,\ldots,r$. Observe that $\max\{\varphi^j,\varphi\}\xrightarrow{d_S}\varphi$ and  $\max\{\psi^j,\psi\}\xrightarrow{d_S}\psi$ by \cite[Proposition~3.5]{DDNLmetric}, each term tends to $0$ by \cref{thm:dsmixedmass}.
\end{proof}
Finally, we prove the continuity of $P[\bullet]_{\mathcal{I}}$.
\begin{theorem}\label{thm:contPI}
	The map $\PSH(X,\theta)_{>0}\rightarrow \PSH(X,\theta)_{>0}$ given by $\varphi\mapsto P[\varphi]_{\mathcal{I}}$ is continuous with respect to the $d_S$-pseudometric.
\end{theorem}
Here $\PSH(X,\theta)_{>0}$ denotes the subset of $\PSH(X,\theta)$ consisting of $\varphi$ with $\int_X \theta_{\varphi}^n>0$.
\begin{proof}
	Let $\varphi_i,\varphi\in    \PSH(X,\theta)_{>0}$, $\varphi_i\xrightarrow{d_S}\varphi$. We want to show that
	\begin{equation}
		P[\varphi_i]_{\mathcal{I}}\xrightarrow{d_S}P[\varphi]_{\mathcal{I}}.
	\end{equation}
	We may assume that the $\varphi_i$'s and $\varphi$ are all model potentials by \cite[Theorem~3.3]{DDNLmetric}. By \cite[Theorem~5.6]{DDNLmetric}, we may assume that $\varphi_i$ is either increasing or decreasing. These cases follow from \cite[Lemma~2.21]{DX22} and \cite[Proposition~4.8, Lemma~4.1]{DDNLmetric}.
\end{proof}

\section{Partial Okounkov bodies}\label{sec:PoB}
Let $X$ be an irreducible smooth complex projective variety of dimension $n$ and $L$ be a big line bundle on $X$. Take a singular psh metric $\phi$ on $L$. We assume that $\vol(L,\phi)>0$.
Let $h$ be a smooth Hermitian metric on $L$. Let $\theta=c_1(L,h)$. Then we can identify $\phi$ with a function $\varphi\in \PSH(X,\theta)$.
We will use interchangeably the notations $(\theta,\varphi)$ and $(L,\phi)$.

For each $k\geq 0$,
\[
	W_k(\theta,\varphi)\coloneqq \mathrm{H}^0(X,L^k\otimes \mathcal{I}(k\varphi)),\quad W(\theta,\varphi)\coloneqq \bigoplus_{k=0}^{\infty} W_k(\theta,\varphi).
\]
We omit $(\theta,\varphi)$ from our notations when there is no risk of confusion. Observe that $W_k(\theta,\varphi)\neq 0$ when $k$ is large enough, as follows from \cref{thm:DXmain}.

Fix a rank $n$ valuation $\nu:\mathbb{C}(X)^{\times}\rightarrow \mathbb{Z}^n$. We will write
\[
\begin{aligned}
\Gamma_{\nu,k}(\theta,\varphi)= & \left\{ k^{-1}\nu(s) : s\in W_k(\theta,\varphi)^{\times} \right\},\quad k\geq 1\\
\Gamma_{\nu}(\theta,\varphi)= & \left\{ (\nu(s),k) : k\in \mathbb{N}, s\in W_k(\theta,\varphi)^{\times} \right\}.
\end{aligned}
\]

In \cite{LM09}, Lazarsfeld--Mustaț{\u{a}} only considered the case where $\nu$ is induced by an admissible flag, but thanks to \cref{rmk:Abhyankar}, their results can be easily extended to the current setup. We will use these results without further comments.

\subsection{Construction of partial Okounkov bodies}
Our goal in this section is to show that $\Gamma_{\nu}(\theta,\varphi)\in \overline{\mathcal{S}'(\Delta_{\nu}(L))}_{>0}$, namely it is an almost semigroup. Then we shall define
\begin{equation}\label{eq:Deltalbdef}
\Delta_{\nu}(\theta,\varphi)\coloneqq \Delta \left(\Gamma_{\nu}(\theta,\varphi)\right)
\end{equation}
using the theory of Okounkov bodies of almost semigroups developed in \cref{subsec:Okobalmosg}. Moreover, we have
\begin{equation}\label{eq:Okov}
    \vol \Delta_{\nu}(\theta,\varphi)=\frac{1}{n!}\int_X \theta_{P[\varphi]_{\mathcal{I}}}^n.
\end{equation}

\subsubsection{The case of analytic singularities}
Assume that $\varphi$ has analytic singularities and $\theta_{\varphi}$ is a K\"ahler current.

For any rational $\epsilon\geq 0$, we define
	\begin{equation}\label{eq:Weps}
		W^{\epsilon}_k=W^{\epsilon}_k(\theta,\varphi)\coloneqq \left\{ s\in \mathrm{H}^0(X,L^k): |s|_{h^k}^2\mathrm{e}^{-k(1-\epsilon)\varphi} \text{ is bounded}\right\}.
	\end{equation}
    Then $W^{\epsilon}\coloneqq \bigoplus_{k=0}^{\infty} W^{\epsilon}_k$ has the property that
    \begin{equation}\label{eq:Weps1}
        \Gamma_{\nu}(W^{\epsilon})\coloneqq \left\{ (\nu(s),k):k\in \mathbb{N}, s\in W^{\epsilon,\times}_k \right\}\in \mathcal{S}'(\Delta_{\nu}(L)).
    \end{equation}
    To see this, we may assume that $\varphi$ has analytic singularities along a $\mathbb{Q}$-divisor $E$, then \eqref{eq:Weps1} follows from the fact that $L-(1-\epsilon)E$ is big proved in \cite[Lemma~2.4]{Xia20}, c.f. \cite[Lemma~2.2]{LM09}.

For any $\epsilon\in \mathbb{Q}_{>0}$, we have that
	\begin{equation}\label{eq:OTc}
		W_k^0\subseteq W_k\subseteq W_k^{\epsilon}
	\end{equation}
	for $k$ large enough depending on $\epsilon$. The first inclusion is of course trivial. The second inclusion is widely known among experts. A detailed proof can be found in \cite[Remark~2.9]{DX21}.

	Let $\pi\colon Y\rightarrow X$ be a resolution such that $\pi^*\varphi$ has analytic singularities along a normal crossing $\mathbb{Q}$-divisor $E$.
	Then we have a natural identification for sufficiently divisible $k$,
	\[
		W_k^{\epsilon}\cong \mathrm{H}^0(Y,\pi^*L^k\otimes \mathcal{O}_Y(-(1-\epsilon) k E)).
	\]
	On the other hand,
	\[
		W^0_k\cong  \mathrm{H}^0(Y,\pi^*L^k\otimes \mathcal{O}_Y(-k E))\subseteq \mathrm{H}^0(Y,\pi^*L^k).
	\]
	We compute the volumes,
	\begin{equation}\label{eq:volDeltas}
		\vol \Gamma_{\nu}(W^{\epsilon})=\frac{1}{n!}\int_X \theta_{(1-\epsilon)\varphi}^n,\quad \vol \Gamma_{\nu}(W^0)=\frac{1}{n!}\int_X \theta_{\varphi}^n.
	\end{equation}
    It follows that $\Gamma_{\nu}(W^{\epsilon})\to \Gamma_{\nu}(W^0)$ and $\Gamma_{\nu}(\theta,\varphi)$ is equivalent to $\Gamma_{\nu}(W^0)$. In particular, $\Gamma_{\nu}(\theta,\varphi)\in \overline{\mathcal{S}'(\Delta_{\nu}(L))}_{>0}$, \eqref{eq:Deltalbdef} makes sense and \eqref{eq:Okov} holds.

\begin{remark}\label{rmk:DeltaanaW0}
	It follows from the proof that if $W^0(\theta,\varphi)$ is defined as in \eqref{eq:Weps} and \eqref{eq:Weps1}:
	\[
		W^0_k(\theta,\varphi)\coloneqq \left\{ s\in \mathrm{H}^0(X,L^k): |s|_{h^k}^2\mathrm{e}^{-k\varphi} \text{ is bounded}\right\},
	\]
	then
	\begin{equation}\label{eq:DeltaanaW0}
		\Delta(\Gamma_{\nu} (W^0(\theta,\varphi)))=\Delta_{\nu}(\theta,\varphi).
	\end{equation}
	If we assume furthermore that $\pi^*\varphi$ has analytic singularity along some normal crossing $\mathbb{Q}$-divisor $E$ on $Y$, then $\Delta_{\nu}(\theta,\varphi)$ is just the translation of $\Delta_{\nu}(\pi^*L-E)$ by $\nu(E)$.
\end{remark}

\subsubsection{The case of Kähler currents}

Now assume that $\theta_{\varphi}$ is Kähler current. Let $\varphi^j\in \PSH(X,\theta)$ be a quasi-equisingular approximation of $\varphi$. Then $\varphi^j\xrightarrow{d_S}P[\varphi]_{\mathcal{I}}$ by \cite[Proposition~3.3]{DX21}.

In this case, we claim that
\begin{equation}\label{eq:WtoWclaim}
	\Gamma_{\nu}(\theta,\varphi^j)\to \Gamma_{\nu}(\theta,\varphi).
\end{equation}
In fact, by \cref{thm:DXmain}, we have
\[
	\begin{aligned}
		d(\Gamma_{\nu}(\theta,\varphi^j),\Gamma_{\nu}(\theta,\varphi)) =& \varlimsup_{k\to\infty} k^{-n} \left( h^0(X,L^k\otimes \mathcal{I}(k\varphi^j))-h^0(X,L^k\otimes \mathcal{I}(k\varphi))\right)\\
  = & \lim_{k\to\infty} k^{-n}h^0(X,L^k\otimes \mathcal{I}(k\varphi^j))-\lim_{k\to\infty} k^{-n}h^0(X,L^k\otimes \mathcal{I}(k\varphi)) \\
		 =&\frac{1}{n!}\int_X \theta_{\varphi^j}^n-\frac{1}{n!}\int_X \theta_{P[\varphi]_{\mathcal{I}}}^n.
	\end{aligned}
\]
Letting $j\to\infty$, we conclude \eqref{eq:WtoWclaim} by \cref{thm:dsmixedmass}. 

Thus, $\Gamma_{\nu}(\theta,\varphi)\in \overline{\mathcal{S}'(\Delta_{\nu}(L))}_{>0}$ and \eqref{eq:Deltalbdef} makes sense.
By \cref{thm:Okocont}, we find that
\[
	\Delta_{\nu}(\theta,\varphi)=\bigcap_{j=0}^{\infty}\Delta_{\nu}(\theta,\varphi^j).
\]
In particular, \eqref{eq:Okov} holds.

\subsubsection{General case}

Now we consider general $\varphi$ with the assumption that $\int_X \theta_{P[\varphi]_{\mathcal{I}}}^n>0$.
We may replace $\varphi$ with $P[\varphi]_{\mathcal{I}}$ and then assume that the non-pluripolar mass of $\varphi$ is positive.
Take a potential $\psi\in \PSH(X,\theta)$ such that $\psi \leq \varphi$ and $\theta_{\psi}$ is a Kähler current. The existence of $\psi$ is proved in \cite[Proposition~3.6]{DX21}.
For each $\epsilon\in \mathbb{Q}\cap (0,1]$, let $\varphi_{\epsilon}=(1-\epsilon)\varphi+\epsilon\psi$.
Then we have $W(\theta,\varphi_{\epsilon})\subseteq W(\theta,\varphi)$.
By \eqref{eq:Okov},
\[
	\vol \Delta_{\nu}(\theta,\varphi_{\epsilon})=\frac{1}{n!}\int_X \theta^n_{P[\varphi_{\epsilon}]_{\mathcal{I}}}.
\]
We claim that 
\[
	\Gamma_{\nu}(\theta,\varphi_{\epsilon})\to \Gamma_{\nu}(\theta,\varphi).
\]
In fact, this follows from the simple computation:
\[
\begin{aligned}
    d\left(\Gamma_{\nu}(\theta,\varphi_{\epsilon}),\Gamma_{\nu}(\theta,\varphi)\right)=& \varlimsup_{k\to\infty} k^{-n} \left( h^0(X,L^k\otimes \mathcal{I}(k\varphi))- h^0(X,L^k\otimes \mathcal{I}(k\varphi_{\epsilon})) \right)\\
    =& \lim_{k\to\infty}k^{-n}h^0(X,L^k\otimes \mathcal{I}(k\varphi))-\lim_{k\to\infty}k^{-n}h^0(X,L^k\otimes \mathcal{I}(k\varphi_{\epsilon}))\\
    =& \frac{1}{n!}\int_X \theta_{\varphi}^n- \frac{1}{n!}\int_X \theta_{P[\varphi_{\epsilon}]_{\mathcal{I}}}^n.
\end{aligned}
\]
By \cite[Proposition~2.7]{DX21}, as $\epsilon$ decreases to $0$, $P[\varphi_{\epsilon}]_{\mathcal{I}}$ increases to $P[\varphi]_{\mathcal{I}}=\varphi$ a.e., which implies the $d_S$-convergence by \cite[Lemma~4.1]{DDNLmetric}. 
Therefore, the right-hand side of the above equation converges to $0$ by \cref{thm:dsmixedmass}. Our claim is proved. It follows that $\Gamma_{\nu}(\theta,\varphi)\in \overline{\mathcal{S}'(\Delta_{\nu}(L))}_{>0}$ and \eqref{eq:Deltalbdef} makes sense. 
By \cref{thm:Okocont},
\[
	\Delta_{\nu}(\theta,\varphi)=\overline{\bigcup_{\epsilon>0}\Delta_{\nu}(\theta,\varphi_{\epsilon})}.
\]
It remains to verify \eqref{eq:Okov}:
\[
	\vol \Delta_{\nu}(\theta,\varphi)=\frac{1}{n!}\lim_{\epsilon\to 0+}\int_X \theta_{P[\varphi_{\epsilon}]_{\mathcal{I}}}^n=\frac{1}{n!}\int_X \theta_{P[\varphi]_{\mathcal{I}}}^n.
\]

\begin{definition}
	Assume that $\varphi\in \PSH(X,\theta)$, $\int_X \theta_{P[\varphi]_{\mathcal{I}}}^n>0$.  We call $\Delta_{\nu}(\theta,\varphi)$ the \emph{partial Okounkov body} of $(L,\phi)$ or of $(\theta,\varphi)$ with respect to $\nu$. When $\nu$ is induced by an admissible flag $(Y_{\bullet})$ on $X$ (see \cref{def:admfl}), we also say that $\Delta_{\nu}(\theta,\varphi)$ the \emph{partial Okounkov body} of $(L,\phi)$ or of $(\theta,\varphi)$ with respect to $(Y_{\bullet})$. In this case, we also write $\Delta_{Y_{\bullet}}$ instead of $\Delta_{\nu}$.
\end{definition}
We use interchangeably the notations $\Delta_{\nu}(\theta,\varphi)$ and $\Delta_{\nu}(L,\phi)$. When there is no risk of confusion, we write $\Delta$ instead of $\Delta_{\nu}$ or $\Delta_{Y_{\bullet}}$.

\begin{remark}
	We have assumed $X$ to be smooth only for simplicity. All of our constructions work equally well when $X$ is normal or merely unibranch, based on the pluripotential theory in these settings developed in \cite{XiaMabuchi}.
\end{remark}

\begin{remark}
    In the transcendental setting, a theory of Okounkov bodies is recently established in \cite{DRWN+} based on the work of \cite{Deng17}. The proof of the volume identity of transcendental Okounkov bodies relies on the technique of partial Okounkov bodies developed in this paper. The transcendental analogue of the partial Okounkov bodies is constructed in a forthcoming joint paper with T.Darvas.
\end{remark}

\subsection{Basic properties of partial Okounkov bodies}
We first show that $\Delta(\theta,\varphi)$ does not depend on the explicit choices of $L$, $h$ and $\varphi$, it just depends on $\ddc\phi$.

\begin{lemma}\label{lma:indepL}
	Let $L'$ be another big line bundle on $X$. Let $h'$ be a smooth Hermitian metric on $L'$ with $c_1(L,h)=c_1(L',h')$.
	Then $\Delta(\theta,\varphi)$ defined with respect to $(L,h)$ is the same as the one defined with respect to $(L',h')$.
\end{lemma}
\begin{proof}
	From our construction, we may assume that $\theta_{\varphi}$ is a Kähler current and $\varphi$ has analytic singularities.
	After taking a birational resolution, it suffices to deal with the case where $\varphi$ has analytic singularities along normal crossing $\mathbb{Q}$-divisors $E$.
	By rescaling, we may also assume that $E$ is a divisor. By \cref{rmk:DeltaanaW0}, we further reduce to the case without the singular potential $\phi$.

	In this case, the assertion is proved in \cite[Proposition~4.1]{LM09}. 
\end{proof}
\begin{lemma}\label{lma:indepvarphi}
	Let $h'$ be another smooth Hermitian metric on $L$. Set $\theta'=c_1(L,h')$. Write $\ddc f=\theta-\theta'$.
	Let $\varphi'=\varphi+f\in \PSH(X,\theta')$. Then
	\begin{equation}\label{eq:DeltaDelta1}
		\Delta(\theta,\varphi)=\Delta(\theta',\varphi').
	\end{equation}
\end{lemma}
\begin{proof}
	This is obvious as $W(\theta,\varphi)=W(\theta',\varphi')$.
\end{proof}

\begin{corollary}\label{cor:Okocurrent}
	The partial Okounkov body $\Delta(L,\phi)$ depends only on $\ddc\phi$, not on the explicit choices of $L,\phi,h$.
\end{corollary}
Thanks to this result, given a closed positive $(1,1)$-current $T\in c_1(L)$ on $X$ with $\int_X T^n>0$, we can define $\Delta(T)$ as $\Delta(\theta,\varphi)$ if $T=\theta+\ddc\varphi$ for some $\varphi\in \PSH(X,\theta)$.
\begin{proof}
	This is a direct consequence of \cref{lma:indepL} and \cref{lma:indepvarphi}.
\end{proof}

Let $\PSH(X,\theta)_{>0}$ denote the subset of $\PSH(X,\theta)$ consisting of potentials $\varphi$ such that $\int_X \theta_{\varphi}^n>0$.
\begin{proposition}\label{prop:IcompimplyDeltacomp}
	Let $\varphi,\psi\in \PSH(X,\theta)_{>0}$. Assume that $\varphi\preceq_{\mathcal{I}} \psi$, then
	\begin{equation}\label{eq:Deltacomp}
		\Delta(\theta,\varphi)  \subseteq \Delta(\theta,\psi).
	\end{equation}
	In particular, as by definition, $\Delta(\theta,V_{\theta})=\Delta(L)$, we have
	\[
		\Delta(\theta,\varphi)\subseteq \Delta(L).
	\]
\end{proposition}
\begin{proof}
	This follows from \cref{cor:Okocomp}.
\end{proof}

\begin{theorem}\label{thm:Okoucont}
	The Okounkov body map
	\[
		\Delta(\theta,\bullet):\left( \PSH(X,\theta)_{>0},d_S\right)\rightarrow \left(\mathcal{K}_n,d_n\right)
	\]
	is continuous.
\end{theorem}
\begin{remark}
    On the other hand, it is of interest to understand the dependence of $\Delta(\theta,\bullet)$ on $\nu$ as well. 
    For some preliminary results and anticipations in the usual Okounkov body setting, see \cite{AI22}. In particular, see \cite[Conjecture~10.1]{AI22} for a concrete continuity conjecture.
\end{remark}
\begin{proof}Let $\varphi_j\to \varphi$ be a $d_S$-convergent sequence in $\PSH(X,\theta)_{>0}$.
	We want to show that
	\begin{equation}\label{eq:Deltavjv}
		\Delta(\theta,\varphi_j)\xrightarrow{d_n}\Delta(\theta,\varphi).
	\end{equation}
	By \cref{prop:IcompimplyDeltacomp} and \cite[Theorem~3.3]{DDNLmetric}, we may assume that all $\varphi_j$'s and $\varphi$ are model potentials.
	By \cref{thm:Bst} and \cite[Theorem~5.6]{DDNLmetric}, we may assume that $\varphi_j$ is either decreasing or increasing. By \cref{thm:contPI}, we may further assume that the $\varphi_j$'s are $\mathcal{I}$-model.
 In both cases, we claim that $\Gamma_{\nu}(\theta,\varphi_j)\rightarrow \Gamma_{\nu}(\theta,\varphi)$. In fact, we can compute their distance as follows
 \[
 \begin{aligned}
 d\left(\Gamma_{\nu}(\theta,\varphi_j), \Gamma_{\nu}(\theta,\varphi)\right) = &\varlimsup_{k\to\infty} k^{-n}\left| h^0(X,L^k\otimes \mathcal{I}(k\varphi_j))-h^0(X,L^k\otimes \mathcal{I}(k\varphi)) \right|\\
 =& \frac{1}{n!} \left| \int_X \theta_{\varphi_j}^n-\int_X \theta_{\varphi}^n\right|,
 \end{aligned}
 \]
 where we applied \cref{thm:DXmain} at the last step. Then \cref{thm:dsmixedmass} implies our claim.
 Hence, \eqref{eq:Deltavjv} follows from \cref{thm:Okocont}.
\end{proof}

Although $W(\theta,\varphi)$ and $\Gamma_{\nu}(\theta,\varphi)$ are not birationally invariant, we could still show that the Okounkov body is.
\begin{proposition}\label{prop:birinvO}
	Let $\pi \colon Y\rightarrow X$ be a birational resolution. Let $(L,\phi)$ be a Hermitian big line bundle on $X$ with positive volume, then
	\[
		\Delta(\pi^*L,\pi^*\phi)=\Delta(L,\phi).
	\]
\end{proposition}
Here we are using the same valuation $\nu$ on the function field $\mathbb{C}(Y)=\mathbb{C}(X)$ of $Y$.
\begin{proof}
    By \cref{def:Icomp}(3), $P_{\theta}[\bullet]_{\mathcal{I}}$ commutes with birational pullbacks, we may assume that $\varphi$ is $\mathcal{I}$-model.
   By \cite[Theorem~3.8]{DX21}, we can find a sequence $\varphi^j\in \PSH(X,\theta)$ with analytic singularities such that $\varphi^j\xrightarrow{d_S}\varphi$. It follows from \eqref{eq:defds} that $\pi^*\varphi^j\xrightarrow{d_S}\pi^*\varphi$.
    By \cref{thm:Okoucont}, we may then reduce to the case where $\varphi$ has analytic singularities. In this case, up to replacing $Y$ by a further sequences of blowups, we may assume that $\pi^*\varphi$ has analytic singularities along a normal crossing $\mathbb{Q}$-divisor $D$. It suffices to apply \cref{rmk:DeltaanaW0}.
\end{proof}

Next we prove the Brunn--Minkowski inequality.
\begin{proposition}
	Let $(L,\phi)$, $(L',\phi')$ be Hermitian big line bundles on $X$ of positive volumes. Then
	\[
		(\vol\Delta(L+L',\phi+\phi'))^{1/n}\geq (\vol \Delta(L,\phi))^{1/n}+(\vol \Delta(L',\phi'))^{1/n}.
	\]
\end{proposition}
\begin{proof}
	This follows from \cref{cor:BMineq}.
\end{proof}

\begin{proposition}\label{prop:suba}
	Let $(L',\phi')$ be another Hermitian big line bundle on $X$ with positive volume.  Then
	\[
		\Delta(L,\phi)+\Delta(L',\phi')\subseteq\Delta(L\otimes L',\phi\otimes \phi').
	\]
\end{proposition}
\begin{proof}
	Take a smooth metric $h'$ on $L'$, let $\theta'=c_1(L',h')$. We identify $\phi'$ with $\varphi'\in \PSH(X,\theta')$. Then we need to show
	\begin{equation}\label{eq:suba}
		\Delta(\theta,\varphi)+\Delta(\theta',\varphi')\subseteq \Delta(\theta+\theta',\varphi+\varphi').
	\end{equation}
	By \cite[Theorem~3.8]{DX21}, we can find $\varphi^j\in \PSH(X,\theta)$, $\varphi'^j\in \PSH(X,\theta')$ such that
	\begin{enumerate}
		\item $\varphi^j$ and $\varphi'^j$ both have analytic singularities and have positive masses.
		\item $\varphi^j\xrightarrow{d_S}\varphi$, $\varphi'^j\xrightarrow{d_S}\varphi'$.
	\end{enumerate}
	Then $\varphi^j+\varphi'^j\in \PSH(X,\theta+\theta')$ and $\varphi^j+\varphi'^j\xrightarrow{d_S} \varphi+\varphi'$ by \cref{cor:dssum}. Thus, by \cref{thm:Okoucont}, we may assume that $\varphi$ and $\psi$ both have analytic singularities. Taking a birational resolution, we may further assume that they have analytic singularities along some normal crossing divisors. By \cref{rmk:DeltaanaW0}, we reduce to the case without singularities, in which case the result is well-known, see \cite[The proof of Corollary~4.12]{LM09} for example.
\end{proof}

\begin{theorem}\label{thm:concOko}
	Let $\varphi,\psi\in \PSH(X,\theta)_{>0}$. Then for any $t\in (0,1)$,
	\begin{equation}\label{eq:Deltaconcave}
		\Delta(\theta,t\varphi+(1-t)\psi)\supseteq t\Delta(\theta,\varphi)+(1-t)\Delta(\theta,\psi).
	\end{equation}
\end{theorem}
\begin{proof}

	We may assume that $t$ is rational as a consequence of \cref{thm:Okoucont}. Similarly, by \cite[Theorem~3.8]{DX21},
	we could reduce to the case where both $\varphi$ and $\psi$ have analytic singularities. Taking a resolution, we may assume that $\varphi$ (resp. $\psi$) has analytic singularities along a normal crossing $\mathbb{Q}$-divisor $E$ (resp. $E'$).
	In this case, let $N>0$ be an integer such that $Nt$ is an integer. Then for any $s\in W_k^0(\theta,\varphi)$, $r\in W_k^0(\theta,\psi)$, we have
	\[
		\left( s^{t}r^{1-t} \right)^N\in W^0_{Nk}(\theta,t\varphi+(1-t)\psi).
	\]
	By \cref{thm:HausOkoun}, \eqref{eq:Deltaconcave} follows.
\end{proof}

\begin{proposition}\label{prop:res}
	For any integer $a>0$,
	\[
		\Delta(a\theta,a\varphi)=a\Delta(\theta,\varphi).
	\]
\end{proposition}

\begin{proof}
	By \cref{thm:Okoucont}, it suffices to treat the case where $\varphi$ has analytic singularities. Taking a birational resolution, we may assume that $\varphi$ has analytic singularities along a normal crossing $\mathbb{Q}$-divisor $E$. By \cref{rmk:DeltaanaW0}, we reduce to the case without the singularity $\varphi$, which is already proved in \cite{LM09}.
\end{proof}
In particular, if $T$ is a closed positive $(1,1)$-current on $X$ with $\int_X T^n>0$ and such that the cohomology class of $T$ lies in the Néron--Severi group with rational coefficients, then we can define $\Delta(T)$ as $a^{-1}\Delta(aT)$ for a sufficiently divisible positive integer $a$. 

We also need the following perturbation. Let $A$ be an ample line bundle on $X$. Fix a smooth Hermitian metric $h_A$ on $A$ such that $\omega\coloneqq c_1(A,h_A)$ is a K\"ahler form on $X$. Then for any $\delta\in \mathbb{Q}_{>0}$, we can define
\[
	\Delta(\theta+\delta\omega,\varphi)\coloneqq \Delta(\theta+\delta\omega+\ddc\varphi)= C^{-1}\Delta(C\theta+C\delta\omega,C\varphi)  ,
\]
where $C\in \mathbb{N}_{>0}$ is any  integer so that $C\delta\in \mathbb{N}$.

\begin{proposition}\label{prop:Deltapert}
	Under the above assumptions, as $\delta\in \mathbb{Q}_{>0}$ decreases to $0$, $\Delta(\theta+\delta\omega,\varphi)$ is decreasing under inclusion with Hausdorff limit $\Delta(\theta,\varphi)$.
\end{proposition}
\begin{proof}
	Let $0\leq \delta<\delta'$ be two rational numbers. Take $C\in \mathbb{N}_{>0}$ divisible enough, so that $C\delta$ and $C\delta'$ are both integers. Then by \cref{prop:suba},
	\[
		\Delta(C\theta+C\delta\omega,C\varphi)\subseteq    \Delta(C\theta+C\delta'\omega,C\varphi) .
	\]
	It follows that
	\[
		\Delta(\theta+\delta\omega,\varphi)\subseteq \Delta(\theta+\delta'\omega,\varphi).
	\]
	On the other hand,
	\[
		\vol   \Delta(\theta+\delta\omega,\varphi)=\frac{1}{n!}\int_X (\theta+\delta\omega)_{P^{\theta+\delta\omega}[\varphi]_{\mathcal{I}}}^n=\frac{1}{n!}\int_X (\theta+\delta\omega)_{P^{\theta}[\varphi]_{\mathcal{I}}}^n,
	\]
    where we applied \cref{cor:I-modelpert}.
	As $\delta\to 0+$, the right-hand side converges to
	\[
		\vol   \Delta(\theta,\varphi)=\frac{1}{n!}\int_X \theta_{P^{\theta}[\varphi]_{\mathcal{I}}}^n.
	\]
	It follows that
	\[
		\Delta(\theta,\varphi)=\bigcap_{\delta\in \mathbb{Q}_{>0}}\Delta(\theta+\delta\omega,\varphi).
	\]
\end{proof}

\subsection{The Hausdorff convergence property of partial Okounkov bodies}
For each $k\in \mathbb{Z}_{>0}$, we introduce
\[
\Delta_k(\theta,\varphi)\coloneqq \Conv\left\{ k^{-1}\nu(f): f\in \mathrm{H}^0(X,L^k\otimes \mathcal{I}(k\varphi))^{\times} \right\}\subseteq \mathbb{R}^n.
\]
Here $\Conv$ denotes the convex hull. The convex hull is a polytope if it is non-empty by \cite[Lemma~1.4]{LM09}. For large enough $\Delta_k(\theta,\varphi)$ is non-empty thanks to \cref{thm:DXmain}.

For later use, we introduce a twisted version as well. If $T$ is a holomorphic line bundle on $X$, we introduce
\[
\Delta_{k,T}(\theta,\varphi)\coloneqq \Conv\left\{ k^{-1}\nu(f): f\in \mathrm{H}^0(X,T\otimes L^k\otimes \mathcal{I}(k\varphi))^{\times} \right\}\subseteq \mathbb{R}^n.
\]
We also write
\[
    \Delta_{k,T}(L)\coloneqq \Conv\left\{ k^{-1}\nu(f): f\in \mathrm{H}^0(X,T\otimes L^k)^{\times} \right\}\subseteq \mathbb{R}^n
\]
and
\[
    \Delta_{k}(L)\coloneqq \Conv\left\{ k^{-1}\nu(f): f\in \mathrm{H}^0(X,L^k)^{\times} \right\}\subseteq \mathbb{R}^n.
\]

We write $\mathcal{I}_{\infty}(\varphi)=\mathcal{I}_{\infty}(\phi)$ for the ideal sheaf on $X$ locally consisting of holomorphic functions $f$ such that $|f|_{\phi}$ is locally bounded.

The main result is the following:
\begin{theorem}[Hausdorff convergence property]\label{thm:HCP}
    Let $T$ be a holomorphic line bundle on $X$.
    As $k\to\infty$, we have $\Delta_{k,T}(\theta,\varphi)\xrightarrow{d_n} \Delta(\theta,\varphi)$.
\end{theorem}
Although we are only interested in the untwisted case, the proof given below requires twisted case.

We first extend \cref{thm:HausOkoun} to the twisted case.
\begin{proposition}\label{prop-Deltaconvtwisted}
    For any holomorphic line bundle $T$ on $X$, as $k\to\infty$
    \[
        \Delta_{k,T}(L)\xrightarrow{d_n} \Delta(L).
    \]
\end{proposition}

\begin{proof}
    As $L$ is big, we can take $k_0\in \mathbb{Z}_{>0}$ so that
    \begin{enumerate}
        \item $T^{-1}\otimes L^{k_0}$ admits a non-zero global holomorphic section $s_0$;
        \item $T\otimes L^{k_0}$ admits a non-zero global holomorphic section $s_1$.
    \end{enumerate}
    Then for $k\in \mathbb{Z}_{>k_0}$, we have injective linear maps
    \[
        \mathrm{H}^0(X,L^{k-k_0})\xrightarrow{\times s_1}\mathrm{H}^0(X,T\otimes L^{k})\xrightarrow{\times s_0}\mathrm{H}^0(X, L^{k+k_0}).
    \]
    It follows that
    \[
        (k-k_0)\Delta_{k-k_0}(L)+\nu(s_1)\subseteq k\Delta_{k,T}(L)\subseteq (k+k_0)\Delta_{k+k_0}(L)-\nu(s_0).
    \]
    By \cref{thm:HausOkoun}, we conclude.
\end{proof}

\begin{lemma}\label{lma-twistedHcp}
    Let $T$ be a holomorphic line bundle on $X$.
    Assume that $\varphi$ has analytic singularities and $\theta_{\varphi}$ is a K\"ahler current, then as $k\to\infty$,
    \[
        \Delta_{k,T}(\theta,\varphi)\xrightarrow{d_n} \Delta(\theta,\varphi).
    \]
\end{lemma}
\begin{proof}
    Up to replacing $X$ by a birational model and twisting $T$ accordingly, we may assume that $\varphi$ has analytic singularities along a normal crossing $\mathbb{Q}$-divisor $D$, c.f. \cref{prop:birinvO}. 
    Take $\epsilon\in (0,1)\cap \mathbb{Q}$.
    In this case, as in \eqref{eq:OTc}, for large enough $k\in \mathbb{Z}_{>0}$ we have
    \[
        \mathrm{H}^0(X,T\otimes L^k\otimes \mathcal{I}_{\infty}(k\varphi)) \subseteq \mathrm{H}^0(X,T\otimes L^k\otimes \mathcal{I}(k\varphi))
        \subseteq \mathrm{H}^0(X,T\otimes L^k\otimes \mathcal{I}_{\infty}(k(1-\epsilon)\varphi)).
    \]
    Take an integer $N\in \mathbb{Z}_{>0}$ so that $ND$ is a divisor and $N\epsilon$ is an integer.

    Let $\Delta'$ be the limit of a subsequence of $(\Delta_{k,T}(\theta,\varphi))_k$, say the sequence defined by the indices $k_1,k_2,\ldots$. We want to show that $\Delta'=\Delta(\theta,\varphi)$.

    There exists $t\in \{0,1,\ldots,N-1\}$ such that $k_i\equiv t$ modulo $N$ for infinitely many $i$, up to replacing $k_i$ by a subsequence, we may assume that $k_i\equiv t$ modulo $N$ for all $i$. Write $k_i=N g_i+t$. 
    Then
    \[
    \begin{split}
        \mathrm{H}^0(X,T\otimes L^{-N+t} \otimes L^{N(g_i+1)}\otimes \mathcal{I}_{\infty}(N(g_i+1)\varphi)) \subseteq \mathrm{H}^0(X,T\otimes L^{k_i}\otimes \mathcal{I}(k_i\varphi)) \\
        \subseteq \mathrm{H}^0(X,T\otimes L^t \otimes L^{Ng_i}\otimes \mathcal{I}_{\infty}(g_i N(1-\epsilon)\varphi)).
    \end{split}
    \]
    So
    \[
    \begin{split}
    (g_i+1) \Delta_{g_i+1,T\otimes L^{-N+t}}(NL-ND)+N(g_i+1)\nu(D)\subseteq (Ng_i+t)\Delta_{k,T}(\theta,\varphi)\\
    \subseteq 
    g_i\Delta_{g_i,T\otimes L^t}(NL-N(1-\epsilon)D)+Ng_i(1-\epsilon)\nu(D).
    \end{split}
    \]
    Letting $i\to\infty$, by \cref{prop-Deltaconvtwisted},
    \[
        \Delta(L-D)+\nu(D)\subseteq \Delta'\subseteq \Delta(L-(1-\epsilon)D)+(1-\epsilon)\nu(D).
    \]
    Letting $\epsilon\to 0+$, we find that
    \[
        \Delta(L-D)+\nu(D)=\Delta'.
    \]
    It follows from \cref{thm:Bst} that 
    \[
        \Delta_{k,T}(\theta,\varphi)\xrightarrow{d_n} \Delta(L-D)+\nu(D)=\Delta(\theta,\varphi)
    \]
    as $k\to\infty$.
\end{proof}

\begin{lemma}\label{lma-Hausconvbetato0}
    Assume that $\theta_{\varphi}$ is a K\"ahler current, then as $\mathbb{Q}\ni\beta\to 0+$, we have
    \[
        \Delta((1-\beta)\theta,\varphi)\to \Delta(\theta,\varphi). 
    \]
\end{lemma}

\begin{proof}
    
    By \cref{prop:suba}, we have
    \[
        \Delta((1-\beta)\theta,\varphi)+\beta\Delta(L)\subseteq \Delta(\theta,\varphi).
    \]
    In particular, if $\Delta'$ is a limit of a subsequence of $(\Delta((1-\beta)\theta,\varphi))_{\beta}$, then 
    \[
        \Delta'\subseteq \Delta(\theta,\varphi).
    \]
    But 
    \[
        \vol \Delta'=\lim_{\beta\to 0+}\Delta((1-\beta)\theta,\varphi)=\lim_{\beta\to 0+}\int_X ((1-\beta)\theta+\ddc P^{(1-\beta)\theta}[\varphi]_{\mathcal{I}})^n=\int_X (\theta+\ddc P^{\theta}[\varphi]_{\mathcal{I}})^n,
    \]
    where the last step follows easily from \cite[Theorem~0.6]{Xia22}. It follows that $\Delta'=\Delta(\theta,\varphi)$. We conclude by \cref{thm:Bst}.
\end{proof}

\begin{proof}[Proof of \cref{thm:HCP}]
Fix a K\"ahler form $\omega\geq \theta$ on $X$.

\textbf{Step~1}. We first handle the case where $\theta_{\varphi}$ is a Kähler current, say $\theta_{\varphi}\geq \beta_0 \omega$ for some $\beta_0\in (0,1)$.

Take a decreasing quasi-equisingular approximation $\varphi_j$ of $\varphi$. 
Up to replacing $\beta_0$ by $\beta_0/2$, we may assume that $\theta_{\varphi_j}\geq \beta_0 \omega$ for all $j\geq 1$.

Let $\Delta'$ be a limit of a subsequence of $(\Delta_{k,T}(\theta,\varphi))_k$. Let us say the indices of the subsequence are $k_1<k_2<\cdots$. By \cref{thm:Bst}, it suffices to show that $\Delta'=\Delta(\theta,\varphi)$.

As $[\varphi] \preceq [\varphi_j]$ for each $j\geq 1$, we have $\Delta'\subseteq \Delta(\theta,\varphi_j)$ by \cref{lma-twistedHcp}. Letting $j\to \infty$, we find
\[
\Delta'\subseteq \Delta(\theta,\varphi).
\]
In particular, it suffices to prove that
\[
    \vol \Delta'\geq \vol \Delta(\theta,\varphi).
\]
Take $\beta\in (0,\beta_0)\cap \mathbb{Q}$. Write $\beta=p/q$ with $p,q\in \mathbb{Z}_{>0}$. Observe that for any $j\geq 1$,
\[
    \theta_{\varphi_j}\geq \beta\omega\geq \beta\theta.
\]
Namely, $\varphi_j\in \PSH(X,(1-\beta)\theta)$. Similarly, $\varphi\in \PSH(X,(1-\beta)\theta)$.
By \cref{lma-Hausconvbetato0}, it suffices to argue that
\begin{equation}\label{eq:volDeltatoprove}
\vol \Delta'\geq \vol \Delta((1-\beta)\theta,\varphi).
\end{equation}
For this purpose, we are free to replace $k_i$'s by a subsequence, so we may assume that $k_i\equiv a$ modulo $q$ for all $i\geq 1$, where $a\in \{0,1,\ldots,q-1\}$. We write $k_i=g_iq+a$.
Observe that for each $i\geq 1$,
\[
    \mathrm{H}^0(X,T\otimes L^{k_i}\otimes \mathcal{I}(k_i\varphi))\supseteq \mathrm{H}^0(X,T\otimes L^{-q+a}\otimes L^{g_iq+q}\otimes \mathcal{I}((g_iq+q)\varphi)).
\]
Up to replacing $T$ by $T\otimes L^{-q+a}$, we may therefore assume that $a=0$.

By \cite[Lemma~4.2]{DX21}, we can find $k'\in \mathbb{Z}_{>0}$ such that for all $k\geq k'$, there is $v_{\beta,k}\in \PSH(X,\theta)$ satisfying
\begin{enumerate}
    \item $P[\varphi]_{\mathcal{I}}\geq (1-\beta)\varphi_k+\beta v_{\beta,k}$;
    \item $v_{\beta,k}$ has positive mass. 
\end{enumerate}

Fix $k\geq k'$. It suffices to show that 
\begin{equation}\label{eq:DeltatransinDeltaprime}
    \Delta((1-\beta)\theta,\varphi_{k})+v'\subseteq \Delta'
\end{equation}
for some $v'\in \mathbb{R}^n$. In fact, if this is true, we have
\[
\vol \Delta'\geq \vol \Delta((1-\beta)\theta,\varphi_{k}).
\]
Letting $k\to\infty$ and applying \cref{thm:Okoucont}, we conclude \eqref{eq:volDeltatoprove}.

It remains to prove \eqref{eq:DeltatransinDeltaprime}. We will fix $k\geq k'$.
Let $\pi:Y\rightarrow X$ be a log resolution of the singularities of $\varphi_k$. 
By the proof of \cite[Proposition~4.3]{DX21}, there is $j_0=j_0(\beta,k)\in \mathbb{Z}_{>0}$ such that for any $j\geq j_0$, we can find a non-zero section $s_j\in \mathrm{H}^0(Y,\pi^* L^{pj}\otimes \mathcal{I}(jp\pi^*v_{\beta,k}))$ such that we get an injective linear map
    \[
        \mathrm{H}^0(Y,\pi^*T\otimes K_{Y/X}\otimes \pi^*L^{(q-p)j}\otimes \mathcal{I}(jq \pi^*\varphi_k))\xrightarrow{\times s_j} \mathrm{H}^0(X,T\otimes L^{jq}\otimes \mathcal{I}(jq\varphi)).
    \]
In particular, when $j=k_i$ for some $i$ large enough, we then find
\[
    \Delta_{k_i,\pi^*T\otimes K_{Y/X}}((1-\beta)q\pi^*\theta,q\pi^*\varphi_{k})+(k_i)^{-1}\nu(s_{k_i})\subseteq q\Delta_{k_i,T}(\theta,\varphi).
\]
We observe that $(k_i)^{-1}\nu(s_{k_i})$ is bounded as both convex bodies appearing in this equation are bounded when $i$ varies. 
Then by \cref{lma-twistedHcp}, there is a vector $v'\in \mathbb{R}^n$ such that
\[
    \Delta((1-\beta)\pi^*\theta,\pi^*\varphi_{k})+v'\subseteq \Delta'.
\]
By \cref{prop:birinvO}, we find \eqref{eq:DeltatransinDeltaprime}.

    \textbf{Step~2}. Next we handle the general case.

Let $\Delta'$ be the limit of a subsequence of $(\Delta_{k,T}(\theta,\varphi))_k$, say the subsequence with indices $k_1<k_2<\cdots$. 
By \cref{thm:Bst}, it suffices to prove that $\Delta'=\Delta(\theta,\varphi)$.
    
    Take $\psi\in \PSH(X,\theta)$ such that 
    \begin{enumerate}
        \item $\theta_{\psi}$ is a K\"ahler current;
        \item $\psi\leq \varphi$.
    \end{enumerate}
    The existence of $\psi$ is proved in \cite[Proposition~3.6]{DX21}.

    Then for any $\epsilon\in \mathbb{Q}\cap (0,1)$,
    \[
        \Delta_{k,T}(\theta,\varphi)\supseteq \Delta_{k,T}(\theta,(1-\epsilon)\varphi+\epsilon\psi)
    \]
    for all $k$. It follows from Step~1 that
    \[
        \Delta' \supseteq \Delta(\theta,(1-\epsilon)\varphi+\epsilon\psi).
    \]
    
    Letting $\epsilon\to 0$ and applying \cref{thm:Okoucont}, we have $\Delta' \supseteq \Delta(\theta,\varphi)$.
    It remains to establish that
    \begin{equation}\label{eq:Deltapvolumeupp}
        \vol \Delta'\leq \vol \Delta(\theta,\varphi).
    \end{equation}
    For this purpose, we are free to replace $k_1<k_2<\cdots$ by a subsequence. Fix $q>0$, we may then assume that $k_i\equiv a$ modulo $q$ for all $i\geq 1$ for some $a\in \{0,1,\ldots,q-1\}$. We write $k_i=g_iq +a$. Observe that 
    \[
        \mathrm{H}^0(X,T\otimes L^{k_i}\otimes \mathcal{I}(k_i \varphi))\subseteq \mathrm{H}^0(X, T\otimes L^a \otimes L^{g_i q}\otimes \mathcal{I}(g_i q \varphi) ).
    \]
    Up to replacing $T$ by $T\otimes L^a$, we may assume that $a=0$.
    
    Take a very ample line bundle $H$ on $X$ and fix a K\"ahler form $\omega\in c_1(H)$, take a non-zero section $s\in \mathrm{H}^0(X,H)$.
    
    We have an injective linear map
    \[
        \mathrm{H}^0(X,T\otimes L^{jq}\otimes \mathcal{I}(jq\varphi))\xrightarrow{\times s^j}\mathrm{H}^0(X,T\otimes H^j\otimes L^{jq}\otimes \mathcal{I}(jq\varphi))
    \]
    for each $j\geq 1$.
    In particular, for each $i\geq 1$,
    \[
        k_i\Delta_{k_i,T}(q\theta,q\varphi)+k_i\nu(s)\subseteq  k_i\Delta_{k_i,T}(\omega+q\theta,q\varphi).
    \]
    Letting $i\to\infty$, by Step~1, we have
    \[
        q\Delta' + \nu(s) \subseteq \Delta(\omega+q\theta,q\varphi).
    \]
    So
    \[
        \vol \Delta'\leq \vol \Delta(q^{-1}\omega+\theta,\varphi)=\int_X (q^{-1}\omega+\theta+\ddc P^{q^{-1}\omega+\theta}[\varphi]_{\mathcal{I}})^n.
    \]
    By \cref{cor:I-modelpert},
    \[
    \vol \Delta'\leq\int_X (q^{-1}\omega+\theta+\ddc P^{\theta}[\varphi]_{\mathcal{I}})^n.
    \]
    Letting $q\to\infty$, we conclude \eqref{eq:Deltapvolumeupp}.
\end{proof}

\begin{theorem}
	The Okounkov body $\Delta(L,\phi)$ is independent of the choice of a very general flag in a family of admissible flags.
\end{theorem}
\begin{proof}
 By \cref{thm:HCP}, it suffices to show that $\Delta_k(W(\theta,\varphi))$ is independent of the choice of a very general flag. For this purpose, we may assume that $k=1$.

	Let $T$ be an irreducible component of the moduli space of admissible flags. Let
	\[
		X\times T=\mathcal{Y}_0\supseteq\cdots\supseteq\mathcal{Y}_n
	\]
	be the universal flag. The Hermitian line bundle $(L,\phi)$ pulls back to $(\mathcal{L},\Phi)$ on $X\times T$.
	We denote quantities at the fiber at $t\in T$ by a sub-index $t$.

	We claim that for each $\sigma\in \mathbb{N}^{n}$, there is a proper Zariski closed set $\Sigma\subseteq T$, so that
	\[
		\dim \mathrm{H}^0(X_t,L_t\otimes \mathcal{I}(\phi_t))^{\geq\sigma}
	\]
	are constants for $t\in T\setminus\Sigma$, where $\mathrm{H}^0(X_t,L_t\otimes \mathcal{I}(\phi_t))^{\geq\sigma}$ denotes the space of sections in $\mathrm{H}^0(X_t,L_t\otimes \mathcal{I}(\phi_t))$ with valuations no less than $\sigma$.

	Let $\mathcal{L}^{\geq\sigma}$ be the coherent subsheaf of $\mathcal{L}$ introduced in \cite[Remark~1.6]{LM09}. After possibly shrinking $T$, we may guarantee that $\mathcal{L}^{\geq \sigma}\otimes \mathcal{I}(\Phi)$ is flat over $T$. By further shrinking $T$, we may guarantee that
	\[
		t\mapsto \dim \mathrm{H}^0\left(X_t,(\mathcal{L}^{\geq \sigma}\otimes \mathcal{I}(\Phi))|_{X_t}\right)
	\]
	is constant.
	Observe that
	\[
		(\mathcal{L}^{\geq \sigma}\otimes \mathcal{I}(\Phi))|_{X_t}\cong L_t^{\geq \sigma}\otimes \mathcal{I}(\phi).
	\]
	Thus, our claim follows.

	From this claim, it follows that the images of $\Gamma_k(W(L,\phi))$ are independent of the choice of a very general flag $(Y_{\bullet})$ as \cite[Proof of Theorem~5.1]{LM09}. Thus, $\Delta(W(L,\phi))$ is independent of the choice of a very general flag.
\end{proof}

\subsection{Recover Lelong numbers from partial Okounkov bodies}

\begin{lemma}\label{lma:qesana}
Suppose that $\varphi\in \PSH(X,\theta)$ such that $\theta_{\varphi}$ is a K\"ahler current. Let $\varphi^j$ be a quasi-equisingular approximation of $\varphi$. Then $\nu(\varphi^j,E)\to \nu(\varphi,E)$ for any prime divisor $E$ over $X$.
\end{lemma}
This result is essentially \cite[Lemma~2.2]{Xia20}, proved under slightly different assumptions. We reproduce the argument for the convenience of the readers.
\begin{proof}
Fix $k\in \mathbb{Z}_{>0}$, $\delta\in \mathbb{Q}_{>0}$, take $j_0>0$, so that when $j>j_0$, $\mathcal{I}((1+\delta)k\varphi^j)\subseteq \mathcal{I}(k\varphi)$.
When $j>j_0$, we get
\[
\frac{1}{k}\ord_E(\mathcal{I}(k\varphi))\leq \frac{1}{k}\ord_E(\mathcal{I}((1+\delta)k\varphi^j)).
\]
By Fekete's lemma, 
\[
\nu(\varphi^j,E)=\sup_{k\in \mathbb{Z}_{>0}}\frac{1}{k}\ord_E(\mathcal{I}(k\varphi^j)).
\]
So
\[
\frac{1}{k}\ord_E(\mathcal{I}(k\varphi))\leq (1+\delta)\nu(\varphi^j,E).
\]
Take sup with respect to $k\in \mathbb{Z}_{>0}$, we get
\[
\nu(\varphi,E)\leq (1+\delta)\nu(\varphi^j,E).
\]
Letting $j\to\infty$ and then $\delta\to 0+$, we get
\[
\nu(\varphi,E)\leq \lim_{j\to\infty}\nu(\varphi^j,E).
\]
The reverse inequality is trivial.
\end{proof}

\begin{theorem}\label{thm:nuOk}
	Let $E$ be a prime divisor on $X$. Let $(Y_{\bullet})$ be an admissible flag with $E=Y_1$. Then
	\begin{equation}\label{eq:numinOk}
		\nu(\varphi,E)=\min_{x\in \Delta(\theta,\varphi)}x_1.
	\end{equation}
\end{theorem}
Here $x_1$ denotes the first component of $x$. The generic Lelong number $\nu(\varphi,E)$ means the minimum of $\nu(\varphi,x)$ for various $x\in E$.
\begin{proof}
	We first reduce to the case where $\theta_{\varphi}$ is a Kähler current. Let $\psi\leq \varphi$, $\theta_{\psi}$ is a Kähler current. Then by \eqref{eq:numinOk} applied to $\varphi_{\epsilon}\coloneqq (1-\epsilon)\varphi+\epsilon\psi$, we have
	\[
		\nu(\varphi_{\epsilon},E)=\min_{x\in \Delta(\theta,\varphi_{\epsilon})}x_1.
	\]
	Let $\epsilon\to 0+$ using \cref{thm:Okoucont}, we conclude \eqref{eq:numinOk}.

	Similarly, taking a quasi-equisingular approximation of $\varphi$ and applying \cref{lma:qesana}, we easily reduce to the case where $\varphi$ also has analytic singularities.
	Replacing $X$ by a birational model, we may assume that $\varphi$ has analytic singularities along a simple normal crossing $\mathbb{Q}$-divisor $F$. Perturbing $L$ by an ample $\mathbb{Q}$-line bundle by \cref{prop:Deltapert}, we may assume that $\theta_{\varphi}$ is a Kähler current. 
 Finally, by rescaling, we may assume that $F$ is a divisor and $L$ is a line bundle and $L-F$ is ample by \cite[Lemma~2.4]{Xia20}. In fact, since $\theta_{\varphi}$ is a K\"ahler current, the same holds for $\theta_{\varphi}-\epsilon\omega$, where $\omega$ is a Hodge form lying in $c_1(A)$ for some ample line bundle $A$ on $X$ and $\epsilon>0$ is a small enough rational number. By \cite[Lemma~2.4]{Xia20}, we deduce that $L-F-\epsilon A$ is nef and big and hence $L-F$ is ample.

	By \cref{thm:HCP}, we know that
	\[
		\min_{x\in \Delta(\theta,\varphi)}x_1=\lim_{k\to\infty}\min_{x\in \Delta_k(\theta,\varphi)}x_1.
	\]
	By definition,
	\[
		\min_{x\in \Delta_k(\theta,\varphi)}x_1=k^{-1}\ord_E \mathrm{H}^0(X,L^k\otimes \mathcal{I}(k\varphi)).
	\]
	It remains to show that
	\begin{equation}\label{eq:temp1}
		\lim_{k\to\infty}k^{-1}\ord_E \mathrm{H}^0(X,L^k\otimes \mathcal{I}(k\varphi))=\lim_{k\to\infty}k^{-1}\ord_E \mathcal{I}(k\varphi).
	\end{equation}
	The $\geq$ direction is trivial, we prove the converse. Observe that
	\[
		\mathrm{H}^0(X,L^k\otimes \mathcal{I}(k\varphi))=\mathrm{H}^0(X,L^k\otimes \mathcal{O}_X(-kF)),\quad \mathcal{I}(k\varphi)=\mathcal{O}(-kF).
	\]
	As $L-F$ is ample, for large enough $k$, we have
	\[
		\ord_E  \mathrm{H}^0(X,L^k\otimes \mathcal{O}_X(-kF))=\ord_E(kF).
	\]
	Thus, \eqref{eq:temp1} is clear.
\end{proof}

\begin{corollary}\label{cor:Deltacontimplyvarphi}
	Let $\varphi,\psi\in \PSH(X,\theta)_{>0}$. If
	\[
		\Delta(\pi^*\theta,\pi^*\varphi)\subseteq \Delta(\pi^*\theta,\pi^*\psi)
	\]
	for all birational models $\pi:Y\rightarrow X$ and all admissible flags on $Y$, then $\varphi\preceq_{\mathcal{I}} \psi$.
\end{corollary}
\begin{proof}
    In view of \cref{thm:nuOk}, the assumption implies the following: for any prime divisor $E$ over $X$, we have $\nu(\varphi,E)\geq \nu(\psi,E)$. This implies $\varphi\preceq_{\mathcal{I}} \psi$: take a birational model $\pi\colon Y\rightarrow X$ and $y\in Y$, we need to show that $\nu(\pi^*\varphi,y)\geq \nu(\pi^*\psi,y)$. Let $E$ be the exceptional divisor of the blow-up of $Y$ at $\{y\}$. As explained in \cite[Corollaire~1.1.8]{Bou02}, we have $\nu(\pi^*\varphi,y)=\nu(\varphi,E)$ and $\nu(\pi^*\psi,y)=\nu(\psi,E)$. Our assertion follows.
\end{proof}
In particular, \cref{thm:IeqDelta} is proved.
This corollary is similar to \cite{Jow10}. It suggests that $\Delta(\theta,\varphi)$ is a universal invariant of the singularities of $\varphi$.

\cref{cor:Deltacontimplyvarphi} has a reminiscence of \cite{BFJ08}: in order to understand plurisubharmonic singularities, we need to consider all birational models of our variety at the same time.

\cref{thm:nuOk} can be regarded as a generalization of the following (slightly generalized form of the) classical result proved by Boucksom, see \cite[Theorem~5.4]{Bou02h}.
\begin{corollary}\label{cor:numin}
	Let $E$ be a prime divisor over $X$. Then
	\begin{equation}
		\nu(V_{\theta},E)=\lim_{k\to\infty}\frac{1}{k}\ord_E \mathrm{H}^0(X,L^k).
	\end{equation}
\end{corollary}

\begin{proof}
	This follows from \cref{thm:nuOk} and the fact that $\Delta(\theta,V_{\theta})=\Delta(L)$.
\end{proof}
We write
\[
	\ord_E\|L\|\coloneqq 	\lim_{k\to\infty}\frac{1}{k}\ord_E \mathrm{H}^0(X,L^k).
\]
\begin{corollary}\label{cor:IVtheta}
	We have
	\[
		\mathcal{I}(V_{\theta})=\left\{ f\in \mathcal{O}_X: \exists\epsilon>0 \textup{ s.t. } \ord_E(f)\geq (1+\epsilon)\ord_E\|L\|-A_X(E)\,\, \forall\textup{prime } E \textup{ over }X\right\},
	\]
	where $A_X(E)$ is the log discrepancy of $E$ over $X$.
\end{corollary}
\begin{proof}
	This follows from \cite[Corollary~10.17]{Bo17} and \cref{cor:numin}.
\end{proof}

\subsection{Okounkov bodies induced by filtrations}\label{subsec:comp}
Assume that $L$ is ample.
\begin{definition}
	A \emph{multiplicative filtration} on $R(X,L)$ is a decreasing, left continuous, multiplicative $\mathbb{R}$-filtration $\mathscr{F}^{\bullet}$ on the ring $R(X,L)$ which is linearly bounded in the sense that there is $C>0$, so that
	\[
		\mathscr{F}^{-k\lambda}\mathrm{H}^0(X,L^k)=\mathrm{H}^0(X,L^k),\quad \mathscr{F}^{k\lambda}\mathrm{H}^0(X,L^k)=0,
	\]
	when $\lambda>C$.

	A multiplicative filtration $\mathscr{F}$ is called a \emph{multiplicative $\mathbb{Z}$-filtration} if $\mathscr{F}^{\lambda}=\mathscr{F}^{\floor{\lambda}}$ for any $\lambda\in \mathbb{R}$.

	A multiplicative $\mathbb{Z}$-filtration $\mathscr{F}$ is called \emph{finitely generated} if the bigraded algebra
	\[
		\bigoplus_{\lambda \in \mathbb{Z},k\in \mathbb{Z}_{\geq 0}} \mathscr{F}^{\lambda}\mathrm{H}^0(X,L^k)
	\]
	is finitely generated over $\mathbb{C}$.
\end{definition}

Let $\mathscr{F}^{\bullet}$ be a multiplicative filtration on $R(X,L)$. Then we can associate a test curve $\psi_{\bullet}$ as in \cite{RWN14, Xia20}.
\begin{equation}\label{eq:testcurvfromfilt}
	\psi_{\tau}\coloneqq \sups_{k\in \mathbb{Z}_{>0}} k^{-1} \sups\left\{\,\log|s|_{h^k}^2: s\in \mathscr{F}^{k\tau}\mathrm{H}^0(X,L^k),\sup_X|s|_{h^k}\leq 1\,\right\}.
\end{equation}
Here $\sup^*$ denotes the upper-semicontinuous regularized supremum.
By \cite[Theorem~3.11]{DX22}, $\psi_{\tau}$ is $\mathcal{I}$-model or $-\infty$ for each $\tau\in \mathbb{R}$.

\begin{theorem}\label{thm:filtOko}
	Let $\mathscr{F}^{\bullet}$ be a finitely generated multiplicative $\mathbb{Z}$-filtration on $R(X,L)$. Let $\psi_{\bullet}$ be the test curve associated with $\mathscr{F}$. For any $\tau<\tau^+$,
	\[
		\Delta\left( \bigoplus_{k=0}^{\infty}\mathscr{F}^{k\tau}\mathrm{H}^0(X,L^k)\right)=\Delta(\theta,\psi_{\tau}).
	\]
\end{theorem}
\begin{proof}
	Observe that $\mathscr{F}^{k\tau}\mathrm{H}^0(X,L^k)\subseteq \mathrm{H}^0(X,L^k\otimes \mathcal{I}(k\psi_{\tau}))$ for any $k\in \mathbb{N}$.
	Thus, by \cref{cor:Okocomp},
	\[
		\Delta\left( \bigoplus_{k=0}^{\infty}\mathscr{F}^{k\tau}\mathrm{H}^0(X,L^k)\right)\subseteq\Delta(\theta,\psi_{\tau}).
	\]
	On the other hand, the two sides have the same volume by \cite[Lemma~4.5]{Xia20}. Thus, equality holds.
\end{proof}

\subsection{Limit partial Okounkov bodies}\label{subsec:limpob}
Let $\varphi\in \PSH(X,\theta)$, not necessarily of positive volume. Take an ample effective divisor $H$ on $X$ and a Kähler form $\omega\in c_1(H)$. Then we just set
\[
	\Delta(\theta,\varphi)\coloneqq \bigcap_{\epsilon \in \mathbb{Q}_{>0}} \Delta(\theta+\epsilon\omega,\varphi).
\]
Clearly, this definition does not depend on the choice of $H$ and $\omega$. As in \cite{CPW18}, we cannot expect $\Delta(\theta,\varphi)$ to be continuous along decreasing sequences of $\varphi$.
Note that \cref{thm:nuOk}, \cref{cor:Deltacontimplyvarphi} and \cref{prop:IcompimplyDeltacomp} extend to this setup without changes.

We propose the following conjecture:
\begin{conjecture}Under the above assumptions,
	\[
		\dim \Delta(\theta,\varphi)=\nd(\theta,\varphi).
	\]
\end{conjecture}
For the definition of the analytic numerical dimension $\nd(\theta,\varphi)$, we refer to \cite[Definition~4]{Cao14}.

We expect that this conjecture follows from the arguments in \cite{CPW18} together with the numerical criterion of \cite{Cao14}.

\section{Chebyshev transform}\label{sec:Che}
Let $X$ be an irreducible smooth complex projective variety of dimension $n$ and $L$ be a big line bundle on $X$. Let $h$ be a fixed smooth Hermitian metric on $L$ and $\theta=c_1(L,h)$.
Consider a singular positive Hermitian metric $\phi$ on $L$ corresponding to $\varphi\in \PSH(X,\theta)$. Assume that $\int_X\theta_{P[\varphi]_{\mathcal{I}}}^n>0$.

Let $v\in C^0(X)$ corresponding to a continuous metric $h\mathrm{e}^{-v/2}$ on $L$. We do not distinguish $v$ and $h\mathrm{e}^{-v/2}$.
Fix a valuation $\nu=(\nu_1,\ldots,\nu_n):\mathbb{C}(X)^{\times}\rightarrow \mathbb{Z}^n$ of rank $n$. Assume that $\nu$ is defined by an admissible flag $(Y_{\bullet})$ on $X$.

The whole section is devoted to the proof of \cref{thm:intdiffc}. Our results are direct extensions of the results of Witt Nystr\"om \cite{WN14}. The latter is motivated by \cite{Zah75}.

\subsection{Equilibrium energy}
Let $\mathcal{E}^{\infty}(X,\theta;P[\varphi]_{\mathcal{I}})$ denote the set of $\psi\in \PSH(X,\theta)$ such that $\psi$ and $P[\varphi]_{\mathcal{I}}$ have the same singularity types.

Let $E^{\theta}_{[\varphi]}:\mathcal{E}^{\infty}(X,\theta;P[\varphi]_{\mathcal{I}})\rightarrow \mathbb{R}$ be the relative Monge--Ampère energy:
\[
	E^{\theta}_{[\varphi]}(\psi)\coloneqq \frac{1}{n+1}\sum_{i=0}^n\int_X\left(\psi-P[\varphi]_{\mathcal{I}}\right)\,\theta_{\psi}^i\wedge \theta_{P[\varphi]_{\mathcal{I}}}^{n-i}.
\]
Define the equilibrium energy $\mathcal{E}^{\theta}_{[\varphi]}:C^0(X)\rightarrow \mathbb{R}$:
\begin{equation}\label{eq:Eequil}
	\mathcal{E}^{\theta}_{[\varphi]}(v)\coloneqq E^{\theta}_{[\varphi]}(P[\varphi]_{\mathcal{I}}(v)).
\end{equation}
Here
\[
P[\varphi]_{\mathcal{I}}(v)=\sups\left\{\eta\in \PSH(X,\theta): \eta\leq v,\eta\preceq_{\mathcal{I}}\varphi \right\}.
\]
Note that this definition is different from the energy defined in \cite{DX21}, so we choose a different notation.

\begin{theorem}\label{thm:Gat}
	The Gateaux differential of $\mathcal{E}^{\theta}_{[\varphi]}$ at $v\in C^0(X)$ is given by $\theta_{P[\varphi]_{\mathcal{I}}(v)}^n$. In other words, for any $f\in C^0(X)$,
	\begin{equation}
		\ddtz\mathcal{E}^{\theta}_{[\varphi]}(v+tf)=\int_X f\,\theta_{P[\varphi]_{\mathcal{I}}(v)}^n.
	\end{equation}
\end{theorem}
\begin{proof}
	This is not exactly \cite[Proposition~5.10]{DX21} because we are using $P[\bullet]_{\mathcal{I}}$ projections instead of $P[\bullet]$ projections, but the proofs are identical.
\end{proof}

The metric $h\mathrm{e}^{-v/2}$ induces an $L^{\infty}$-type norm $\|\bullet\|_{L^{\infty}(kv)}$ on $\mathrm{H}^0(X,L^k\otimes \mathcal{I}(k\varphi))$:
\[
	\|s\|_{L^{\infty}(kv)}\coloneqq \sup_X |s|_{h^k}\mathrm{e}^{-kv/2}.
\]
In particular, $\det \|\bullet\|_{L^{\infty}(kv)}$ is a Hermitian metric on $\det \mathrm{H}^0(X,L^k\otimes \mathcal{I}(k\varphi))$.
\begin{theorem}\label{thm:relvol}
	Let $v,v'\in C^0(X)$,
	\begin{equation}\label{eq:relvolene}
		\lim_{k\to\infty}\frac{n!}{k^{n+1}}\log \left(\det \|\bullet\|_{L^{\infty}(kv)}/\det \|\bullet\|_{L^{\infty}(kv')}\right)=\mathcal{E}^{\theta}_{[\varphi]}(v)-\mathcal{E}^{\theta}_{[\varphi]}(v').
	\end{equation}
\end{theorem}
\begin{remark}
	When $\varphi=V_{\theta}$, the left-hand side of \eqref{eq:relvolene} is known as the \emph{relative volume} between the two metrics $h\mathrm{e}^{-v/2}$ and $h\mathrm{e}^{-v'/2}$. They are studied in detail in \cite{BBWN11}.
\end{remark}
This theorem partially generalizes \cite[Theorem~A]{BB10}. We remind the readers that our conventions of multiplier ideal sheaves are different from those in \cite{BB10} and \cite{BBWN11}, which explains the difference between our coefficients and theirs.

For the definition of Bernstein--Markov property, see \cite[Definition~2.3]{BB10}. 
\begin{proof}
	We may assume that $v'=0$. Let $\nu$ be a smooth volume form on $X$. Then recall that $\nu$ satisfies the Bernstein--Markov property with respect to $tv$ for all $t\in [0,1]$ \cite[Theorem~2.4]{BB10}. We may replace the $L^{\infty}$-norm on the left-hand side with the $L^2(\nu)$-norm by \cite[Lemma~6.5]{DX21}.
	We recall the definition of the partial Bergman kernel:
	\[
		B^k_{tv,\varphi, \nu}(x) \coloneqq  \sup \left\{|s|_{h^k}^2\mathrm{e}^{-kv}(x): \int_X |s|_{h^k}^2\mathrm{e}^{-tv} \leq 1,  s \in \mathrm{H}^0(X,L^k \otimes \mathcal{I}(k\varphi)) \right\}
	\]
	and
	\[
		\beta^k_{tv,\varphi,\nu}:	=\frac{n!}{k^n} B^k_{tv,\varphi,\nu} \,\mathrm{d}\nu,
	\]
	where $k\in \mathbb{Z}_{>0}$.

	By \cite[Theorem~1.2]{DX21},
	\[
		\beta^k_{tv,\varphi,\nu}\rightharpoonup \theta^n_{P_X[\varphi]_{\mathcal{I}}(tv)}
	\]
	as $k\to\infty$ for all $t\in [0,1]$. By the dominated convergence theorem,
	\[
		\lim_{k\to\infty}\int_0^1\int_X v\,\beta^k_{tv,\varphi,\nu}\,\mathrm{d}t=\int_0^1\int_X v\,\theta^n_{P_X[\varphi]_{\mathcal{I}}(tv)}\,\mathrm{d}t
	\]
	and \eqref{eq:relvolene} follows.
\end{proof}

\begin{proposition}\label{prop:equieneconv2}
	Let $\varphi\in \PSH(X,\theta)$ such that $\theta_{\varphi}$ is a Kähler current. Let $(\varphi^j)_{j\in \mathbb{N}}$ be a quasi-equisingular approximation of $\varphi$.
	Then
	\begin{equation}\label{eq:Ethetaconv2}
		\lim_{j\to \infty}\mathcal{E}^{\theta}_{[\varphi^j]}(v)=\mathcal{E}^{\theta}_{[\varphi]}(v).
	\end{equation}
\end{proposition}
\begin{proof}
	By \cref{thm:Gat}, for $j\in \mathbb{N}$,
	\[
		\begin{aligned}
			\mathcal{E}^{\theta}_{[\varphi^j]}(v) & =\int_0^1 \int_X v\,\theta_{P[\varphi^j]_{\mathcal{I}}(tv)}^n \,\mathrm{d}t, \\
			\mathcal{E}^{\theta}_{[\varphi]}(v)   & =\int_0^1 \int_X v\,\theta_{P[\varphi]_{\mathcal{I}}(tv)}^n \,\mathrm{d}t.
		\end{aligned}
	\]
	It follows from \cite[Proposition~3.3]{DX21} and \cite[Theorem~1.2]{DDNL18mono} that as $j\to\infty$,
	\[
		\theta_{P[\varphi^j]_{\mathcal{I}}(tv)}^n\rightharpoonup \theta_{P[\varphi]_{\mathcal{I}}(tv)}^n.
	\]
	By the dominated convergence theorem, \eqref{eq:Ethetaconv2} follows.
\end{proof}

\begin{proposition}\label{prop:equieneconv3}
	Let $\varphi,\psi\in \PSH(X,\theta)$. Assume that $\psi\leq \varphi$. Set $\varphi_{\epsilon}=(1-\epsilon)\varphi+\epsilon\psi$ for any $\epsilon\in [0,1]$.
	Then
	\begin{equation}\label{eq:Ethetaconv3}
		\lim_{\epsilon\to 0+}\mathcal{E}^{\theta}_{[\varphi_{\epsilon}]}(v)=\mathcal{E}^{\theta}_{[\varphi]}(v).
	\end{equation}
\end{proposition}
\begin{proof}
	The proof is similar to that of \cref{prop:equieneconv2}. We just replace \cite[Proposition~3.3]{DX21}  by \cite[Proposition~2.7]{DX21}.
\end{proof}

We finally recall a technical lemma.
\begin{lemma}[{\cite[Corollary~3.4]{WN14}}]\label{lma:cF}
	Let $C\subseteq \mathbb{R}^{n+1}$ be an open convex cone. Let $F$ be a subadditive function on $C\cap \mathbb{Z}^{n+1}$ defined outside a compact set. Then for any sequence $\alpha_k\in C\cap \mathbb{Z}^{n+1}$ tending to infinity such that $\alpha_k/|\alpha_k|$ converges to some point $p\in C$. Then the limit
	\[
		c[F](p)\coloneqq \lim_{k\to\infty}\frac{F(\alpha_k)}{|\alpha_k|}
	\]
	exists and depends only on $p$ and $F$.
	Moreover, $c[F]$ is a convex function on $C\cap \{x_{n+1}=1\}$.
\end{lemma}
Here $|\alpha_k|$ denotes the absolute value of the last component of $\alpha_k$.

Recall that a real-valued function $F$ defined on a semigroup $\Gamma$ is said to be \emph{sub-additive} if for any $x,y\in \Gamma$, $F(x+y)\leq F(x)+F(y)$.

\subsection{The case of analytic singularities}

Assume that $\varphi$ has analytic singularities. 

Let $\pi:Y\rightarrow X$ be a resolution such that $\pi^*\varphi$ has analytic singularity along a normal crossing $\mathbb{Q}$-divisor $E$.
We define as before
\[
	W_k^{0}= \mathrm{H}^0(Y,\pi^*L^k\otimes \mathcal{O}_Y(-k E))\subseteq \mathrm{H}^0(X,L^k).
\]
Fix $a\in \Gamma_k(W^0)$.
Let $p$ be the center of $\nu$ on $X$. Let $z=(z_1,\ldots,z_n)$ be a regular sequence in $\mathcal{O}_{X,p}$ such that $(Y_i)_x$ is the zero locus of $z_1,\ldots,z_i$. Fix a local trivialization of $L$ near $p$.
Define
\[
	A^a_{k}\coloneqq \left\{s\in W^{0}_k: \nu(s)\geq ka , s=z^{ka}+\text{higher order terms near } p\right\}.
\]
Define
\[
	F[v](ka,k)=\inf_{s\in  A_{a,k}} \log |s|_{L^{\infty}(kv)}.
\]
Recall the following two lemmas proved in \cite[Lemma~5.3, Lemma~5.4]{WN14}.
\begin{lemma}
	$F[v]$ is subadditive on $\Gamma(W^{0})$.
\end{lemma}
\begin{lemma}
	There is $C>0$, so that for any $(ka,k)\in \Gamma(W^{0})$,
	\[
		F[v](ka,k)\geq C|(ka,k)|.
	\]
\end{lemma}
\begin{proof}
	It suffices to apply \cite[Lemma~5.4]{WN14}.
\end{proof}

Let $c_{[\varphi]}[v]:\Int \Delta(\theta,\varphi)\rightarrow \mathbb{R}$ be the convex function $c[F[v]]$ defined by \cref{lma:cF}.

\begin{theorem}\label{thm:WepsCheb}
	We have
	\[
		\int_{\Delta(W(\theta,\varphi))}\left(c_{[\varphi]}[v]-c_{[\varphi]}[0]\right)\,\mathrm{d}\lambda=-\mathcal{E}_{[\varphi]}^{\theta}(v).
	\]
\end{theorem}
\begin{proof}
	The proof follows \emph{verbatim} from that of \cite[Theorem~6.2]{WN14}, taking into account \cref{thm:relvol}.
\end{proof}

Observe that
\begin{equation}\label{eq:supcvdiff}
	\sup_{\Int \Delta(W(\theta,\varphi))}\left|c_{[\varphi]}[v]-c_{[\varphi]}[0]\right|\leq \|v\|_{C^0(X)}/2.
\end{equation}

The following result is obvious:
\begin{lemma}\label{lma:decc}
	Let $\varphi,\varphi'\in \PSH(X,\theta)$ be potentials with analytic singularities. If $[\varphi]\preceq [\varphi']$, then
	\[
		c_{[\varphi]}[v]\geq c_{[\varphi']}[v]
	\]
	when restricted to $\Int \Delta(\theta,\varphi)$.
\end{lemma}


\subsection{The case of Kähler currents}
Assume that $\theta_{\varphi}$ is a K\"ahler current.
Let $\varphi^j$ be a quasi-equisingular approximation of $\varphi$. Then $c_{[\varphi^j]}[v]$ restricted to $\Int \Delta(W(\theta,\varphi))$ is an increasing sequence.
Thus, we can define $c_{[\varphi]}[v]:\Int\Delta(\theta,\varphi)\rightarrow \mathbb{R}\cup \{\infty\}$ by
\[
	c_{[\varphi]}[v]\coloneqq \lim_{j\to\infty}c_{[\varphi^j]}[v].
\]
\begin{lemma}\label{lma:cvupp3}
	Let $s\in W_k(\theta,\varphi)$, locally written as $z^{ka}$ plus higher order terms near $p$. Then
	\[
		c_{[\varphi]}[v](a)\leq k^{-1}\log\|s\|_{L^{\infty}(kv)}.
	\]
\end{lemma}
\begin{proof}
	This follows from the corresponding result for the $\varphi^j$'s.
\end{proof}
By convexity, $c_{[\varphi]}[v]$ takes finite values.

It follows that \eqref{eq:supcvdiff} still holds in this case. By the dominated convergence theorem, \cref{prop:equieneconv2} and the previous case
we find
\[
	\int_{\Delta(\theta,\varphi)}\left(c_{[\varphi]}[v]-c_{[\varphi]}[0]\right)\,\mathrm{d}\lambda=-\mathcal{E}_{[\varphi]}^{\theta}(v).
\]

It follows from \cref{lma:decc} that our definition of $c_{[\varphi]}(v)$ is independent of the choice of $\varphi^j$.
\begin{lemma}\label{lma:mono2}
	Let $\varphi,\varphi'\in \PSH(X,\theta)$ be potentials such that $\theta_{\varphi}$ and $\theta_{\varphi'}$ are both Kähler currents. If $[\varphi]\preceq_{\mathcal{I}} [\varphi']$, then
	\[
		c_{[\varphi]}[v]\geq c_{[\varphi']}[v]
	\]
	when restricted to $\Int \Delta(\theta,\varphi)$.
\end{lemma}
\begin{proof}
	This follows from \cref{lma:decc}.
\end{proof}


\subsection{General case}
Let $\varphi\in \PSH(X,\theta)$ such that $\int_X \theta_{P[\varphi]_{\mathcal{I}}}^n>0$.  We may replace $\varphi$ with $P[\varphi]_{\mathcal{I}}$ and therefore assume that the non-pluripolar mass of $\varphi$ is positive.

Let $\eta\in \PSH(X,\theta)$ be a potential so that $\theta_{\eta}$ is a
Kähler current and $\eta\preceq \varphi$. The existence of such $\eta$ is guaranteed by \cite[Proposition~3.6]{DX21}.
Define $\varphi_{\epsilon}\coloneqq (1-\epsilon)\varphi+\epsilon\eta$.
Then we define $c_{[\varphi]}[v]:\Int\Delta(\theta,\varphi)\rightarrow \mathbb{R}\cup\{-\infty\}$ as
\[
	c_{[\varphi]}[v]\coloneqq \lim_{\epsilon\to 0+}c_{[\varphi_{\epsilon}]}[v].
\]
This is a decreasing limit by \cref{lma:mono2}. On the other hand, $c_{[\varphi]}[v]\geq c_{[V_{\theta}]}[v]$, the latter is finite by \cite{WN14}. Thus, $c_{[\varphi]}[v]$ is real-valued.  \eqref{eq:supcvdiff} extends to this situation. By the dominated convergence theorem and \cref{prop:equieneconv3} again,
\[
	\int_{\Delta(\theta,\varphi)}(c_{[\varphi]}[v]-c_{[\varphi]}[0])\,\mathrm{d}\lambda=-\mathcal{E}_{[\varphi]}^{\theta}(v).
\]
We do not know if $c_{[\varphi]}[v]$ is independent of the choice of $\eta$.

\section{A generalization of Boucksom--Chen theorem}\label{sec:BC}
In this section, let $X$ be an irreducible smooth projective variety of dimension $n$. Let $L$ be a big line bundle on $X$. Take a smooth Hermitian metric $h$ on $L$ with $\theta=c_1(L,h)$.

Fix a rank $n$ valuation $\nu:\mathbb{C}(X)^{\times}\rightarrow \mathbb{Z}^n$.

\subsection{The theory of test curves}
Let $V=\langle L^n\rangle$.
\begin{definition}\label{def:testcurve}
	A \emph{test curve (of finite energy)} with respect to $(X,\theta)$ is a map $\psi=\psi_{\bullet}:\mathbb{R}\to \PSH(X,\theta)\cup\{-\infty\}$, such that
	\begin{enumerate}
		\item $\psi_{\bullet}$ is concave in $\bullet$.
		\item $\psi_{\tau}$ is a model potential or $-\infty$ for any $\tau$.
		\item $\psi$ is usc as a function in the $\mathbb{R}$-variable.
		\item $\lim_{\tau\to-\infty}\psi_{\tau}=V_{\theta}$ in $L^1$.
		\item $\psi_{\tau}=-\infty$ for $\tau$ large enough.
		\item
		      \begin{equation}\label{eq:Etc}
			      \mathbf{E}(\psi_{\bullet})\coloneqq \tau^+V+\int_{-\infty}^{\tau^+} \left(\int_X \theta_{\psi_{\tau}}^n-V\right)\,\mathrm{d}\tau>-\infty.
		      \end{equation}
	\end{enumerate}
	Here $\tau^+\coloneqq \inf \{\tau\in \mathbb{R}: \psi_{\tau}=-\infty\}$. The set of test curves of finite energy with respect to $(X,\theta)$ is denoted by $\TC^1(X,\theta)$.
	We say $\psi$ is \emph{normalized} if $\tau^+=0$.
	The test curve is called \emph{bounded} if $\psi_{\tau}=V_{\theta}$ for $\tau$ small enough. Let $\tau^-\coloneqq \sup\{\tau\in \mathbb{R}: \psi_{\tau}=V_{\theta}\}$ in this case.
	The set of bounded test curves is denoted by $\TC^{\infty}(X,\theta)$.

	We say a test curve is \emph{$\mathcal{I}$-model} if $\psi_{\tau}$ is $\mathcal{I}$-model for each $\tau<\tau^+$. The set of $\mathcal{I}$-model test curves is denoted by $\TC^1_{\mathcal{I}}(X,\theta)$.
\end{definition}

\subsection{Okounkov test curves}
Let $\Delta\in \mathcal{K}^n$. Assume that $V=n!\vol\Delta>0$.
\begin{definition}\label{def:Otc}

	An \emph{Okounkov test curve} relative to $\Delta$ is an assignment $(\Delta_{\tau})_{\tau\leq \tau^+}$ ($\tau^+\in \mathbb{R}$) such that
	\begin{enumerate}
		\item $\Delta_{\tau}$ is a decreasing assignment of convex bodies in $\mathbb{R}^n$ for $\tau\leq \tau^+$;
		\item $\Delta_{\tau}$ converges to $\Delta$ as $\tau\to-\infty$ with respect to the Hausdorff metric (c.f. \cref{subsec:Hausmetric});
		\item $\Delta_{\tau}$ is concave in the $\tau$ variable;
		\item The energy is finite:
		      \[
			      \mathbf{E}(\Delta_{\bullet})\coloneqq \tau^+V+V\int_{-\infty}^{\tau^+} \left(\frac{n!}{V}\vol \Delta_{\tau}-1\right)\,\mathrm{d}\tau>-\infty\,;
		      \]
		\item Continuity holds at $\tau^+$:
		      \[
			      \Delta_{\tau^+}=\bigcap_{\tau<\tau^+}\Delta_{\tau}.
		      \]
	\end{enumerate}
\end{definition}
\begin{proposition}\label{prop:Otccont}
	Any Okounkov test curve $(\Delta_{\tau})_{\tau\leq \tau^+}$ relative to $\Delta$ is continuous for $\tau< \tau^+$.
\end{proposition}

\begin{proof}
	We first claim that $\vol\Delta_{\tau'}>0$ for all $\tau'<\tau^+$. By Condition~(2) and \cref{thm:contvol}, we know that $\vol\Delta_{\tau''}>0$ when $\tau''$ is small enough. Fix one such $\tau''$.
	Any $\tau'<\tau^+$ can be written as a convex combination of $\tau^+$ and $\tau''$, thus $\Delta_{\tau'}$ has positive volume by Condition~(3).

	Next we claim that $\vol\Delta_{\tau}$ is continuous for $\tau< \tau^+$. In fact, by Condition~(3) and the Minkowski inequality, we know that $\log \vol\Delta_{\tau}$ is concave for $\tau<\tau^+$. The continuity follows.

	Next we show that
	\[
		\Delta_{\tau}=\bigcap_{\tau'<\tau}\Delta_{\tau'}.
	\]
	The $\supseteq$ direction is obvious. By the continuity of the volume, both sides have the same volume and the volume is positive, hence, equality holds by \cref{lma:volcbimpeq}.

	Similarly, we have
	\[
		\Delta_{\tau}=\overline{\bigcup_{\tau'>\tau}\Delta_{\tau'}}.
	\]
	The continuity of $\Delta_{\tau}$ at $\tau<\tau^+$ is proved. 
\end{proof}

\begin{definition}\label{def:tf}
	A \emph{test function} on $\Delta$ is a function $F:\Delta\rightarrow [-\infty,\infty)$ such that
	\begin{enumerate}
		\item $F$ is concave;
		\item $F$ is finite on $\Int \Delta$;
		\item $F$ is usc;
		\item the energy is finite:
		      \begin{equation}\label{eq:EF}
			      \mathbf{E}(F)\coloneqq n!\int_{\Delta} F\,\mathrm{d}\lambda>-\infty.
		      \end{equation}
	\end{enumerate}
\end{definition}
Let $\tau^+=\sup_{\Delta} F$, then
\begin{equation}\label{eq:EFlevelset}
	\mathbf{E}(F)=  \tau^+V+V\int_{-\infty}^{\tau^+}\left(\frac{n!}{V}\vol \{F\geq\tau\}-1\right)\,\mathrm{d}\tau.
\end{equation}

Let $\Delta_{\bullet}$ be an Okounkov test curve relative to $\Delta$.
We define the \emph{Legendre transform} of $\Delta_{\bullet}$ as
\[
	G[\Delta_{\bullet}]:\Delta\rightarrow [-\infty,\infty),\quad a\mapsto    \sup\left\{\tau<\tau^+:a\in \Delta_{\tau}\right\}.
		\]
		Conversely, a test function $F$ on $\Delta$, set $\tau^+=\sup_{\Delta} F$.
		We define the \emph{inverse Legendre transform} of $F$ as
		\[
		\Delta[F]:(-\infty,\tau^+]\rightarrow \mathcal{K}_n,\quad \Delta[F]_{\tau}=\{F\geq\tau\}.
\]
\begin{theorem}\label{thm:Okotestcurve}
	The Legendre transform and inverse Legendre transform are inverse to each other,
	defining a bijection between the set of Okounkov test curves relative to $\Delta$ and test functions on $\Delta$.
	Moreover, if $\Delta_{\bullet}$ is an Okounkov test curve relative to $\Delta$, then
	\begin{equation}\label{eq:Deltaene}
		\mathbf{E}(\Delta_{\bullet})=\mathbf{E}\left(G[\Delta_{\bullet}]\right).
	\end{equation}
\end{theorem}
\begin{proof}
	Let $\Delta_{\bullet}$ be an Okounkov test curve relative to $\Delta$. We prove that $G[\Delta_{\bullet}]$ is a test function on $\Delta$.

	Firstly $G[\Delta_{\bullet}]$ is concave by Condition~(1) and Condition~(3) in \cref{def:Otc}. More precisely, take $a,b\in \Delta$. We want to prove that for any $t\in (0,1)$,
	\begin{equation}\label{eq:GDeltaconc}
		G[\Delta_{\bullet}](ta+(1-t)b)\geq  tG[\Delta_{\bullet}](a)+(1-t)G[\Delta_{\bullet}](b).
	\end{equation}
	There is nothing to prove if $G[\Delta_{\bullet}](a)$ or $G[\Delta_{\bullet}](b)$ is $-\infty$. So we assume that both are finite.
	Take $\epsilon>0$, then $a\in \Delta_{G[\Delta_{\bullet}](a)-\epsilon}$ and $b\in \Delta_{G[\Delta_{\bullet}](b)-\epsilon}$. Thus,
	\[
		ta+(1-t)b\in t  \Delta_{G[\Delta_{\bullet}](a)-\epsilon}+(1-t)\Delta_{G[\Delta_{\bullet}](b)-\epsilon}\subseteq \Delta_{tG[\Delta_{\bullet}](a)+(1-t)G[\Delta_{\bullet}](b)-\epsilon}.
	\]
	We deduce that
	\[
		G[\Delta_{\bullet}](ta+(1-t)b)\geq  tG[\Delta_{\bullet}](a)+(1-t)G[\Delta_{\bullet}](b)-\epsilon.
	\]
	Since $\epsilon>0$ is arbitrary, \eqref{eq:GDeltaconc} follows.

	Next $G[\Delta_{\bullet}]$ is finite on $\Int \Delta$ by Condition~(2). In fact, as $\Delta_{\tau}$ is increasing and converges to $\Delta$ as $\tau\to-\infty$, we have
	\[
		\Delta=\overline{\bigcup_{\tau}\Delta_{\tau}}.
	\]
	Hence, by \cite[Theorem~1.1.15]{Sch14} and the assumption that $\vol\Delta>0$, $\bigcup_{\tau}\Delta_{\tau}$ contains $\Int\Delta$.

	Thirdly, we show that $G[\Delta_{\bullet}]$ is usc.
	Let $a_i\in \Delta$ with $a_i\to a\in \Delta$. Define $\tau_i=G[\Delta_{\bullet}](a_i)$. Let $\tau=\varlimsup_i \tau_i$.
	We need to show that
	\begin{equation}\label{eq:ainDelta1}
		G[\Delta_{\bullet}](a)\geq \tau.
	\end{equation}
	There is nothing to prove if $\tau=-\infty$. We assume that it is not this case. Up to subtracting a subsequence we may assume that $\tau_i\to \tau$. In particular, we can assume that $\tau_i\neq -\infty$ for all $i$. Fix $\epsilon>0$, then $a_i\in\Delta_{\tau_i-\epsilon}$.
	Observe that $\Delta_{\tau_i-\epsilon}\xrightarrow{d_n}\Delta_{\tau-\epsilon}$. By \cref{thm:Hausconvcond} it follows that $a\in \Delta_{\tau-\epsilon}$.
	Thus,\eqref{eq:ainDelta1} follows since $\epsilon>0$ is arbitrary.

	Finally, \eqref{eq:Deltaene} follows from \eqref{eq:EFlevelset}, and it follows that $\mathrm{E}(G[\Delta_{\bullet}])>-\infty$.

	Conversely, if $F:\Delta\rightarrow [-\infty,\infty)$ is a test function on $\Delta$. Let $\Delta[F]$ be the inverse Legendre transform of $F$.
	Then one can similarly show that $\Delta[F]$ is an Okounkov test curve.

	Firstly, for each $\tau<\tau^+\coloneqq \sup_{\Delta}F$, $\Delta[F](\tau)$ is a convex body as $F$ is concave and usc.
	Moreover, $\Delta[F]_{\tau}$ is clearly decreasing in $\tau$. Hence, $\Delta[F]_{\tau^+}$ is also a convex body.

	Secondly, for each $a\in \Delta$, we can write $a=\lim_{i}a_i$ with $a_i\in \Int \Delta$.  By assumption, $F$ is finite at $a_i$. Thus,
	\[
		a\in \overline{\{F>-\infty\}}=\overline{\bigcup_{\tau}\Delta[F]_{\tau}}.
	\]
	By \cref{thm:Hausconvcond}, $\Delta[F]_{\tau}\xrightarrow{d_n}\Delta$ as $\tau\to -\infty$.

	Thirdly, $\Delta[F]$ is concave. To see, take $\tau,\tau'\leq \tau^+$, we need to prove that for any $t\in (0,1)$,
	\begin{equation}\label{eq:Deconc}
		\Delta[F]_{t\tau+(1-t)\tau'}\supseteq t\Delta[F]_{\tau}+(1-t)\Delta[F]_{\tau'}.
	\end{equation}
	Let $a\in \Delta[F]_{\tau}$ and $b\in \Delta[F]_{\tau'}$. We have $F(a)\geq\tau$ and $F(b)\geq \tau'$. As $F$ is concave, we have $F(ta+(1-t)b)\geq t\tau+(1-t)\tau'$. Thus,
	\[
		ta+(1-t)b\in   \Delta[F]_{t\tau+(1-t)\tau'}
	\]
	and \eqref{eq:Deconc} follows.

	Fourthly, \eqref{eq:EF} follows immediately from \eqref{eq:EFlevelset}.

	Finally, we show that $\Delta[F]_{\bullet}$ is continuous at $\tau^+$. This amounts to
	\[
		\{F\geq \tau^+\}    =\bigcap_{\tau<\tau^+}\{F\geq \tau\},
	\]
	which is obvious.

	To see that these two operations are inverse to each other, observe that by definition for any Okounkov test curve $\Delta_{\bullet}$, any $a\in \Delta$  and any $\tau\leq \tau^+$, $G[\Delta_{\bullet}](a)\geq \tau$ if and only if $a\in \Delta_{\tau-\epsilon}$ for any $\epsilon>0$. By \cref{prop:Otccont}, this happens if and only if $a\in \Delta_{\tau}$, that is,
	\[
		\{G[\Delta_{\bullet}]\geq\tau\}=\Delta_{\tau}.
	\]

	Conversely, for any test function $F:\Delta\rightarrow [-\infty,\infty)$, any $\tau\leq \tau^+$, by definition,
	\[
		\{F\geq \tau\}=\Delta[F]_{\tau}.
	\]
\end{proof}

\begin{definition}
	Let $\Delta_{\bullet}$ be an Okounkov test curve relative to $\Delta$. We define the \emph{Duistermaat--Heckman measure} $\DHm(\Delta_{\bullet})$ as
	\[
		\DHm(\Delta_{\bullet})\coloneqq G[\Delta_{\bullet}]_*(\mathrm{d}\lambda).
	\]
	It is a Radon measure on $\mathbb{R}$.
\end{definition}
Observe that
\begin{equation}\label{eq:massDH}
	\int_{\mathbb{R}}\DHm(\Delta_{\bullet})= \vol \Delta.
\end{equation}

\subsection{Boucksom--Chen theorem}
Let $\psi_{\bullet}\in \TC^1_{\mathcal{I}}(X,\theta)$.
Let $\tau^+=\inf \{\tau\in \mathbb{R}:\psi_{\tau}=-\infty \}$.

\begin{lemma}
	The curve
	\[
		\Delta[\psi_{\bullet}]_{\tau}\coloneqq \left\{
		\begin{aligned}
			\Delta(\theta,\psi_{\tau})                           & ,\quad \tau<\tau^+, \\
			\bigcap_{\tau'<\tau^+}\Delta[\psi_{\bullet}]_{\tau'} & ,\quad \tau=\tau^+
		\end{aligned}
		\right.
	\]
	is an Okounkov test curve relative to $\Delta(L)$. Moreover,
	\begin{equation}\label{eq:EpsiEDeltapsi}
		\mathbf{E}(\psi_{\bullet})=\mathbf{E}(\Delta[\psi_{\bullet}]_{\bullet}).
	\end{equation}
\end{lemma}
\begin{proof}
	We verify the conditions in \cref{def:Otc}.
	Condition~(1) follows from \cref{prop:IcompimplyDeltacomp}. Condition~(2) follows from the fact
	\[
		\lim_{\tau\to-\infty}\vol\Delta_{\tau}=  \vol\Delta.
	\]
	Condition~(3) follows from \cref{thm:concOko} and \cref{prop:IcompimplyDeltacomp}. Condition~(4) is a translation of \eqref{eq:Etc}. Condition~(5) is obvious.

	Finally, \eqref{eq:EpsiEDeltapsi} follows from \eqref{eq:Etc} and \eqref{eq:volD1}.
\end{proof}
\begin{definition}
	Let $\psi_{\bullet}\in \TC^1_{\mathcal{I}}(X,\theta)$. Define the \emph{Duistermaat--Heckman measure} of $\psi_{\bullet}$ as
	\[
		\DHm(\psi_{\bullet})\coloneqq \DHm(\Delta[\psi_{\bullet}]_{\bullet}).
	\]
\end{definition}
We write
\[
	G[\psi_{\bullet}]=G[\Delta[\psi_{\bullet}]].
\]
Then
\[
	\DHm(\psi_{\bullet})=G[\psi_{\bullet}]_*(\mathrm{d}\lambda).
\]

Now consider the (not necessarily multiplicative) filtration:
\[
	\mathscr{F}^k_{\tau}\mathrm{H}^0(X,L^k)\coloneqq \left\{
	\begin{aligned}
		\mathrm{H}^0(X,L^k\otimes \mathcal{I}(k\psi_{\tau})) & ,\quad \tau<\tau^+,     \\
		0                                                    & ,\quad \tau\geq \tau^+.
	\end{aligned}
	\right.
\]
Let $e_j(\mathrm{H}^0(X,L^k),\mathscr{F}^k)$ be the jumping numbers of $\mathscr{F}^k$ listed in the decreasing order.
In order words,
\[
	e_j\left(\mathrm{H}^0(X,L^k),\mathscr{F}^k\right)\coloneqq \sup\left\{\tau\in\mathbb{R}:\dim \mathscr{F}^k_{\tau}\mathrm{H}^0(X,L^k)\geq j \right\}.
\]
Let
\[
	\mu_k\coloneqq \frac{1}{k^n}\sum_{j= 1}^{h^0(X,L^k)}\delta_{e_j\left(\mathrm{H}^0(X,L^k),\mathscr{F}^k\right)}.
\]
\begin{theorem}\label{thm:BCgen}
	Let $\psi_{\bullet}\in \TC^1_{\mathcal{I}}(X,\theta)$. Then
	as $k\to\infty$, $\mu_k$ converges weakly to $\DHm(\psi_{\bullet})$.
\end{theorem}
As explained in \cite{RWN14, DX22, Xia20}, $\TC^1_{\mathcal{I}}(X,\theta)$ is the completion of the space of filtrations, so this theorem indeed generalizes \cite[Theorem~A]{BC11}, in the case of full graded linear series. 
\begin{proof}
	It suffices to show the convergence holds as distributions.
	By our definition, $\mu_k$ is the distributional derivative of the function
	\[
		h_k(\tau)\coloneqq \left\{
		\begin{aligned}
			k^{-n}h^0(X,L^k\otimes \mathcal{I}(k\psi_\tau)) & ,\quad \tau<\tau^+,     \\
			0                                               & ,\quad \tau\geq \tau^+.
		\end{aligned}
		\right.
	\]
	On the other hand, $\DHm(\psi_{\bullet})$ is the distributional derivative of $h(\tau)\coloneqq \vol\{G[\Delta[\psi_{\bullet}]_{\bullet}]\geq\tau\}=\vol \Delta_{\tau}$ by Fubini--Tonelli theorem.

	By \cref{thm:DXmain}, $h_k(\tau)\to h(\tau)$ for all $\tau\neq \tau^+$.
	By the dominated convergence theorem $h_k\to h$ in $L^1_{\mathrm{\loc}}(\mathbb{R})$.
	Hence, $\mu_k\rightharpoonup \DHm(\psi_{\bullet})$.
\end{proof}

\begin{corollary}
	For any $\psi_{\bullet}\in \TC^1_{\mathcal{I}}(X,\theta)$.
	The Duistermaat--Heckman measure $\DHm(\psi_{\bullet})$ is independent of the choice of the valuation $\nu$.
\end{corollary}

\subsection{Applications to non-Archimedean geometry}\label{subsec:apptonA}
Assume that $L$ is ample and $\theta$ is a Kähler form. We write $\omega=\theta$ instead. 

\subsubsection*{Finite energy geodesic rays}
Let $\mathcal{E}^1(X,\omega)$ denote the space of $\omega$-psh functions with finite energy:
\[
	\mathcal{E}^1(X,\omega)\coloneqq \left\{\varphi\in \PSH(X,\omega):\int_X \omega_{\varphi}^n=\int_X\omega^n,\int_X |\varphi|\,\omega_{\varphi}^n<\infty \right\}.
\]
See \cite{Dar19} for a detailed introduction. Recall that $\mathcal{E}^1(X,\omega)$ admits a natural metric $d_1$: for $\varphi,\psi\in \mathcal{E}^1(X,\omega)$, given by
\[
	d_1(\varphi,\psi)\coloneqq \mathrm{E}(\varphi)+\mathrm{E}(\psi)-2\mathrm{E}(\varphi\land\psi).
\]
Here 
\[
\varphi\land\psi\coloneqq \sup \left\{\eta\in \PSH(X,\omega): \eta\leq \varphi,\eta\leq \psi\right\}.
\]
In \cite[Theorem~2.10]{DDNL18fullmass}, Darvas--Di Nezza--Lu proved that $\varphi\land\psi\in \mathcal{E}^1(X,\omega)$. They proved in \cite[Section~3]{DDNL18big} that $d_1$ is indeed a metric.
The Monge--Amp\`ere energy functional $\mathrm{E}:\mathcal{E}^1(X,\omega)\rightarrow \mathbb{R}$ is defined as
\[
	\mathrm{E}(\varphi)=\frac{1}{n+1}\sum_{i=0}^n\int_X \varphi\,\omega_{\varphi}^i\wedge \omega^{n-i}.
\]

In this case, let $\mathcal{R}^1(X,\omega)$ denote the set of geodesic rays in $\mathcal{E}^1(X,\omega)$ emanating from $0$. For a detailed study of $\mathcal{R}^1(X,\omega)$, we refer to \cite{DL20}.
Here we only recall the definition of the metric on $\mathcal{R}^1(X,\omega)$. Given $\ell,\ell'\in \mathcal{R}^1(X,\omega)$, we define
\[
	d_1(\ell,\ell')\coloneqq \lim_{t\to\infty}\frac{1}{t}d_1(\ell_t,\ell'_t).
\]
By \cite[Corollary~5.5]{CC3}, $t\mapsto d_1(\ell_t,\ell'_t)$ is convex, guaranteeing the existence of the limit. 
It is shown in \cite{DL20} that $(\mathcal{R}^1(X,\omega),d_1)$ is a complete metric space.

The following notion is introduced in \cite{XiaMabuchi}:
\begin{definition}
	A \emph{rooftop metric space} is a triple $(E,d,\land)$: $(E,d)$ is a metric space and $\land:E\times E\rightarrow E$ is an associative, commutative binary operator on $E$ satisfying
	\[
		d(a\land c,b\land c)\leq d(a,b)
	\]
	for any $a,b,c\in E$.
\end{definition}
For $\ell,\ell'\in \mathcal{R}^1(X,\omega)$, define $\ell\land \ell'$ as the greatest geodesic in $\mathcal{R}^1(X,\omega)$ that lies below both $\ell$ and $\ell'$.
It is shown in \cite[Theorem~7.6]{XiaMabuchi} that $\land$ is well-defined and $(\mathcal{R}^1(X,\omega),d_1,\land)$ is a complete rooftop metric space.

The energy functional $\mathbf{E}:\mathcal{R}^1(X,\omega)\rightarrow \mathbb{R}$ is defined as
\[
	\mathbf{E}(\ell)\coloneqq \mathrm{E}(\ell_1).
\]
Recall that we have the following two maps: given any $\ell\in \mathcal{R}^1(X,\omega)$, its \emph{inverse Legendre transform} is defined as
\[
	\hat{\ell}_{\tau}\coloneqq \inf_{t\geq 0} (\ell_t-t\tau).
\]
Conversely, given any $\psi_{\bullet}\in \TC^1(X,\omega)$, we define its \emph{Legendre transform} by
\[
	\check{\psi}_t\coloneqq   \sup_{\tau\in \mathbb{R}}(\psi_{\tau}+t\tau).
\]
They are inverse to each other, as proved in \cite[Theorem~3.7]{DX22}.

\subsubsection*{Non-Archimedean pluripotential theory}
Let $X^{\An}$ be the Berkovich analytification of $X$ with respect to the trivial valuation on $X$ and $L^{\An}$ be the analytification of $L$. See \cref{subsec:nAp} for a brief introduction. In the same section, we also recalled the definition of the space $\mathcal{E}^1(L^{\An})$ of non-Archimedean psh metrics on $L^{\An}$ with finite energy and the energy functional $\mathrm{E}:\mathcal{E}^1(L^{\An})\rightarrow \mathbb{R}$.

Next we briefly explain the relation between the non-Archimedean pluripotential theory and the complex pluripotential theory. Firstly, given a geodesic ray $\ell\in \mathcal{R}^1(X,\omega)$, one can associate a non-Archimedean potential $\ell^{\An}\in \mathcal{E}^1(L^{\An})$ as in \cite[Definition~4.2, Theorem~6.2]{BBJ21}. The construction of $\ell^{\An}$ requires the notion of Gauss extension of valuations, as explained in \cite[Section~3.1]{BBJ21}. The map 
\[
\mathcal{R}^1(X,\omega)\rightarrow \mathcal{E}^1(L^{\An})
\]
is surjective but not injective. It admits a canonical section
\[
	\iota:\mathcal{E}^1(L^{\An})\hookrightarrow   \mathcal{R}^1(X,\omega)
\]
sending $\phi \in \mathcal{E}^1(L^{\An})$ to the maximal element $\ell\in \mathcal{E}^1(L^{\An})$ with $\ell^{\An}=\phi$. See \cite[Theorem~6.6]{BBJ21}.

The geodesics lying in the image of $\iota$ are known as \emph{maximal geodesic rays} or \emph{approximable geodesic rays}. Moreover,
\begin{equation}\label{eq:ENAEequal}
	\mathbf{E}(\iota(\alpha))=\mathrm{E}(\alpha)
\end{equation}
for any $\alpha\in \mathcal{E}^1(L^{\An})$, see \cite[Corollary~6.7]{BBJ21}.

Maximal geodesic rays are closely related to test curves:
\begin{theorem}
	The Legendre transform is a bijection from $\TC^1_{\mathcal{I}}(X,\omega)$ (resp. $\TC^1(X,\omega)$) to $\iota(\mathcal{E}^{1}(L^{\An}))$ (resp. $\mathcal{R}^1(X,\omega)$), the converse is given by the inverse Legendre transform.
	Moreover, for any $\psi_{\bullet}\in \TC^1(X,\omega)$,
	\begin{equation}\label{eq:Eequal}
		\mathbf{E}(\psi_{\bullet})=\mathbf{E}(\check{\psi}).
	\end{equation}
\end{theorem}
This is one of the main theorems of \cite[Theorem~3.7, Theorem~3.17]{DX22}. It is based on the previous works \cite{RWN14, DDNL18big}.

\subsubsection*{Duistermaat--Heckman measures}

The space $\mathcal{E}^1(L^{\An})$ is closely related to the theory of test configurations. For the latter, we refer to \cite[Section~2]{BHJ17} for a brief introduction. Recall that two test configurations $(\mathcal{X},\mathcal{L})$ and $(\mathcal{X}',\mathcal{L}')$ of $(X,L)$ are said to be equivalent if they can be dominated by a common test configuration, see \cite[Definition~6.1]{BHJ17}. There is a natural injection from the set of equivalence classes of test configurations to $\mathcal{E}^1(L^{\An})$. Moreover, this injection has dense image and $\mathcal{E}^1(L^{\An})$ is the $d_1$-completion of the space of test configurations (modulo the equivalence relation). These results are explained in detail in \cite[Section~3.2]{DX22}.

Given a test configuration $(\mathcal{X},\mathcal{L})$, Witt Nystr\"om \cite{WN12} constructed a naturally defined Radon measure $\DHm(\mathcal{X},\mathcal{L})$ on $\mathbb{R}$, called the \emph{Duistermaat--Heckman measure}. See \cite[Section~3.2]{BHJ17} for more details. It is not hard to see from the definition that $\DHm(\mathcal{X},\mathcal{L})$ depends only on the equivalence class of $(\mathcal{X},\mathcal{L})$. 

In the sequel, we will define the Duistermaat--Heckman measure of an element in $\mathcal{E}^1(L^{\An})$.
As the space $\mathcal{E}^1(L^{\An})$ is the completion of the space of test configurations (modulo the equivalence relation), our definition can be seen as an extension of Witt Nystr\"om's results \cite{WN12}.

\begin{definition}\label{def:DHmNA}
	For any $\alpha\in \mathcal{E}^1(L^{\An})$, define the \emph{Duistermaat--Heckman measure} of $\alpha$ as
	\[
		\DHm(\alpha)\coloneqq  \DHm\left(\widehat{\iota(\alpha)}\right).
	\]
\end{definition}
We get a map $\DHm:\mathcal{E}^1(L^{\An})\rightarrow \mathcal{M}(\mathbb{R})$.
Here $\mathcal{M}(\mathbb{R})$ denotes the space of Radon measures on $\mathbb{R}$.

For the proof of the next theorem, we need to recall several basic constructions of test curves.

The space $\TC^1(X,\omega)$ is a rooftop metric space.
Its rooftop structures $(d_1,\land)$ are induced from the corresponding structures on $\mathcal{R}^1(X,\omega)$.
\begin{corollary}Let $\psi_{\bullet},\varphi_{\bullet},\eta_{\bullet}\in \TC^1(X,\omega)$.
	\begin{enumerate}
		\item The rooftop operator on $\TC^1(X,\omega)$ is given by
		      \begin{equation}\label{eq:rooftoptc}
			      \left(\psi\land\varphi\right)_{\tau}=\psi_{\tau}\land\varphi_{\tau}.
		      \end{equation}
		      It is the maximal element in $\TC^1(X,\omega)$ that lies below both $\psi_{\bullet}$ and $\varphi_{\bullet}$. In particular,
		      \begin{equation}\label{eq:d1rooftopineq}
			      d_1((\psi\land\eta)_{\bullet},(\varphi\land\eta)_{\bullet})\leq d_1(\psi_{\bullet},\varphi_{\bullet}).
		      \end{equation}
		\item The metric on  $\TC^1(X,\omega)$  is given by
		      \begin{equation}\label{eq:d1TC1}
			      d_1(\psi_{\bullet},\varphi_{\bullet})\coloneqq \mathbf{E}(\psi_{\bullet})+\mathbf{E}(\varphi_{\bullet})-2\mathbf{E}((\psi\land\varphi)_{\bullet}).
		      \end{equation}
	\end{enumerate}
\end{corollary}
\begin{proof}
	(1) Note that $\eqref{eq:d1rooftopineq}$ is part of our definition of a rooftop structure.

	Observe that the bijection $\TC^1(X,\omega)\rightarrow \mathcal{R}^1(X,\omega)$ is order-preserving. In order to prove our claim, it suffices to show that $(\varphi\land \psi)_{\bullet}$ defined by \eqref{eq:rooftoptc} is indeed in $\TC^1(X,\omega)$, which is obvious.

	(2) This follows simply from (1) and \eqref{eq:Eequal}.
\end{proof}
If $\varphi_{\bullet},\psi_{\bullet}\in \TC_{\mathcal{I}}^1(X,\omega)$, then $(\psi\land \varphi)_{\bullet}\in \TC_{\mathcal{I}}^1(X,\omega)$ as well. This follows from the simple observation that the rooftop of two $\mathcal{I}$-model potentials is still $\mathcal{I}$-model.
Now the $d_1$ metric on $\TC^1(X,\omega)$ restricts to a metric $d_1$ on $\TC^1_{\mathcal{I}}(X,\omega)$. The rooftop structure also restricts to a rooftop structure on $\TC^1_{\mathcal{I}}(X,\omega)$.

We need the following constructions on test curves.
\begin{enumerate}
	\item Increasing limit. Let $\psi^{\alpha}_{\bullet}\in \TC^1(X,\omega)$ be an increasing net. Assume that $\tau^+_{\psi^{\alpha}}$ is bounded from above.
	      Define
	      \[
		      \tilde{\psi}_{\tau}\coloneqq C[\sups_{\alpha}\psi^{\alpha}_{\tau}].
	      \]
	      Let $\tau^+=\inf\{\tau:\tilde{\psi}_{\tau}=-\infty\}$. We define
	      \[
		      \psi_{\tau}=
		      \left\{
		      \begin{aligned}
			      \tilde{\psi}_{\tau}                           & , \quad \tau\neq \tau^+\,; \\
			      \lim_{\sigma\to \tau^+-}\tilde{\psi}_{\sigma} & ,\quad \tau\neq \tau^+.
		      \end{aligned}
		      \right.
	      \]
	      It is easy to verify that $\psi_{\bullet}\in \TC^1(X,\omega)$.
	\item Decreasing limit. Let $\psi^{\alpha}_{\bullet}\in \TC^1(X,\omega)$ be an increasing net and $\eta_{\bullet}\in \TC^1(X,\omega)$. Assume that $\psi^{\alpha}_{\bullet}\geq \eta_{\bullet}$ for all $\alpha$. Define
	      \[
		      (\inf\psi)_{\tau}\coloneqq \inf_{\alpha}\psi^{\alpha}_{\tau}.
	      \]
	      Then if $(\inf_{\psi})_{\bullet}$ is not identically $-\infty$, then $(\inf\psi)_{\bullet}\in \TC^1(X,\omega)$.
	\item Max. Let $\varphi_{\bullet},\psi_{\bullet}\in \TC^1(X,\omega)$. There is the smallest test curve $(\varphi\lor\psi)_{\bullet}\in \TC^1(X,\omega)$ such that $(\varphi\lor\psi)_{\bullet}\geq \varphi_{\bullet}$, $(\varphi\lor\psi)_{\bullet}\geq \psi_{\bullet}$. In fact, we could simply define
	      \[
		      (\varphi\lor\psi)_{\tau}\coloneqq \inf\left\{\eta_{\tau}:\eta_{\bullet}\in \TC^1(X,\omega),\eta_{\bullet}\geq \varphi_{\bullet},\eta_{\bullet}\geq \psi_{\bullet} \right\}.
	      \]
	      In terms of the Legendre transform, $(\varphi\lor\psi)^{\check{}}$ is the minimal geodesic ray lying above both $\check{\varphi}$ and $\check{\psi}$.
	      We observe that
	      \begin{equation}\label{eq:d1ineq}
		      d_1(\varphi_{\bullet},\psi_{\bullet})\leq d_1(\varphi_{\bullet},(\varphi\lor\psi)_{\bullet})+
		      d_1(\psi_{\bullet},(\varphi\lor\psi)_{\bullet}) \leq C_0d_1(\varphi_{\bullet},\psi_{\bullet})
	      \end{equation}
	      for some $C_0(n)>0$. See \cite[Proposition~2.15]{DDNLmetric} for the proof of the latter inequality. Moreover, if $\eta_{\bullet}\in \TC^1(X,\omega)$ and if $\varphi_{\bullet}\leq \psi_{\bullet}$, then
	      \begin{equation}\label{eq:d1max}
		      d_1((\varphi\lor\eta)_{\bullet},(\psi\lor\eta)_{\bullet})\leq d_1(\varphi_{\bullet},\psi_{\bullet}).
	      \end{equation}
	      This follows from the corresponding inequality of geodesic rays, which in turn follows from \cite[Proposition~4.12]{XiaMabuchi} (Proposition~6.8 in the arXiv version).

	      We also observe that the operator $\lor$ is associative and commutative, hence, we could also define $\psi^{1}_{\bullet}\lor\cdots \lor \psi^{k}_{\bullet}$ in the obvious way.
\end{enumerate}

\begin{lemma}\label{lma:incdecd1}
	Let $\psi^j_{\bullet},\psi_{\bullet}\in \TC^1(X,\omega)$. Assume that one of the following conditions holds
	\begin{enumerate}
		\item $\psi^j_{\bullet}$ is increasing and $\psi_{\bullet}$ is the increasing limit of $\psi^j_{\bullet}$.
		\item $\psi^j_{\bullet}$ is decreasing and $\psi_{\bullet}=(\inf\psi)_{\bullet}$.
	\end{enumerate}
	Then $\psi^j_{\bullet}\xrightarrow{d_1}\psi_{\bullet}$.
\end{lemma}
\begin{proof}
	We assume that condition (2) holds, the other case is similar. First observe that $\tau^+_{\psi^j}\to \tau^+_{\psi}$.
	It suffices to observe that
	\[
		d_1(\psi^j_{\bullet},\psi_{\bullet})=(\tau^+_{\psi^j}-\tau^+_{\psi})\int_X\omega^n+\int_{-\infty}^{\infty}\left(\int_X\omega_{\psi^j_{\tau}}^n-\int_X\omega_{\psi_{\tau}}^n\right)\,\mathrm{d}\tau.
	\]
	The assertion is a simple consequence of dominated convergence theorem.
\end{proof}

\begin{theorem}\label{thm:contDH}
	The map $\DHm:\mathcal{E}^1(L^{\An})\rightarrow \mathcal{M}(\mathbb{R})$ is continuous.

	For any $\alpha\in \mathcal{E}^1(L^{\An})$,
	\begin{equation}\label{eq:momDHm}
		\int_{\mathbb{R}} x\,\mathrm{d}\DHm(\alpha)(x)=\mathrm{E}(\alpha)
	\end{equation}
	and
	\begin{equation}\label{eq:massDHm}
		\int_{\mathbb{R}}\DHm(\alpha)=\frac{1}{n!}(L^n).
	\end{equation}
\end{theorem}
\begin{proof}
	We first prove the continuity of $\DHm$.

	By the dominated convergence theorem, it suffices to show that $G[\psi_{\bullet}](x)$ depends continuously on $\psi_{\bullet}$ for almost all $x\in \Int\Delta(L)$. To be more precise, let $\psi^j_{\bullet}\in \TC^1_{\mathcal{I}}(X,\omega)$ be a sequence converging to $\psi_{\bullet}$. We want to show that
	\[
		G[\psi^j_{\bullet}](x)\to G[\psi_{\bullet}](x)
	\]
	for almost all $x\in \Int \Delta(L)$.
	We will reduce to the case where $\psi^j_{\bullet}$ is either increasing or decreasing.  In these cases, it suffices to show that $G[\psi^j_{\bullet}]\to G[\psi_{\bullet}]$ in $L^1$.
	By \eqref{eq:Deltaene} and \eqref{eq:EpsiEDeltapsi}, this amounts to showing that $\mathbf{E}(\psi^j_{\bullet})\to \mathbf{E}(\psi_{\bullet})$. The latter follows from \cref{lma:incdecd1}.

	In order to make the reduction, we will prove that after passing to a subsequence, there exists an increasing sequence $\varphi_{\bullet}^j\in \TC^1_{\mathcal{I}}(X,\omega)$ and a decreasing sequence $\eta_{\bullet}^j\in \TC^1_{\mathcal{I}}(X,\omega)$ such that $\varphi_{\bullet}^j\leq \psi_{\bullet}^j \leq \eta_{\bullet}^j$
	and $\varphi^j_{\bullet}\xrightarrow{d_1}\psi_{\bullet}$, $\eta^j_{\bullet}\xrightarrow{d_1}\psi_{\bullet}$. In fact, we can relax the requirement to $\varphi^j_{\bullet},\eta^j_{\bullet}\in \TC^1(X,\omega)$, not necessarily $\mathcal{I}$-model. Then it suffices to replace both test curves by their pointwise $\mathcal{I}$-projections, which satisfy the same conditions by \cite[Theorem~3.18]{DX22}.

	Up to subtracting a subsequence, we may assume that for all $j$,
	\[
		d_1(\psi^{j}_{\bullet},\psi_{\bullet})\leq 2^{-j}.
	\]

	For $k\geq j\geq 0$, we set
	\[
		\eta^{j,k}_{\bullet}\coloneqq \psi^j_{\bullet}\lor \cdots\lor \psi^{k}_{\bullet}\in \TC^1(X,\omega).
	\]
	Let $\eta^j_{\bullet}\in \TC^1(X,\omega)$ be the increasing limit of $\eta^{j,k}_{\bullet}$ as $k\to\infty$.
	We then have
	\[
		\begin{aligned}
			d_1(\eta^{j,k}_{\bullet},\psi_{\bullet})\leq & d_1(\psi_{\bullet},(\psi\lor\psi^j)_{\bullet})+d_1((\psi\lor\psi^j)_{\bullet},(\psi\lor\psi^j\lor\psi^{j+1})_{\bullet})+\cdots \\
			                                             & +d_1((\psi\lor\psi^j\lor\cdots\lor\psi^{k-1})_{\bullet},(\psi\lor\psi^j\lor\cdots\lor\psi^{k})_{\bullet})                      \\
			\leq                                         & d_1(\psi_{\bullet},(\psi\lor\psi^j)_{\bullet})+\cdots+d_1(\psi_{\bullet},(\psi\lor\psi^k)_{\bullet})                           \\
			\leq                                         & C_0 \sum_{i=j}^k d_1(\psi_{\bullet},\psi^i_{\bullet})                                                                          \\
			\leq                                         & C_0 2^{1-j}.
		\end{aligned}
	\]
	Here the second inequality follows from \eqref{eq:d1max}, the third inequality follows from \eqref{eq:d1ineq}.
	Then by \cref{lma:incdecd1}, we find that $d_1(\eta^{j}_{\bullet},\psi_{\bullet})\leq C 2^{1-j}$.
	Thus, $\eta^j_{\bullet}\xrightarrow{d_1}\psi_{\bullet}$.

	Similarly, for $k\geq j\geq 0$, let
	\[
		\varphi^{j,k}_{\bullet}\coloneqq \psi^j_{\bullet}\land \cdots\land \psi^{k}_{\bullet}\in \TC^1(X,\omega).
	\]
	The same argument as above shows that for $k\geq j\geq 0$, $d_1(\varphi^{j,k}_{\bullet},\psi_{\bullet})\leq 2^{1-j}$.
	Let
	\[
		\psi^j_{\tau}\coloneqq \inf_{k\geq j}\varphi^{j,k}_{\tau}.
	\]
	By monotone convergence theorem, we find that $\psi^j\in \TC^1(X,\omega)$. Thus, by \cref{lma:incdecd1}, $d_1(\varphi^{j}_{\bullet},\psi_{\bullet})\leq 2^{1-j}$.

	Next we prove \eqref{eq:momDHm}. Let $\alpha\in \mathcal{E}^1(L^{\An})$. Let $\psi_{\bullet}$ be the test curve corresponding to $\alpha$.
	We need to compute
	\[
		\int_\mathbb{R}x\,\DHm(\alpha)(x)=\int_{\Delta(L)}G[\psi_{\bullet}]\,\mathrm{d}\lambda.
	\]
	By \eqref{eq:EFlevelset}, \eqref{eq:Deltaene} and \eqref{eq:EpsiEDeltapsi}, the right-hand side is just $\mathbf{E}(\psi_{\bullet})$,
	which is equal to $\mathrm{E}(\alpha)$ by \eqref{eq:ENAEequal} and \eqref{eq:Eequal}.

	Finally, \eqref{eq:massDHm} follows from \eqref{eq:massDH}.
\end{proof}
\begin{remark}\label{rmk:Ino}
    On the subspace $\mathcal{H}^{\NA}$, the Duistermaat--Heckman measure is the same as the one defined in \cite[Section~3.2]{BHJ17}. This follows from \cref{thm:filtOko} and \cite[Theorem~A]{BC11}. On the other hand, in \cite[Definition~3.56]{Ino22}, Inoue defined the Duistermaat--Heckman measure for a general non-Archimedean metric on $L^{\An}$. As explained in \cite[Remark~1.4]{Ino22}, his definition agrees with ours for metrics in $\mathcal{E}^1(L^{\An})$.
\end{remark}

\section{Toric setting}\label{sec:tor}

This section is devoted to a toric interpretation of the partial Okounkov body construction.

\subsection{Technical lemmata}

\begin{lemma}\label{lma:polybdd}
	Let $\alpha,\beta_1,\ldots,\beta_m\in \mathbb{Z}^n$. Let $\Delta$ be the convex polytope generated by $\beta_1,\ldots,\beta_m$.  Then the following are equivalent:
	\begin{enumerate}
		\item
		      \begin{equation}\label{eq:zalpha}
			      |z^{\alpha}|^2\left(\sum_{i=1}^m |z^{\beta_i}|^2 \right)^{-1}
		      \end{equation}
		      is a bounded function on $\mathbb{C}^{*n}$.
		\item $\alpha\in \Delta$.
	\end{enumerate}
\end{lemma}
\begin{proof}
	(2) implies (1): Write $\alpha=\sum_i t_i\beta_i$, where $t_i\in [0,1]$, $\sum_i t_i=1$. Then
	\[
		|z^{\alpha}|^2\left(\sum_{i=1}^m |z^{\beta_i}|^2 \right)^{-1}=\prod_i |z^{\beta_i}|^{2t_i}\left(\sum_{i=1}^m |z^{\beta_i}|^2 \right)^{-1}\leq \prod_i \sum_j|z^{\beta_j}|^{2t_i}\left(\sum_{i=1}^m |z^{\beta_i}|^2 \right)^{-1}\leq 1.
	\]

	(1) implies (2): Assume that $\alpha\not\in \Delta$. Let $H$ be a hyperplane that separates $\alpha$ and $\Delta$. Say $H$ is defined by $a_1x_1+\cdots+ a_nx_n=C$. Set
	\[
		z(t)\coloneqq (t^{a_1},\ldots,t^{a_n}).
	\]
	Then clearly \eqref{eq:zalpha} evaluated at $z(t)$ is not bounded.
\end{proof}
\begin{lemma}\label{lma:polybdd2}
	Let $\beta_1,\ldots,\beta_m\in \mathbb{N}^n$ and $\beta\in \mathbb{R}^n$. Then the following are equivalent
	\begin{enumerate}
		\item $\log \sum_{i=1}^m \mathrm{e}^{x\cdot\beta_i}-(x,\beta)$ is bounded from below.
		\item $\beta$ is in the convex hull of the $\beta_i$'s.
	\end{enumerate}
\end{lemma}
\begin{proof}
	The proof follows the same pattern as \cref{lma:polybdd}.
\end{proof}

\subsection{Toric Okounkov bodies}\label{subsec:torOko}
Let $X$ be an $n$-dimensional smooth projective toric variety, corresponding to a smooth complete fan $\Sigma$ in $N_{\mathbb{R}}\cong \mathbb{R}^n$.
Let $N$ be the lattice in $N_{\mathbb{R}}$, whose dual is the character lattice $M$.
Let $T\coloneqq N\otimes_{\mathbb{Z}} \mathbb{C}^*$ be the corresponding torus.
Define $M_{\mathbb{R}}=N_{\mathbb{R}}^{\vee}$. Given any $T$-invariant divisor $D$ on $X$, let $P_D\subseteq M_{\mathbb{R}}$ be the polyhedron associated with $D$.

Let $D_1,\ldots,D_s$ be the class of prime $T$-invariant divisors on $X$, each corresponding to a ray $\rho_i$ in $\Sigma$. Let $v_i$ be the primitive generator of $\rho_i$.
Any $T$-invariant admissible flag $Y_{\bullet}$ has the following form after renumbering the $D_i$'s:
\[
	Y_i=D_1\cap\cdots\cap D_i.
\]
Now the $v_i$'s induce an isomorphism $\Phi:M\rightarrow \mathbb{Z}^n$, $u\mapsto ((u,v_i))_i$.
Let $\Phi_{\mathbb{R}}:M_{\mathbb{R}}\rightarrow \mathbb{R}^n$ be the extension of $\phi$ to $M_{\mathbb{R}}$ and $\sigma$ be the cone generated by the $v_i$'s. Let $U_{\sigma}$ be the corresponding orbit of $T$. Given any $T$-invariant line bundle, there is a unique $T$-invariant divisor $D$ with $D|_{U_{\sigma}}=0$ such that $\mathcal{O}_X(D)=L$.

It is shown in \cite[Proposition~6.1]{LM09} that
\begin{equation}
	\Gamma_k(L)=\Phi_{\mathbb{R}} \left((kP_{D})\cap M\right)
\end{equation}
for sufficiently divisible $k$. We will omit $\Phi_{\mathbb{R}}$ from out notations from now on.

Let $T_c$ be the compact torus in $T$.
Next consider a $T_c$-invariant metric $\phi$ on $L$.  
An unpublished result of Yi Yao says that in the toric setting, two invariant potentials $\phi'$, $\phi''$ are $\mathcal{I}$-equivalent if and only if $\overline{\nabla\phi'_{\mathbb{R}}(\mathbb{R}^n)}=\overline{\nabla\phi''_{\mathbb{R}}(\mathbb{R}^n)}$. In other words, in the toric setting, for the invariant potentials, the $P[\bullet]$-envelope is the same as the $P[\bullet]_{\mathcal{I}}$-envelope.
In particular, 
\[
	\vol(L,\phi)=\frac{1}{n!}\int_X (\ddc\phi)^n
\]
always holds, without having to take the $P[\bullet]_{\mathcal{I}}$-envelope. For the proof of a more general result, we refer to \cite[Theorem~3.13, Proposition~3.11]{BBGHdJ21}.

Let $U_0$ be the maximal orbit of $T$. The basis $(v_i)$ allows us to identify $U_0=\mathbb{C}^{*n}$. We denote the coordinates on $\mathbb{C}^{*n}$ by $(z_1,\ldots,z_n)$, $z_i=x_i+\mathrm{i}y_i$.
Fix a $T$-invariant section $s_0$ of $L$ on $U_0$ corresponding to $D$. Then we can identify $\phi$ with a $T_c$-invariant function on $U_0$. Given the identification $U_0=\mathbb{C}^{*n}$, $\phi$ can be identified with a convex function $\phi_{\mathbb{R}}:\mathbb{R}^n\rightarrow \mathbb{R}$ such that $\nabla \phi_{\mathbb{R}} \subseteq P_D$.
We let $P_{D,\phi}$ be the closure of the image of $\nabla\phi$. By \cite[Lemma~2.5]{BB13}, $P_{D,\phi}$ corresponds to the closure of
\[
	Q_{D,\phi}\coloneqq \left\{y\in M_{\mathbb{R}}: \phi(x)-(x,y)\text{ is bounded from below}\right\}.
\]

We will be more explicit at this point. Assume that
\[
	\phi= \log \sum_{i=1}^a |s_i|^2+\mathcal{O}(1),
\]
where $s_i\in \mathrm{H}^0(X,L)$. Let $\beta_i$ be the lattice points in $P_D$ corresponding to $s_i$.
In this case, $Q_{D,\phi}$ is just the convex polytope generated by the $\beta_i$'s by \cref{lma:polybdd2}.

Consider $\alpha\in M\cap P_D$. It corresponds to a Laurent polynomial $z^{\alpha}$ on $\mathbb{C}^{*n}$.
Observe that $\alpha\in Q_{D,\phi}$ if and only if $|z^{\alpha}|^2\mathrm{e}^{-\phi}$ is bounded from above.
This is just a reformulation of \cref{lma:polybdd}.

Thus, we find
\begin{equation}
	\Gamma_k (W^0(L,\phi))=(kQ_{D,\phi})\cap M
\end{equation}
when $k$ is sufficiently divisible.
Hence, $\Delta(L,\phi)\supseteq P_{D,\phi}$.
Comparing the volumes, we find that equality holds.

Next we deal with $T_c$-invariant $\phi$ such that $\ddc\phi$ is a Kähler current. 
Let $\phi^j$ be an equivariant quasi-equisingular approximation of $\phi$ constructed as in \cite[Corollary~13.23]{Dem12}.
Then by definition,
\[
	\Delta(L,\phi)=\bigcap_j \Delta(L,\phi^j).
\]
On the other hand,
\[
	P_{D,\phi}\subseteq \bigcap_j P_{D,\phi^j}.
\]
Hence, $P_{D,\phi}\subseteq \Delta(L,\phi)$.
On the other hand, the volume of both sides agree, so they are indeed equal thanks to the assumption that $\phi$ has analytic singularities.

In general, if $\phi$ is $T_c$-invariant and has positive volume. Let $\psi\leq \phi$ be a potential with $\ddc\psi$ being a Kähler current. We may guarantee that $\psi$ is $T_c$-invariant. Then by definition, if we set $\phi_{\epsilon}=(1-\epsilon)\phi+\epsilon\psi$, then
\[
	\Delta(L,\phi)=\overline{\bigcup_{\epsilon\in (0,1)}\Delta(L,\phi_{\epsilon})}.
\]
While
\[
	P_{D,\phi}\supseteq \overline{\bigcap_{\epsilon} P_{D,P[\phi_{\epsilon}]_{\mathcal{I}}}}.
\]
Thus, $\Delta(L,\phi)\supseteq P_{D,\phi}$.
Comparing the volumes, we find that these convex bodies are equal.

\begin{theorem}\label{thm:torDeltaP}
	Let $\phi$ be a $T_c$-invariant psh metric on $L$ with positive volume. Then
	\[
		\Delta(L,\phi)=P_{D,\phi}
	\]
	under the identification $\Phi_{\mathbb{R}}$ as above.
\end{theorem}

\subsection{Mixed volumes of line bundles}
Let $X,T$ be as in \cref{subsec:torOko}.

\begin{lemma}\label{lma:mixedvollinebdl}
	Let $L_1,\ldots,L_n$ be big and nef $T$-invariant line bundles on $X$. Assume that the flag is $T$-invariant.
	Then
	\begin{equation}\label{eq:mixedvol}
		\frac{1}{n!}(L_1,\ldots,L_n)=\vol(\Delta(L_1),\ldots,\Delta(L_n)).
	\end{equation}
\end{lemma}
Here $\vol$ denotes the mixed volume functional. We refer to \cite[Section~5.1]{Sch14} for the precise definition.

As pointed out by Rémi Reboulet, this result is already proved in \cite[Proposition~3.4.3]{BGPS14}.

\begin{proof}
	\textbf{Step 1}. We first assume that all $L_i$'s are ample.

	In this case, we know that for any $t_i\in \mathbb{N}$ ($i=1,\ldots,n$),
	\[
		\Delta\left(\sum_{i=1}^n t_i L_i\right)=\sum_{i=1}^nt_i\Delta(L_i)
	\]
	by \cite[Theorem~3.1]{Kav11}. Hence,
	\[
		\vol  \Delta\left(\sum_{i=1}^n t_i L_i\right)=\sum_{\alpha\in \mathbb{N}^n,|\alpha|=n}\binom{n}{\alpha} t^{\alpha}\vol\left(\Delta(L_1)^{\alpha_1},\ldots,\Delta(L_n)^{\alpha_n}\right).
	\]
	On the other hand, by \eqref{eq:volD1},
	\[
		\vol  \Delta\left(\sum_{i=1}^n t_i L_i\right)= \frac{1}{n!} \sum_{\alpha\in \mathbb{N}^n,|\alpha|=n} \binom{n}{\alpha} t^{\alpha}\left(L_1^{\alpha_1},\ldots,L_n^{\alpha_n}\right).
	\]
	Comparing the coefficients, we find \eqref{eq:mixedvol}.

	\textbf{Step 2}. General case.

	The results of Step 1 generalize immediately to ample $\mathbb{Q}$-divisors. Hence, the nef case follows from a simple perturbation argument.
\end{proof}
The following example is due to Chen Jiang.
\begin{example}\label{ex:mixvol}
	If the flag is not toric invariant, \cref{lma:mixedvollinebdl} fails. For example, consider $X=\mathbb{P}^1\times \mathbb{P}^1$, $L_1=\mathcal{O}(1,2)$ and $L_2=\mathcal{O}(2,1)$. Take a flag $X=Y_0\supseteq Y_1\supseteq Y_2$ with $Y_1$ being the diagonal. In this case, \eqref{eq:mixedvol} fails.

	In this case, $(L_1,L_2)=5$. By a simple computation using \cite[Theorem~6.4]{LM09}, we find
	$\Delta(L_1)=\Delta(L_2)$ is the trapezoid shown in \cref{cap:Oko}. In particular,
	\[
		\vol(\Delta(L_1),\Delta(L_2))=2<\frac{5}{2!}.
	\]
	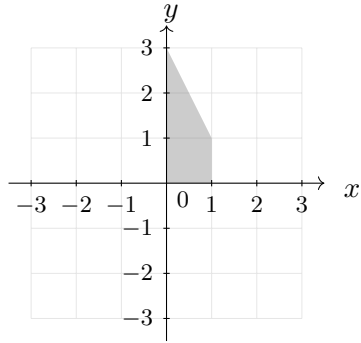
\begin{figure}[ht]\label{fig:Oko}
		\centering
		\begin{tikzpicture}[scale=0.6,domain=-3:5]

			\begin{scope}
				\path [clip] (-3,-5)--(5,3)--(-3,3)--cycle;
				\path[fill=black,opacity=0.2] (0,0) -- (1,0) -- (1,1) -- (0,3) -- cycle;
			\end{scope}

			\draw[very thin, color=gray,opacity=0.2] (-3.0,-3.0) grid (3.0,3.0);
			\draw[->,color=black] (-3.5,0) -- (3.5,0);
			\foreach \x in {-3,-2,-1,1,2,3}
			\draw[shift={(\x,0)},color=black] (0pt,2pt) -- (0pt,-2pt) node[below] {\footnotesize $\x$};
			\draw[->,color=black] (0,-3.5) -- (0,3.5);
			\draw (3.7,0.2) node[anchor=north west] {$x$};
			\draw (-0.3,4.2) node[anchor=north west] {$y$};
			\foreach \y in {-3,-2,-1,1,2,3}
			\draw[shift={(0,\y)},color=black] (2pt,0pt) -- (-2pt,0pt) node[left] {\footnotesize $\y$};
			\draw[color=black] (0pt,-10pt) node[right] {\footnotesize $0$};

		\end{tikzpicture}
		\caption{Okounkov body}\label{cap:Oko}
	\end{figure}
\end{example}

For simplicity, we call $(L,\phi)$ a \emph{$T$-invariant Hermitian big line bundle} on $X$ if $(T,\phi)$ is a Hermitian big line bundle on $X$, $L$ is $T$-invariant and $\phi$ is $T_c$-invariant.
\begin{corollary}\label{cor:mixvolnpp}
	Let $(L_i,\phi_i)$ ($i=1,\ldots,n$) be $T$-invariant Hermitian big line bundles on $X$ with positive volumes. 
	If the $T$-invariant flag satisfies that $Y_n$ is not contained in any of the polar loci of the $\phi_i$'s, then
	\begin{equation}\label{eq:mixeqmix}
		\frac{1}{n!}\int_X \ddc\phi_1\wedge\cdots\wedge\ddc \phi_n=\vol(\Delta(L_1,\phi_1),\ldots,\Delta(L_n,\phi_n)).
	\end{equation}
\end{corollary}

\begin{proof}
	According to \cref{prop:Deltapert}, by perturbing $L_i$, we may assume that each $\ddc\phi_i$ is a K\"ahler current.

	Observe that both sides of \eqref{eq:mixeqmix} are continuous under $d_S$-approximations of $\phi_i$: the left-hand side follows from \cref{thm:dsmixedmass} and the right-hand side follows from \cref{thm:Okoucont}.

	Hence, by \cite[Lemma~3.7]{DX21}, we may assume that each $\phi_i$ has analytic singularities. Taking a birational resolution, we may assume that $\phi_i$ has analytic singularities along normal crossing $\mathbb{Q}$-divisor $E_i$. By \cref{rmk:DeltaanaW0}, we reduce to the situation of \cref{lma:mixedvollinebdl}.
\end{proof}

We have finished the proof of \cref{thm:tormixedmass}.

\begin{corollary}
	Let $L_1,\ldots,L_n$ be big $T$-invariant line bundles on $X$. Assume that the flag $(Y_{\bullet})$ is $T$-invariant and $Y_n$ is not contained in the non-K\"ahler locus of any $c_1(L_i)$.
	Then
	\begin{equation}
		\frac{1}{n!}\langle L_1,\ldots,L_n\rangle=\vol(\Delta(L_1),\ldots,\Delta(L_n)).
	\end{equation}
\end{corollary}
Here $\langle \bullet\rangle$ denotes the movable intersection in the sense of \cite{BDPP13, BFJ09}.
\begin{proof}
	It suffices to apply \cref{cor:mixvolnpp} to the case where $\phi_i$ has minimal singularities.
\end{proof}

Finally, we propose the following conjecture concerning the mixed volume of partial Okounkov bodies in the non-toric setting:
\begin{conjecture}
    Let $(L_i,\phi_i)$ ($i=1,\ldots,n$) be Hermitian big line bundles on $X$ (not necessarily a toric variety) with positive volumes. Then
	\begin{equation}
		\frac{1}{n!}\int_X \ddc\phi_1\wedge\cdots\wedge\ddc \phi_n=\sup_{\nu}\vol(\Delta_{\nu}(L_1,\phi_1),\ldots,\Delta_{\nu}(L_n,\phi_n)),
	\end{equation}
    where $\nu$ runs over all rank $n$ valuations $\mathbb{C}(X)^{\times}\rightarrow \mathbb{Z}^n$.
\end{conjecture}
To the best of the author's knowledge, this conjecture is open even when the $\phi_i$'s have minimal singularities.

\clearpage


\printbibliography

\end{document}